\documentclass[11pt,reqno]{amsart}

\usepackage{lmodern} 
\usepackage[T1]{fontenc}
\usepackage[british]{babel}
\usepackage{microtype}
\microtypesetup{protrusion=true, expansion=true}

\usepackage{enumitem}
    \setlist[enumerate]{label=(\roman*),leftmargin=*}

\usepackage{mathtools}
    \theoremstyle{plain}
        \newtheorem{thm}{Theorem}[section]
        \newtheorem{lem}[thm]{Lemma}
        \newtheorem{prp}[thm]{Proposition}
        \newtheorem{cor}[thm]{Corollary}
        \newtheorem{qst}[thm]{Question}

    \theoremstyle{definition}
        \newtheorem{dft}[thm]{Definition}
        \newtheorem{exm}[thm]{Example}
        \newtheorem{rmk}[thm]{Remark}
    
    \newcommand{\dw}{\mathbf{w}_2}
    \newcommand{\T}{\mathbb{T}}
    \newcommand{\Z}{\mathbb{Z}}
    \newcommand{\Q}{\mathbb{Q}}
    \newcommand{\R}{\mathbb{R}}
    \newcommand{\C}{\mathbb{C}}
    \newcommand{\N}{\mathbb{N}}
    \newcommand{\Leb}{\mathfrak{m}}
    \newcommand{\D}{\mathfrak{D}}
    \newcommand{\lto}{\longrightarrow}
    \newcommand{\Prob}{\mathcal{P}}
    
    \newcommand{\Dist}{\mathcal{D}'}                  
    \newcommand{\Koop}{\mathcal{K}}                   
    \newcommand{\Transf}{\mathcal{L}}   
    \newcommand{\Hm}{\dot H^{-1}}
    \newcommand{\PF}{\mathcal{P}}
    
    \DeclareMathOperator{\tr}{tr}    
    \DeclareMathOperator{\Id}{Id}
    \DeclareMathOperator{\Lip}{Lip}
    \DeclareMathOperator{\spec}{spec}
    \DeclareMathOperator{\adj}{adj}
    \DeclareMathOperator*{\esssup}{ess\,sup}
    \DeclareMathOperator{\vol}{vol}
    \DeclareMathOperator{\Hess}{Hess}
    \DeclareMathOperator{\Div}{div}

\usepackage[
  unicode=true,
  bookmarks=true,
  bookmarksnumbered=false,
  bookmarksopen=false,
  breaklinks=true,
  colorlinks=true,
  pdfborder={0 0 0}
]{hyperref}

\begin{document}	
%%%%%%%% PRE-TEXTUAL
\title[Pushforward dynamics on Wasserstein spaces]
{Pushforward dynamics on Wasserstein spaces and measure rigidity} 
	
\author[D. Finamore]{Douglas Finamore}
    \address{Douglas Finamore \\
        Centro de Matemática, Computação e Cognição - UFABC\\
    	Alameda da Universidade | 09606-405, São Bernardo do Campo, SP, Brazil.}
    \email{douglas.finamore@pm.me}		
	
\author[A. Gomes]{Andr\'e Magalh\~aes de S\'a Gomes}
	\address{Andr\'e Magalh\~aes de S\'a Gomes\\
		Instituto de Matemática e Estatística - USP\\
    	Rua do Matão, 1010 | 05508-090, São Paulo, SP, Brazil.}
	\email{andremsg@usp.br}   
		
\author[C. Rodrigues]{Christian S. Rodrigues}
	\address{Christian S.~Rodrigues\\
		Instituto de Matemática, Estatística e Computação Científica - UNICAMP\\
    	Rua Sérgio Buarque de Holanda, 651 | 13083-859, Campinas, SP, Brazil,
		and Max-Planck-Institute for Mathematics in the Sciences\\
		Inselstr. 22\\
		04103 Leipzig\\
		Germany
	}
	\email{rodrigues@ime.unicamp.br}	
\date{\today}	
	
\begin{abstract}
    We study the pushforward action \( \phi_\ast \) on the Wasserstein space \( \Prob(M) \) of a closed Riemannian manifold, at a measure \( \mu_0 \) preserved by \( \phi \).
    We show that if \( \phi \) is a \( C^2 \) covering map and \( \mu_0 \) has positive \( C^1 \) density, then \( \phi_\ast \) is G\^ateaux differentiable at \( \mu_0 \) along exponential deformations, with derivative the transfer operator of \( \phi \) acting on vector fields followed by the orthogonal projection onto the tangent space.
    The derivative is the adjoint of the Koopman operator restricted to the tangent space, and its fixed space consists of the directions in which \( \mu_0 \) can be deformed without losing invariance to first order.
    Then, for appropriate pairs of endomorphisms of \( \T^d \), we compute the intersection of their fixed spaces, show it contains an infinite family of linearly independent continuous vector fields, and produce for every \( n \) an embedded \( n \)-dimensional family of measures nearly invariant under both endomorphisms.
    First-order rigidity therefore fails in every dimension.
    In particular, rigidity phenomena such as higher-dimensional analogues of the Furstenberg conjecture are, if true, genuinely non-linear.
\end{abstract}	
	

\raggedbottom
\maketitle	
\tableofcontents

\pagebreak

%%%%%%%% MAIN BODY
\section{Introduction}
    \label{sec:introduction}
    A continuous map \( \phi \) of a compact metric space \( X \) induces a natural map on the space \( \Prob(X) \) of Borel probability measures, the pushforward \( \phi_\ast\mu\coloneqq\mu\circ\phi^{-1} \), which is again continuous for the weak\( ^\ast \) topology.
In this manner a dynamical system, that is, a semigroup acting on \( X \), induces a dynamical system on \( \Prob(X) \).
These two systems are generally dynamically distinct.
The fixed points of \( \phi_\ast \) are the invariant measures of \( \phi \), an object of which the orbit structure of \( \phi \) carries no direct trace, and \( \Prob(X) \) is infinite dimensional, so much so that \( \phi_\ast \) has infinite topological entropy as soon as \( \phi \) has positive entropy \cite{bauer_topological_1975}, while zero entropy is inherited \cite{glasner_quasi-factors_1995}.

The topological dynamics of the pushforward, set out in \cite{bauer_topological_1975,glasner_quasi-factors_1995}, has received renewed attention in the last decade, both for the properties it inherits from \( \phi \), such as transitivity, mixing and recurrence \cite{baraviera_dynamics_2013,li_recurrence_2015,bernardes_dynamics_2016}, and for the invariants appropriate to systems of infinite entropy, such as uniform positive entropy and mean dimension \cite{bernardes_uniformly_2022,liu_relative_2023,lopes_level-2_2024,vermersch_measure-theoretic_2024,burguet_topological_2025}.

When the space being acted upon is a closed Riemannian manifold \( M \), further structure becomes available.
Besides metrising the weak\( ^\ast \) topology of \( \Prob(M) \), the \( 2 \)-Wasserstein distance \( \dw \) supports a weak differentiable structure, introduced formally by Otto \cite{otto_geometry_2001}, and developed rigorously in \cite{ambrosio_gradient_2008,lott_geometric_2008,gigli_second_2012,gomes_differential_2024}.
The tangent directions at \( \mu\in\Prob(M) \) are the vector fields of
\[
    L^2_0(\mu)=\overline{\{\nabla\psi:\ \psi\in C^\infty(M)\}}^{\,L^2(\mu;TM)},
\]
each of which deforms \( \mu \) along the curve \( t\mapsto(\exp tv)_\ast\mu \), that is, the pushforward of \( \mu \) under \( x\mapsto\exp_x(tv(x)) \).
One may then ask what can be said about the differential aspects of the induced dynamics. 
We ask this with a rigidity question in view: Furstenberg's conjecture \cite{furstenberg_disjointness_1967} and its higher-rank relatives have a first-order form that becomes visible only once the derivative of \( \phi_\ast \) is known.
Theorem~\ref{thm:main.application} shows, however, that this form fails on \( \T^d \) in every dimension.

For \( \phi \)-invariant measures, that is, for the fixed points of \( \phi_\ast \), there is a natural guess for the derivative.
If \( \mu_0 \) is such a fixed point, then the measures absolutely continuous with respect to \( \mu_0 \) are acted upon linearly, through their densities, by the Perron--Frobenius transfer operator \( \PF_{\phi} \) of \( \phi \), and a first-order deformation of \( \mu_0 \) is precisely such a density.
One could naturally expect the derivative of \( \phi_\ast \) at \( \mu_0 \), that is, its \emph{best linear approximation}, to be a transfer operator in disguise.
Kloeckner confirmed this intuition by computing the derivative for expanding maps of the circle \cite[Theorem 5.1]{kloeckner_optimal_2013}, and the answer is a transfer operator composed with a correction which he calls the centring operator.

On general \(M\) the tangent directions in \(\Prob(M)\) are vector fields rather than densities.
Thus, the operator acting on them is not \( \PF_\phi \) but its counterpart on sections of \( TM \): the adjoint \( \Transf_{\mu_0}\coloneqq(\Koop_{\mu_0})^\ast \) in \( L^2(\mu_0;TM) \) of the Koopman operator \( \Koop_{\mu_0}u=(D\phi)^\ast(u\circ\phi) \).
The issue is that \( \Transf_{\mu_0} \) does not carry \( L^2_0(\mu_0) \) into itself, so an orthogonal projection \( P_{\mu_0} \) onto the tangent space must be inserted, analogue to Kloeckner's centring operator.
On the circle the space it discards is the line spanned by \( \rho^{-1} \) and the projection is the subtraction of a constant, whereas in higher dimension that space is generally infinite dimensional.
This same projection already appears, in different forms, in two available approaches.
Lessel and Schick \cite{lessel_differentiable_2020} apply a projector because the fibrewise average need not be tangent, and Gomes, Rodrigues and San~Martin \cite[\S3]{gomes_differential_2024} discard the same directions on the ground that a solenoidal \( z \) satisfies \( \frac{d}{dt}\big|_{0}(\exp tz)_\ast\mu=0 \) in the sense of distributions.
Here we justify both conventions by turning them into a metric fact.
In Theorem~\ref{thm:solenoidal} we show that, for \( \mu=\rho\,\vol \) with \( \rho\in C^1(M) \) positive, for every \( v \) and every solenoidal \( z \) in \( L^2(\mu;TM) \),
\[
    \dw\big((\exp t(v+z))_\ast\mu,\ (\exp tv)_\ast\mu\big)=o(t),\qquad t\searrow0 .
\]
A divergence-free field stirs the mass of \( \mu \) inside itself rather than displacing it, so that it cannot possibly influence \( \dw \) at first order.

The projection \(P_\mu\) appears three times, in three different guises.
It appears \emph{metrically} in Lemma~\ref{lem:duality.lower}, where the \(c\)-transform of the quadratic cost completes a square and what survives is the norm of the projected field. 
It appears \emph{geometrically} in Theorem~\ref{thm:solenoidal}, where the discarded component generates a measure-preserving flow whose first-order deformation is invisible to \(\dw\). 
It appears \emph{dynamically} in Theorem~\ref{thm:derivative.acim}, where \(\Transf_{\mu_0}\) fails to preserve the tangent space and the projection is what remains once the solenoidal part has been shown to cost nothing. 
The projection is therefore forced into the theory, rather than chosen as a convenience.

On this rests the first-order calculus of Section~\ref{sec:derivatives}, whose main statement, Theorem~\ref{thm:derivative.acim}, is that for a \( C^2 \) covering map \( \phi \) preserving \( \mu_0=\rho\,\vol \), with \( \rho\in C^1(M) \) positive, and for every \( v\in L^2_0(\mu_0) \),
\[
    \dw\Big(\phi_\ast\big((\exp tv)_\ast\mu_0\big),\ \big(\exp t\,\D_\phi v\big)_\ast\mu_0\Big)=o(t),
    \qquad \D_\phi\coloneqq P_{\mu_0}\Transf_{\mu_0},
\]
with a remainder uniform on bounded subsets of finite-dimensional subspaces.
Our statement is metric and pointwise: it concerns a single measure and a single tangent direction, without passing through absolutely continuous curves. The notion of derivative from \cite{lessel_differentiable_2020}, on the other hand, is kinematic, being formulated along curves, and admits no pointwise form, as its authors observe.
Moreover, our statement covers every \( C^2 \) expanding map of every closed manifold, such maps being covering maps \cite{shub_endomorphisms_1969}, whereas the arguments of \cite{kloeckner_optimal_2013}, resting on monotone rearrangements, apply only to the circle.

Proposition~\ref{prp:derivative.functoriality} identifies \( \D_\phi \) with the adjoint of the Koopman operator restricted to \( L^2_0(\mu_0) \), so that it is indeed the transfer operator of \( \phi \) acting on vector fields, and identifies its fixed space \( E_\phi \) with the set of directions in which \( \mu_0 \) can be deformed without losing invariance at first order.
Whether \( E_\phi\neq\{0\} \) in general is therefore a spectral question about a transfer operator acting on sections of \( TM \) rather than on functions, a setting for which the modern theory of transfer operators \cite{baladi_positive_2000} has no ready counterpart, and which we leave open as Question~\ref{qst:general.eigen}.

Once the general theory is in place we turn to a model in which that question is more tractable.
Since \( \T^d \) is a compact Lie group, \( \Prob(\T^d) \) is parallelisable in the sense of \cite[\S4]{gomes_differential_2024}: the Fourier basis provides a global frame for the tangent directions.
Taking advantage of this, we build on \cite{kloeckner_optimal_2013} and exhibit the derivative \( \D_A \) explicitly for the expanding endomorphism \( \varphi_A \) of \( \T^d \) determined by a matrix \( A\in\mathrm M_d(\Z) \), in coordinates obtained by renormalising the Fourier frame so as to make it dynamically equivariant.
In those coordinates a fixed vector of \( \D_A \) is a sequence constant along the orbits under the transpose \( A^\top \) on the frequency lattice, and the spectral theory of \( \D_A \) is replaced by arithmetic.
The description of those orbits is Lemma~\ref{lem:freeness}, a unique-factorisation statement for pairs of commuting expanding integer matrices with coprime determinants, which plays on \( \T^d \) the role played on the circle, in \cite[\S7]{kloeckner_optimal_2013_preprint}, by the factorisation of an integer into powers of the two multipliers and a factor prime to both.
On the torus, the flat geometry also makes transport plans explicit, so that several of the general results admit, on the torus, a second, geometric proof, whereas the general analytic arguments exhibit no transport plans at all.

The rigidity question is the following: Furstenberg's conjecture \cite{furstenberg_disjointness_1967} asserts that an atomless probability measure on \( S^1 \) invariant under \( x\mapsto2x \) and \( x\mapsto3x \) is the Lebesgue measure \(\Leb\).
It is currently open, with positive answers in particular cases such as the positive entropy regime \cite{rudolph_x2x3_1990}, whose multidimensional counterpart is Host's theorem \cite[Theorem 3]{host_uniform_2000}.
Rigidity for higher-rank actions is the subject of a large literature, of which \cite{berend_multi_1983,katok_invariant_1996,einsiedler_invariant_2006} is a small sample.
Following \cite[\S1.4]{kloeckner_optimal_2013_preprint}, one weakens such a statement until it becomes an assertion about tangent spaces.
The sets \( I_A \) and \( I_B \) of measures invariant under \( \varphi_A \) and under \( \varphi_B \) are the fixed-point sets of \( (\varphi_A)_\ast \) and \( (\varphi_B)_\ast \).
Were they submanifolds, no \( C^1 \) curve could leave \( \Leb \) inside \( I_A\cap I_B \) as soon as \( E_A\cap E_B=\{0\} \).
This last condition, unlike the submanifold structure, is perfectly well posed in our framework.
Theorem~\ref{thm:main.application} says that it fails in every dimension: the space \( E_A\cap E_B \) is infinite dimensional, and one can deform \(\Leb\) tangentially to arbitrary order while preserving invariance only at first order.
This is no counterexample to rigidity, but rather an obstruction to a proof strategy, and it says that whatever mechanism makes Host's theorem true is invisible in the tangent space at \( \Leb \).

\paragraph*{\textbf{Organisation of the paper}.}

Section~\ref{sec:setting} fixes the setting: it recalls the distributional model of the tangent space, and the identification of \( T_\mu\Prob(M) \) with \( L^2_0(\mu) \).

Section~\ref{sec:derivatives} is the general theory.
Its first part requires no regularity of \( \mu \) and establishes the transfer and Koopman operators together with the two one-sided estimates for \( \dw \).
Its second part replaces transport plans by densities, and records our main results for measures that are absolutely continuous with \( C^1 \) density with respect to the Riemannian volume.

Section~\ref{sec:torus} specialises everything to \( M=\T^d \) and \( \phi=\varphi_A \), computes the derivative at the Lebesgue measure \( \Leb \) and proves Theorem~\ref{thm:main.application}.
We conclude with Section~\ref{sec:conclusion}, which collects the questions left open.
    
\section{Tangent directions on Wasserstein spaces}
    \label{sec:setting}
    Throughout the paper \( M \) denotes a \emph{closed} (compact, boundaryless), connected, oriented Riemannian manifold.
We will denote by \(\langle \cdot, \cdot \rangle\) its metric tensor, by \(d\) its distance function, by \( \vol \) its normalised volume measure, and by \( \Prob(M) \) the space of Borel probability measures on \( M \), endowed with the weak\( ^\ast \) topology.
The \( 2 \)-Wasserstein metric is defined on \( \Prob(M) \) via
\[
    \dw(\mu,\nu) \coloneqq \left(\inf_{\pi\in\Pi(\mu,\nu)}\int_{M^2}d(x_1,x_2)^2 d\pi(x_1,x_2)\right)^{1/2};
\]
with \( \Pi(\mu,\nu) \) being the set of all transport plans (couplings) between \( \mu \) and \( \nu \).
A classical fact in Optimal Transport Theory is that this infimum is achieved by a transport plan that is called \textit{optimal transport} between \( \mu \) and \( \nu \).
For general \( M \) the function \( \dw \) may fail to be finite on \( \Prob(M) \), as the infimum might explode.
The \( 2 \)-Wasserstein space, denoted by \( \Prob_2(M) \), is defined to be the set of probability measures with finite moments of order \( 2 \) -- or, equivalently, with finite \( 2 \)-Wasserstein distance to any (and therefore, to all) Dirac measure.
Since our \( M \) is compact, every \( \mu\in\Prob(M) \) has finite second moment, so \( \Prob(M)=\Prob_2(M) \) and, by \cite[Theorem~6.9]{villani_optimal_2009}, \( \dw \) metrises the weak\( ^\ast \) topology.
We therefore write \( \Prob(M) \) throughout.
Likewise \( C^\infty_c(M)=C^\infty(M) \) and every diffeomorphism of \( M \) is compactly supported, so we drop the subscript \( c \) from \( C^\infty_c \) and from \( \mathfrak X_c \).

\subsection{Distributions and the action of the diffeomorphism group}
\label{sec:distributional}

Before introducing tangent spaces we set up the ambient linear space in which they will live, following what was done in \cite{gangbo_differential_2011} for Wasserstein spaces of Euclidean spaces.
The Riemannian formalism on \( \Prob(M) \) originates with \cite{otto_geometry_2001} and was made rigorous in \cite{lott_ricci_2009}; the description of the tangent space as a closure of gradients, and its identification with a space of distributions, are as in \cite[\S2--3]{gomes_differential_2024}, whose notation we follow where it does not conflict with ours.
We recall it here to fix notation and to isolate what is needed later on.

Let \( \Dist(M) \) denote the space of distributions on \( M \), that is, the topological dual of \( C^\infty(M) \), and let \( \mathcal M(M)\subset\Dist(M) \) be the subspace of signed Radon measures.
We write \( \mu(g) \) for the duality pairing of \( \mu\in\Dist(M) \) with a test function \( g\in C^\infty(M) \).
Under the inclusion \( \Prob(M)\subset\mathcal M(M)\subset\Dist(M) \), the weak\( ^\ast \) topology of \( \Prob(M) \) is the one induced by \( \Dist(M) \).

The group \( \operatorname{Diff}(M) \) of diffeomorphisms of \( M \) acts continuously on \( \Dist(M) \) by pushforward.
The subset \( \Prob(M) \) is invariant under the action, and the restricted action is the usual pushforward of probabilities.

Given \( \mu\in\Prob(M) \), we write
\[
    \mathcal O_\mu=\{f_\ast\mu:\ f\in\operatorname{Diff}(M)\}
    \quad\textrm{and}\quad
    \operatorname{Diff}(M)_\mu=\{f\in\operatorname{Diff}(M):\ f_\ast\mu=\mu\}
\]
for the orbit and the isotropy subgroup of \( \mu \).
The orbits \( \mathcal O_\mu \) partition \( \Prob(M) \), and it is these orbits, rather than \( \Prob(M) \) itself, that carry a differentiable structure; this is made precise in Proposition~\ref{prp:orbit.tangent} below.

We will be particularly interested in the \emph{weighted divergence with respect to the probability measure \( \mu \)}, which is the distribution
\[
    \nabla\!\cdot(v\mu):\psi\longmapsto-\int_M\langle\nabla\psi,v\rangle\,d\mu,
\]
where \(v\in L^1(\mu;TM)\).
We will use \(\Div\) to denote the usual divergence of vector fields.
Note that if \(\mu = \rho\vol \) for \(\rho \in C^1(M)\), then
\begin{equation}\label{eq:divergence.equivariance}
    \nabla\!\cdot(v\mu)=\Div(\rho v)\,\vol.
\end{equation}

By \cite{leslie_differential_1967}, the group \( \operatorname{Diff}(M) \) is a Fréchet Lie group whose Lie algebra is the space \( \mathfrak X(M) \) of smooth vector fields on \( M \).
The nature of the isotropy subgroup \( \operatorname{Diff}(M)_\mu \) is a more delicate matter, though we shall need no structure on it.
When \( \mu=\rho\,\vol \) with \( \rho\in C^\infty(M) \) and \( \rho>0 \), it is a Lie subgroup with Lie algebra
\[
    \mathfrak X(M)_\mu=\big\{v\in\mathfrak X(M):\ \nabla\!\cdot(v\mu)=0\big\},
\]
by \cite{ebin_groups_1970,omori_infinite_2006}, and Moser's theorem \cite{moser_volume_1965} identifies its orbit \( \mathcal O_\mu \) with the set of probability measures having a smooth positive density.
For a general \( \mu\in\Prob(M) \) we make no such claim, and we write \( \mathfrak X(M)_\mu \) for the right-hand side above, which is the kernel of the map \eqref{eq:orbit.derivative} below and is all that we shall use.
Consider the orbit map
\begin{align*}
    j:\operatorname{Diff}(M)/\operatorname{Diff}(M)_\mu&\lto\mathcal O_\mu \\
    [f]&\longmapsto f_\ast\mu .
\end{align*}
It is a bijection by construction.
It is continuous, because the weak\( ^\ast \) topology on \( \mathcal O_\mu \) is coarser than the quotient Fréchet topology on the source.
For \( \mu \) with smooth density, it is smooth as a map into the topological vector space \( \mathcal M(M) \), with injective derivative
\begin{equation}\label{eq:orbit.derivative}
    Dj:\ \mathfrak X(M)/\mathfrak X(M)_\mu\lto T_\mu\mathcal O_\mu,
    \qquad Dj(v)=-\nabla\!\cdot(v\mu).
\end{equation}
It is not, however, an open map, precisely because the weak\( ^\ast \) topology is strictly coarser than the induced Fréchet one; so \( j \) is not a homeomorphism onto its image.
For general \( \mu \), one should read Equation~\eqref{eq:orbit.derivative} as the \emph{definition} of the linear operator
\[
    \iota_\mu:L^2(\mu;TM)\lto\Dist(M),\qquad \iota_\mu(v)=-\nabla\!\cdot(v\mu),
\]
which is the only use we make of it in what follows.
Its kernel is described in Section~\ref{sec:tangent}, where \( \iota_\mu \) is restricted to \( L^2_0(\mu) \).

\subsection{The Eulerian and distributional models}
\label{sec:tangent}

Take \( \mu\in\Prob(M) \).
Let \( L^2(\mu;TM) \) denote the Hilbert space of Borel vector fields whose squared norm is \( \mu \)-integrable, and set
\begin{equation}\label{eq:L20}
    L^2_0(\mu)\coloneqq\overline{\{\nabla\psi:\ \psi\in C^\infty(M)\}}^{\,L^2(\mu;TM)} ,
\end{equation}
a closed subspace of \( L^2(\mu;TM) \); we write \( P_\mu:L^2(\mu;TM)\to L^2_0(\mu) \) for the orthogonal projection.
The \emph{tangent space} of \( \Prob(M) \) at \( \mu \) is
\begin{equation}\label{eq:tangent}
    T_\mu\Prob(M)\coloneqq\big\{-\nabla\!\cdot(v\mu):\ v\in L^2_0(\mu)\big\}\ \subset\ \Dist(M).
\end{equation}

The motivation for \eqref{eq:tangent} is the following.
A curve \( (\mu_t)_{t\in I} \) in \( (\Prob(M),\dw) \) is \emph{absolutely continuous} if there is \( k\in L^1(I) \) with \( \dw(\mu_s,\mu_t)\le\int_s^tk(r)\,dr \) for \( s<t \) in \( I \); it is \( AC^2 \) if moreover \( k\in L^2(I) \).
For such a curve there is, for a.e.\ \( t\in I \), a \emph{unique} \( v_t\in L^2_0(\mu_t) \) satisfying the \emph{continuity equation}
\begin{equation}\label{eq:continuity}
    \partial_t\mu_t+\nabla\!\cdot(v_t\mu_t)=0\qquad\text{in }\Dist(I\times M),
\end{equation}
and this \( v_t \) realises the metric derivative:
\begin{equation}\label{eq:metric.derivative}
    \lVert v_t\rVert_{L^2(\mu_t)}=\lvert\dot\mu\rvert(t)
    \coloneqq\lim_{h\to0}\frac{\dw(\mu_{t+h},\mu_t)}{\lvert h\rvert}.
\end{equation}
Uniqueness fails in the larger space \( L^2(\mu_t) \): any \emph{solenoidal} (i.e., divergence-free) perturbation of \( v_t \) also solves \eqref{eq:continuity}.
We call \( v_t \) the \emph{velocity field} of \( (\mu_t) \) and write \( \lvert\dot\mu\rvert(t) \) for its metric derivative throughout.
\begin{rmk}
\label{rmk:C1.test}
    The continuity equation \eqref{eq:continuity} is stated against test functions of class \( C^\infty \), but it holds against test functions of class \( C^1 \) as well, and we shall use it in that form.
    Indeed, for \( v\in L^2(\mu;TM) \) the linear functional
    \[
        \psi\longmapsto\int_M\langle\nabla\psi,v\rangle\,d\mu
    \]
    satisfies \( \lvert\int_M\langle\nabla\psi,v\rangle\,d\mu\rvert\le\lVert\nabla\psi\rVert_\infty\lVert v\rVert_{L^2(\mu)} \), so it is continuous for the norm of \( C^1(M) \), and \( C^\infty(M) \) is dense in \( C^1(M) \) because \( M \) is compact.
    The same observation applies to any identity in \( \Dist(M) \) whose two sides are continuous for that norm, and we shall not repeat it.
    See \cite[Chapter~8]{ambrosio_gradient_2008} for the corresponding discussion in the Euclidean setting.
\end{rmk}

These facts admit the following converse (cf. {\cite[Theorem 1.28]{gigli_second_2012}}).

\begin{lem}
\label{lem:gigli.converse}
    If \( (\mu_t)_{t\in I} \) and a Borel family \( (v_t) \) satisfy the continuity equation \eqref{eq:continuity} and \( t\mapsto\lVert v_t\rVert_{L^2(\mu_t)}\in L^1(I) \), then \( (\mu_t) \) is absolutely continuous and \( \lvert\dot\mu_t\rvert\le\lVert v_t\rVert_{L^2(\mu_t)} \) for a.e.\ \( t\in I \).
\end{lem}

The tangent space \eqref{eq:tangent} is, in fact, tangent to the orbit through \( \mu \) rather than to \( \Prob(M) \) as a whole. 
It can be obtained by discarding the solenoidal components inside \( L^2(\mu;TM) \), i.e. the fields \( z \) with \( \nabla\!\cdot(z\mu)=0 \).
These are exactly the orthogonal complement of \( L^2_0(\mu) \): by \eqref{eq:L20}, a field \( z \) is orthogonal to \( L^2_0(\mu) \) if and only if \( \int_M\langle\nabla\psi,z\rangle\,d\mu=0 \) for every \( \psi\in C^\infty(M) \), which is precisely the vanishing of \( \nabla\!\cdot(z\mu) \).
For smooth \( z \) these are the elements of \( \mathfrak X(M)_\mu \), and when \( \mu=\rho\,\vol \) with \( \rho\in C^1(M) \) and \( \rho>0 \) the solenoidal fields of \( L^2(\mu;TM) \) form the closure of \( \mathfrak X(M)_\mu \), by Lemma~\ref{lem:regular.solenoidal} below.

That the discarded directions are, at least formally, invisible is recorded in \cite[\S3]{gomes_differential_2024}: if \( \nabla\!\cdot(z\mu)=0 \), then \( \frac{d}{dt}\big|_{t=0}(\exp tz)_\ast\mu=0 \) in \( \Dist(M) \), so that no first-order test against a smooth function distinguishes \( \exp t(v+z) \) from \( \exp tv \).
This says nothing about \( \dw \), and the two statements are not equivalent: a curve may have vanishing distributional derivative at \( t=0 \) and still move at unit metric speed.
We show later, in Section~\ref{sec:solenoidal}, that the discarding costs nothing metrically either, as solenoidal fields are metrically invisible at first order.

\begin{prp}\label{prp:orbit.tangent}
    Let \( \mu\in\Prob(M) \).
    Then \( -\nabla\!\cdot(v\mu)\in T_\mu\Prob(M) \) for every \( v\in\mathfrak X(M) \), the image of the derivative \eqref{eq:orbit.derivative} of the orbit map \( j \) corresponds under \eqref{eq:tangent} to \( P_\mu\left(\mathfrak X(M)\right)\subset L^2_0(\mu) \), and
    \[
        \overline{P_\mu\left(\mathfrak X(M)\right)}^{\,L^2(\mu;TM)}=L^2_0(\mu).
    \]
    The completion is taken in the Eulerian model, not in \( \Dist(M) \).
\end{prp}

\begin{proof}
    Let \( v\in L^2(\mu;TM) \) and write \( v=P_\mu v+z \) with \( z\in L^2_0(\mu)^\perp \).
    By the characterisation of \( L^2_0(\mu)^\perp \) recorded above we have \( \nabla\!\cdot(z\mu)=0 \), and therefore
    \[
        -\nabla\!\cdot(v\mu)=-\nabla\!\cdot\left((P_\mu v)\,\mu\right).
    \]
    Applying this to \( v\in\mathfrak X(M) \) and using \eqref{eq:orbit.derivative}, the image of the derivative of \( j \) is the set of \( -\nabla\!\cdot(u\mu) \) with \( u\in P_\mu(\mathfrak X(M)) \), which is contained in \( T_\mu\Prob(M) \) because \( P_\mu(\mathfrak X(M))\subset L^2_0(\mu) \).
    Conversely, every \( \psi\in C^\infty(M) \) has \( \nabla\psi\in\mathfrak X(M) \) and \( P_\mu\nabla\psi=\nabla\psi \), so \( P_\mu(\mathfrak X(M)) \) contains the gradients of smooth functions, whose closure in \( L^2(\mu;TM) \) is \( L^2_0(\mu) \) by \eqref{eq:L20}.
    The two inclusions give the displayed identity, and \eqref{eq:tangent} is its image under \( v\mapsto-\nabla\!\cdot(v\mu) \).
\end{proof}

\begin{rmk}[Two models of the tangent space]\label{rmk:two.models}
    The assignment
    \[
        \begin{split}
            \iota_\mu:\ L^2_0(\mu)&\lto\Dist(M) \\
            v &\mapsto -\nabla\!\cdot(v\mu)
        \end{split}
    \]
    is injective, its kernel in \( L^2(\mu;TM) \) being \( L^2_0(\mu)^\perp \) by the characterisation recorded above, and \eqref{eq:tangent} says exactly that \( T_\mu\Prob(M)=\iota_\mu\big(L^2_0(\mu)\big) \).
    Two models of the tangent space are therefore in play: the \emph{distributional} model \( T_\mu\Prob(M)\subset\Dist(M) \), which is the one relevant to Section~\ref{sec:distributional} and to all smoothness statements, and the \emph{Eulerian} model \( L^2_0(\mu)\subset L^2(\mu;TM) \), which carries the Hilbert structure and in which \eqref{eq:metric.derivative} is stated.
    We use \( \iota_\mu \) to identify the two and, unless the distinction matters, we speak of tangent vectors as vector fields, writing \( v\in T_\mu\Prob(M) \) for \( v\in L^2_0(\mu) \).
    The sign in \( \iota_\mu \) is fixed so that \( \iota_\mu(v_t) \) is the distributional time-derivative of a curve with velocity field \( v_t \), consistently with \eqref{eq:continuity} and \eqref{eq:orbit.derivative}.

    The object defined in \eqref{eq:tangent} is a closed subspace of \(L^2(\mu;TM)\) for every \(\mu\).
    It coincides with the tangent cone of \cite{gigli_second_2012} precisely for the regular measures of \cite{gigli_inverse_2011}, and we shall not need that identification.
\end{rmk}
    
\section{First-order theory of \texorpdfstring{\(\phi_\ast\)}{phi*}}
    \label{sec:derivatives}
    \subsection{Couplings and duality}
\label{sec:couplings}
Throughout this section \( \phi\in C^1(M,M) \) is fixed, and \( I\ni0 \) denotes an interval.
Given \( \mu\in\Prob(M) \) we write \( \nu\coloneqq\phi_\ast\mu \) and \( (\mu^y)_{y\in M} \) for the disintegration of \( \mu \) over \( \nu \) along \( \phi \), so that \( \mu=\int_M\mu^y\,d\nu(y) \) with \( \mu^y\left(\phi^{-1}(y)\right)=1 \) for \( \nu \)-almost every \( y \).
To the pair \( (\mu,\nu) \) we attach two bounded operators.

\begin{lem}
\label{lem:transfer.operator}
    Let \( \mu\in\Prob(M) \) and \( \nu=\phi_\ast\mu \).
    The \emph{Koopman operator} of \( \phi \) at \( \mu \),
    \[
        \Koop_\mu:L^2(\nu;TM)\lto L^2(\mu;TM),
        \qquad
        \left(\Koop_\mu u\right)(x)\coloneqq(D_x\phi)^{\!\ast}\,u\left(\phi(x)\right),
    \]
    is bounded, with \( \lVert\Koop_\mu\rVert_{\mathrm{op}}\le\Lip(\phi) \).
    Its adjoint \( \Transf_\mu\coloneqq(\Koop_\mu)^{\!\ast}:L^2(\mu;TM)\to L^2(\nu;TM) \), the \emph{transfer operator} of \( \phi \) at \( \mu \), is the fibrewise conditional expectation
    \begin{equation}\label{eq:transfer.disintegration}
        \left(\Transf_\mu v\right)(y)=\int_{\phi^{-1}(y)}D_x\phi\,v(x)\,d\mu^y(x),
        \qquad \nu\text{-a.e. }y \in M .
    \end{equation}
\end{lem}

\begin{proof}
    Since \( \nu=\phi_\ast\mu \) we have \( \int_M\lvert u\circ\phi\rvert^2\,d\mu=\int_M\lvert u\rvert^2\,d\nu \) for every \( u\in L^2(\nu;TM) \), and \( \lvert(D_x\phi)^{\!\ast}\xi\rvert\le\Lip(\phi)\lvert\xi\rvert \) for every \( \xi\in T_{\phi(x)}M \), the quantity \( \Lip(\phi)=\lVert D\phi\rVert_\infty \) being finite because \( M \) is compact and \( \phi \) is of class \( C^1 \).
    Combining the two gives \( \lVert\Koop_\mu u\rVert_{L^2(\mu)}\le\Lip(\phi)\lVert u\rVert_{L^2(\nu)} \), so \( \Koop_\mu \) is bounded and so is its adjoint, with the same bound.

    Write \( Tv \) for the right-hand side of \eqref{eq:transfer.disintegration}.
    For \( x\in\phi^{-1}(y) \) the vector \( D_x\phi\,v(x) \) lies in the fixed vector space \( T_yM \), so the integral is meaningful, and \( y\mapsto Tv(y) \) is \( \nu \)-measurable by the disintegration theorem.
    By Jensen's inequality a conditional expectation does not increase the \( L^2 \) norm, so
    \[
        \lVert Tv\rVert^2_{L^2(\nu)}\le\int_M\lvert D_x\phi\,v(x)\rvert^2\,d\mu(x)\le\Lip(\phi)^2\lVert v\rVert^2_{L^2(\mu)},
    \]
    and in particular \( Tv\in L^2(\nu;TM) \).
    Finally, for \( u\in L^2(\nu;TM) \),
    \[
    \begin{split}
        \big\langle v,\Koop_\mu u\big\rangle_{L^2(\mu)}
        &=\int_M\big\langle D_x\phi\,v(x),\,u(\phi(x))\big\rangle\,d\mu(x)\\
        &=\int_M\Big\langle\int_{\phi^{-1}(y)}D_x\phi\,v(x)\,d\mu^y(x),\ u(y)\Big\rangle\,d\nu(y)
        =\big\langle Tv,u\big\rangle_{L^2(\nu)} ,
    \end{split}
    \]
    and \( u \) being arbitrary, \( \Transf_\mu v=Tv \).
\end{proof}

\begin{lem}[The transfer operator intertwines the pushforward]
\label{lem:transfer.intertwines}
    Let \( \mu\in\Prob(M) \), let \( \nu=\phi_\ast\mu \) and let \( v\in L^2(\mu;TM) \).
    Then
    \begin{equation}\label{eq:transfer.commutes}
        \nabla\!\cdot\big(\Transf_\mu v\ \nu\big)=\phi_\ast\Big(\nabla\!\cdot\big(v\,\mu\big)\Big)
    \end{equation}
    in \( \Dist(M) \), that is, \( \iota_\nu(\Transf_\mu v)=\phi_\ast\iota_\mu(v) \).
\end{lem}

\begin{proof}
    Fix \( \psi\in C^\infty(M) \).
    Then \( \psi\circ\phi \) is of class \( C^1 \), so the right-hand side of \eqref{eq:transfer.commutes} is meaningful by Remark~\ref{rmk:C1.test}.
    Since \( \Koop_\mu\nabla\psi=\nabla(\psi\circ\phi) \), the adjunction of Lemma~\ref{lem:transfer.operator} gives
    \[
        \int_M\big\langle\nabla\psi,\Transf_\mu v\big\rangle\,d\nu
        =\int_M\big\langle\Koop_\mu\nabla\psi,\,v\big\rangle\,d\mu
        =\int_M\big\langle\nabla(\psi\circ\phi),v\big\rangle\,d\mu .
    \]
    The two extremes are \( -\nabla\!\cdot(\Transf_\mu v\ \nu)(\psi) \) and \( -\nabla\!\cdot(v\mu)(\psi\circ\phi)=-\phi_\ast\big(\nabla\!\cdot(v\mu)\big)(\psi) \).
\end{proof}

The transfer operator on vector fields and the pushforward of distributions are thus intertwined by the map \( \iota_\mu \) of Remark~\ref{rmk:two.models}.

\begin{prp}
    \label{prp:functoriality.pushforward}
    Let \( (\mu_t)_{t\in I} \) be an absolutely continuous curve in \( \Prob(M) \) with velocity field \( v_t\in T_{\mu_t}\Prob(M) \) and \( t\mapsto\lVert v_t\rVert_{L^2(\mu_t)}\in L^1(I) \).
    Put \( \nu_t\coloneqq\phi_\ast\mu_t \) and \( w_t\coloneqq\Transf_{\mu_t}v_t \).
    Then, for a.e.\ \( t\in I \),
    \begin{enumerate}
        \item\label{item:w.bound} \( \lVert w_t\rVert_{L^2(\nu_t)}\le\Lip(\phi)\,\lVert v_t\rVert_{L^2(\mu_t)} \);
        \item\label{item:w.cont} the pair \( (\nu_t,w_t) \) solves the continuity equation in \( \Dist(I\times M) \), and consequently \( \nu_t \) is absolutely continuous with \( \lvert\dot\nu_t\rvert\le\Lip(\phi)\,\lvert\dot\mu_t\rvert) \).
    \end{enumerate}
\end{prp}

\begin{proof}
    Item~\ref{item:w.bound} is the bound on \( \lVert\Transf_{\mu_t}\rVert_{\mathrm{op}} \) of Lemma~\ref{lem:transfer.operator}.

    For item~\ref{item:w.cont}, fix \( \psi\in C^\infty(M) \).
    Then \( \psi\circ\phi\in C^1(M) \), and by Remark~\ref{rmk:C1.test} we may test the continuity equation of \( (\mu_t,v_t) \) against it.
    Since \( \nabla(\psi\circ\phi)(x)=(D_x\phi)^{\!\ast}\nabla\psi(\phi(x))=\left(\Koop_{\mu_t}\nabla\psi\right)(x) \), we obtain, in \( \Dist(I) \),
    \[
        \frac{d}{dt}\int_M\psi\,d\nu_t
        =\frac{d}{dt}\int_M\psi\circ\phi\,d\mu_t
        =\int_M\big\langle\Koop_{\mu_t}\nabla\psi,\,v_t\big\rangle\,d\mu_t
        =\int_M\big\langle\nabla\psi,\,w_t\big\rangle\,d\nu_t ,
    \]
    the last equality being the definition of \( w_t \) as the adjoint image \( \Transf_{\mu_t}v_t \).
    This is the continuity equation for \( (\nu_t,w_t) \), and \( t\mapsto\lVert w_t\rVert_{L^2(\nu_t)} \) lies in \( L^1(I) \) by item~\ref{item:w.bound}, so Lemma~\ref{lem:gigli.converse} gives the absolute continuity of \( (\nu_t) \) and the bound on its metric derivative.
\end{proof}

\begin{rmk}
    For \( F=f_\ast \), with \( f \) smooth and proper between complete Riemannian manifolds, the operator \eqref{eq:transfer.disintegration} and the bound of Lemma~\ref{lem:transfer.operator} appear in \cite[Theorem 3.3 and Proposition 3.5]{lessel_differentiable_2020}, together with the conclusion of Proposition~\ref{prp:functoriality.pushforward}.
    The hypotheses in force here are strictly weaker, \( M \) being closed and \( \phi \) only of class \( C^1 \), and the proofs above are self-contained.
\end{rmk}

Although Proposition~\ref{prp:functoriality.pushforward} produces a solution of the continuity equation for \(\nu_t\), the transfer operators \( \Transf_{\mu_t} \) do not restrict to operators \( T_{\mu_t}\Prob(M)\to T_{\nu_t}\Prob(M) \). 
Still, at each measure \( \nu\in \Prob(M) \), we have the orthogonal decomposition
\[
\begin{split}
    L^2(\nu,TM) &=\overline{\nabla C^\infty(M)}^{L^2(\nu; TM)}\oplus \{z\in L^2(\nu;TM): \nabla\!\cdot (z\nu)=0\} \\ 
                &= T_\nu \Prob(M) \oplus T_\nu \Prob(M)^\perp
\end{split}
\]
with orthogonal projection \( P_\nu:L^2(\nu, TM)\to T_\nu \Prob(M) \). 
By \cite[Proposition 1.30]{gigli_second_2012}, we have that if \( \tau_t \) is a curve in \( \Prob(M) \) and \( (\tau_t,z_t) \) satisfies the continuity equation, then \( P_{\tau_t}z_t\in T_{\tau_t}\Prob(M) \) is the vector field whose norm realizes the metric derivative of \( \tau_t \).

Composing the transfer operator with this projection therefore produces a genuine tangent field.
For \( f_\ast \) between Wasserstein spaces of complete manifolds this is the route taken in \cite[Corollary 3.13]{lessel_differentiable_2020}.

\begin{prp}[Velocity field of the image curve]\label{prp:image.velocity}
    In the setting of Proposition~\ref{prp:functoriality.pushforward}, the velocity field of \( \nu_t \) is
    \[
        v_t^\nu\coloneqq P_{\nu_t}\Transf_{\mu_t}v_t\ \in\ L^2_0(\nu_t),
        \qquad\text{and}\qquad
        \lvert\dot\nu_t\rvert=\big\lVert v_t^\nu\big\rVert_{L^2(\nu_t)}
        \quad\text{for a.e.\ }t\in I .
    \]
\end{prp}

\begin{proof}
    By Proposition~\ref{prp:functoriality.pushforward}\ref{item:w.cont} the pair \( (\nu_t,\Transf_{\mu_t}v_t) \) solves the continuity equation, and \( t\mapsto\lVert\Transf_{\mu_t}v_t\rVert_{L^2(\nu_t)}\in L^1(I) \) by Proposition~\ref{prp:functoriality.pushforward}\ref{item:w.bound}.
    The claim is \cite[Proposition~1.30]{gigli_second_2012} applied to that pair: the projection onto \( L^2_0(\nu_t) \) of any solution of the continuity equation is the velocity field, and its norm realises the metric derivative.
\end{proof}

\begin{prp}\label{prp:pushforward.rate}
    Let \( \phi:M\to M \) be Lipschitz and let \( \mu_0\in\Prob(M) \) satisfy \( \phi_\ast\mu_0=\mu_0 \).
    Then, for every \( \mu\in\Prob(M) \),
    \begin{equation}\label{eq:pushforward.rate}
        \dw\left(\mu,\phi_\ast\mu\right)\ \le\ \left(1+\Lip(\phi)\right)\,\dw(\mu_0,\mu) .
    \end{equation}
    In particular, if \( (\mu_t)_{t\in I} \) is any curve through \( \mu_0 \) and \( \nu_t\coloneqq\phi_\ast\mu_t \), then
    \[
        \dw(\mu_t,\nu_t)=O\left(\dw(\mu_0,\mu_t)\right);
    \]
    if moreover \( (\mu_t) \) is absolutely continuous with velocity field \( (v_t) \) satisfying \( V \coloneqq\esssup_{s\in[0,t_0]}\lVert v_s\rVert_{L^2(\mu_s)}<\infty \), then \( \dw(\mu_t,\nu_t)\le(1+\Lip(\phi))V\,t \) for \( t\in[0,t_0] \).
\end{prp}

\begin{proof}
    For a transport plan \( \Pi \) between \( \mu_0 \) and \( \mu \), the pushforward \( (\phi\times\phi)_\ast\Pi \) is a transport plan between \( \phi_\ast\mu_0 \) and \( \phi_\ast\mu \).
    Its cost is at most \( \Lip(\phi)^2 \) times that of \( \Pi \), and the map \( \phi_\ast \) is \( \Lip(\phi) \)-Lipschitz on \( (\Prob(M),\dw) \).
    Using \( \phi_\ast\mu_0=\mu_0 \),
    \[
        \dw(\mu,\phi_\ast\mu)\le\dw(\mu,\mu_0)+\dw(\phi_\ast\mu_0,\phi_\ast\mu)\le\left(1+\Lip(\phi)\right)\dw(\mu_0,\mu),
    \]
    which is \eqref{eq:pushforward.rate}.

    For the last assertion, we integrate Equation \eqref{eq:metric.derivative}, obtaining \( \dw(\mu_0,\mu_t)\le\int_0^t\lVert v_s\rVert_{L^2(\mu_s)}\,ds\le V t \), as desired.
\end{proof}

Estimate \eqref{eq:pushforward.rate} is parametrisation free, and the rate \( O(t) \) is a statement about the parametrisation of \( (\mu_t) \) rather than about \( \phi \).
Some hypothesis beyond integrability of the speed is genuinely needed: by H\"older, \( \lVert v_\cdot\rVert_{L^2(\mu_\cdot)}\in L^p(I) \) yields only
\[
    \dw(\mu_0,\mu_t)\le\int_0^t\lVert v_s\rVert_{L^2(\mu_s)}\,ds\le t^{1-1/p}\,\lVert v_\cdot\rVert_{L^p(0,t)}=o\left(t^{1-1/p}\right),
\]
and this is optimal, in the sense that for every \( \alpha\in(0,1) \) the conclusion \( \dw(\mu_t,\nu_t)=O(t) \) fails for a curve whose speed lies in \( L^p \) for all \( p<(1-\alpha)^{-1} \); Example~\ref{exm:rate.sharpness} exhibits such a curve on the circle.

Proposition~\ref{prp:pushforward.rate} bounds \( \dw(\mu_t,\nu_t) \) from above and nothing more, so it cannot separate a genuinely linear rate from the degenerate case \( \dw(\mu_t,\nu_t)=o(t) \).
That case is the one of interest: it says exactly that \( \mu_0 \) can be deformed along \( (\mu_t) \) without losing invariance to first order.
Deciding whether or not this holds needs the exact rate, and upper and lower bounds are not equally easy to obtain.

Any coupling gives an upper bound, and the coupling by the two exponential maps already gives \( \dw(\mu_t,\nu_t)\le t\lVert v_0-v_0^\nu\rVert_{L^2(\mu_0)}+o(t) \), once both curves are known to be exponential to first order.
A lower bound has to come from the dual side, and the dual object chosen decides the constant: going through \( \mathbf{w}_1 \) and Kantorovich--Rubinstein, which asks only for a Lipschitz test function, costs a definite factor, so we use the duality of the quadratic cost itself.

\begin{lem}\label{lem:duality.lower}
    Let \( \mu\in\Prob(M) \) and \( u\in L^2(\mu;TM) \).
    Let \( (\alpha_t)_{t>0} \) and \( (\beta_t)_{t>0} \) be families in \( \Prob(M) \) such that \( \alpha_t\rightharpoonup\mu \) weakly\( ^\ast \) as \( t\searrow0 \) and, for every \( \psi\in C^\infty(M) \),
        \[
            \int_M\psi\,d\beta_t-\int_M\psi\,d\alpha_t
            =t\,\big\langle\nabla\psi,\,u\big\rangle_{L^2(\mu)}+o(t),\qquad t\searrow0 .
        \]
    Then
    \[
        \liminf_{t\searrow0}\ \frac{\dw(\beta_t,\alpha_t)}{t}\ \ge\ \big\lVert P_\mu u\big\rVert_{L^2(\mu)} .
    \]
\end{lem}

\begin{proof}
    We consider the cost \( c(x,y)\coloneqq\tfrac12d(x,y)^2 \) and its \( c \)-transforms, in the sense of \cite[Chapter~5]{villani_optimal_2009}; of the duality theory we shall use only the elementary inequality \( \psi_t(x)+\psi_t^{\,c}(y)\le c(x,y) \), and not the duality theorem itself.

    Fix a test function \( \psi\in C^\infty(M) \), set \( L\coloneqq\lVert\nabla\psi\rVert_\infty \) and \( H\coloneqq\lVert\Hess\psi\rVert_\infty \), which are both finite since \( M \) is compact, and let \( r_0>0 \) be the injectivity radius of \( M \).
    For \( t>0 \) put \( \psi_t\coloneqq t\psi \) and let
    \[
        \psi_t^{\,c}(y)\coloneqq\inf_{x\in M}\big[c(x,y)-\psi_t(x)\big]
    \]
    be its \( c \)-transform, so that \( \psi_t(x)+\psi_t^{\,c}(y)\le c(x,y) \) for all \( x,y\in M \).

    We claim that
    \begin{equation}\label{eq:ctransform}
        \Big\lVert\,\psi_t^{\,c}+t\psi+\tfrac{t^2}{2}\lvert\nabla\psi\rvert^2\,\Big\rVert_\infty
        \ \le\ HL^2\,t^3
        \qquad\text{for }t<\min\{r_0/(2L),\,1/(2H)\} .
    \end{equation}
    Indeed, Taylor's formula along geodesics gives, for \( \xi\in T_yM \) with \( \lvert\xi\rvert<r_0 \),
    \begin{equation}\label{eq:taylor}
        \big\lvert\psi(\exp_y\xi)-\psi(y)-\langle\nabla\psi(y),\xi\rangle\big\rvert\le\tfrac{H}{2}\lvert\xi\rvert^2 .
    \end{equation}
    For the upper bound in \eqref{eq:ctransform}, take \( x=\exp_y\left(t\nabla\psi(y)\right) \), which is a valid choice because \( tL<r_0 \).
    Then \( d(x,y)=t\lvert\nabla\psi(y)\rvert \) and \eqref{eq:taylor} gives \( \psi(x)\ge\psi(y)+t\lvert\nabla\psi(y)\rvert^2-\tfrac{H}{2}t^2\lvert\nabla\psi(y)\rvert^2 \), whence
    \begin{align*}
        \psi_t^{\,c}(y)
        &\le\tfrac{t^2}{2}\lvert\nabla\psi(y)\rvert^2-t\psi(y)-t^2\lvert\nabla\psi(y)\rvert^2+\tfrac{H}{2}t^3\lvert\nabla\psi(y)\rvert^2\\
        &\le-t\psi(y)-\tfrac{t^2}{2}\lvert\nabla\psi(y)\rvert^2+\tfrac{HL^2}{2}t^3 .
    \end{align*}
    For the lower bound, any \( x \) at least as good as \( x=y \) satisfies \( \tfrac12d(x,y)^2\le t[\psi(x)-\psi(y)]\le tL\,d(x,y) \), so \( d(x,y)\le2tL<r_0 \).
    Letting \( x=\exp_y\xi \) with \( r\coloneqq\lvert\xi\rvert\le2tL \) in Equation \eqref{eq:taylor} gives
    \[
        c(x,y)-\psi_t(x)\ \ge\ -t\psi(y)+\tfrac12(1-Ht)r^2-tr\lvert\nabla\psi(y)\rvert
        \ \ge\ -t\psi(y)-\frac{t^2\lvert\nabla\psi(y)\rvert^2}{2(1-Ht)} ,
    \]
    the last step by minimising the quadratic in \( r \).
    Since \( Ht<\tfrac12 \) we have \( (1-Ht)^{-1}\le1+2Ht \), and \eqref{eq:ctransform} follows.

    Now, for any transport plan \( \pi \) between \( \beta_t \) and \( \alpha_t \), we have \( \int\psi_t\,d\beta_t+\int\psi_t^{\,c}\,d\alpha_t=\int[\psi_t(x)+\psi_t^{\,c}(y)]\,d\pi\le\int c\,d\pi \); taking \( \pi \) optimal, we obtain
    \begin{equation}\label{eq:duality}
        \tfrac12\,\dw(\beta_t,\alpha_t)^2\ \ge\ \int_M\psi_t\,d\beta_t+\int_M\psi_t^{\,c}\,d\alpha_t .
    \end{equation}
    Inserting \eqref{eq:ctransform} into \eqref{eq:duality},
    \[
        \tfrac12\,\dw(\beta_t,\alpha_t)^2
        \ \ge\ t\Big[\int\psi\,d\beta_t-\int\psi\,d\alpha_t\Big]
        -\frac{t^2}{2}\int\lvert\nabla\psi\rvert^2\,d\alpha_t-HL^2t^3 .
    \]
    By hypothesis the bracket equals \( t\langle\nabla\psi,u\rangle_{L^2(\mu)}+o(t) \).
    The integrand \( \lvert\nabla\psi\rvert^2 \) is continuous, so the weak convergence of \( \alpha_t \) gives \( \int\lvert\nabla\psi\rvert^2\,d\alpha_t\to\lVert\nabla\psi\rVert^2_{L^2(\mu)} \).
    Hence
    \[
        \liminf_{t\searrow0}\frac{\dw(\beta_t,\alpha_t)^2}{t^2}
        \ \ge\ 2\big\langle\nabla\psi,u\big\rangle_{L^2(\mu)}-\big\lVert\nabla\psi\big\rVert^2_{L^2(\mu)} .
    \]
    The left-hand side does not depend on \( \psi \), so we may pass to the supremum on the right.
    The functional \( g\mapsto2\langle g,u\rangle_{L^2(\mu)}-\lVert g\rVert^2_{L^2(\mu)} \) is continuous on \( L^2(\mu;TM) \), and smooth gradients are dense in \( L^2_0(\mu) \) by \eqref{eq:L20}, so the supremum over \( \psi\in C^\infty(M) \) equals the supremum over \( g\in L^2_0(\mu) \).
    For such \( g \) one has \( \langle g,u\rangle=\langle g,P_\mu u\rangle \), hence
    \[
        2\langle g,u\rangle-\lVert g\rVert^2=\lVert P_\mu u\rVert^2-\lVert g-P_\mu u\rVert^2\le\lVert P_\mu u\rVert^2 ,
    \]
    with equality at \( g=P_\mu u\in L^2_0(\mu) \).
\end{proof}

In this form the lemma is a partial converse to Lemma~\ref{lem:pointwise.continuity}: differentiability of the integrals of test functions, which is strictly weaker than \(\dw\)-differentiability\footnote{
		See Remark~\ref{rmk:weaker} at the end of Section~\ref{sec:torus}.
}, nevertheless controls the \(\dw\)-separation from below.

The choice of dual object is what fixes the constant, and the linear route cannot give the right one.
Bounding \( \dw\ge\mathbf w_1 \) and invoking Kantorovich--Rubinstein yields, in the setting of Lemma~\ref{lem:duality.lower}, only
\[
    \liminf_{t\searrow0}\frac{\dw(\beta_t,\alpha_t)}{t}\ \ge\ \sup\big\{\langle g,u\rangle_{L^2(\mu)}:\ g\in L^2_0(\mu),\ \lVert g\rVert_\infty\le1\big\},
\]
a supremum bounded above by \( \lVert u\rVert_{L^1(\mu)} \), and \( \lVert u\rVert_{L^1(\mu)}<\lVert u\rVert_{L^2(\mu)} \) unless \( \lvert u\rvert \) is \( \mu \)-almost everywhere constant.
No choice of \( g \) improves this: the linear cost does not see the quadratic term of \eqref{eq:ctransform}, which is precisely what produces the square completed in the last step of the proof.
On \( \T^1 \) with \( \mu=\Leb \) and \( u(x)=\cos2\pi x \), for instance, the linear route gives at best \( \lVert u\rVert_{L^1}=2/\pi \), against the correct value \( \lVert u\rVert_{L^2}=1/\sqrt2 \).

\begin{cor}[Metric derivative of exponential deformations]
    \label{cor:metric.derivative}
    Let \( \mu\in\Prob(M) \) and \( v\in L^2_0(\mu) \).  
    Then
    \[
        \dw\left(\mu,\ (\exp tv)_\ast\mu\right)=t\,\lVert v\rVert_{L^2(\mu)}+o(t),\qquad t\searrow0 .
    \]
\end{cor}
\begin{proof}
    The upper bound is the coupling \( (\Id,\exp tv)_\ast\mu \).
    For the lower bound apply Lemma~\ref{lem:duality.lower} with \( \alpha_t\equiv\mu \), \( \beta_t\coloneqq(\exp tv)_\ast\mu \) and \( u\coloneqq v \).
    The weak convergence is straightforward, and the hypothesis follows from \eqref{eq:taylor}, since
    \[
        \Big\lvert\int\psi\,d\beta_t-\int\psi\,d\mu-t\langle\nabla\psi,v\rangle_{L^2(\mu)}\Big\rvert
        \le\tfrac{H}{2}\,t^2\lVert v\rVert^2_{L^2(\mu)} .
    \]
    Finally \( \lVert P_\mu v\rVert=\lVert v\rVert \), as \( v\in L^2_0(\mu) \).
\end{proof}

\begin{lem}[Pointwise continuity equation]
    \label{lem:pointwise.continuity}
    Let \( (\mu_t) \) be a curve in \( \Prob(M) \), differentiable at \( 0 \) in the sense that \(\dw(\mu_t,(\exp tv)_\ast\mu_0)=o(t)\) for some \(v\in L^2(\mu_0;TM)\). 
    Then \(\left.\frac{d}{dt}\right\rvert_{0}\int\psi\,d\mu_t=\int\langle\nabla\psi,v\rangle\,d\mu_0\) for every \(\psi\in C^\infty(M)\).
\end{lem}

On the circle this is \cite[Lemma 7.1]{kloeckner_optimal_2013_preprint}, where the extension to a general manifold, with \( \nabla\psi \) in place of \( \psi' \), is indicated without proof.

\begin{proof}
    Let \( \Pi_t \) be an optimal plan between \( \mu_t \) and \( (\exp tv)_\ast\mu_0 \).  
    Then, using that \( \psi \) is \( \lVert \nabla\psi \rVert_\infty \)-Lipschitz and Cauchy-Schwarz, 
    \begin{align*}
        \left\lvert\int\psi\,d\mu_t-\int\psi\,d(\exp tv)_\ast\mu_0)\right\rvert
        &\le\lVert \nabla\psi \rVert_\infty\int d(x,y)\,d\Pi_t(x,y)\\
        &\le\lVert \nabla\psi \rVert_\infty\left(\int d(x,y)^2d\Pi_t(x,y)\right)^{1/2} \\
        &=\lVert \nabla\psi \rVert_\infty\,\dw(\mu_t,(\exp tv)_\ast\mu_0)=o(t).
    \end{align*}
    So it suffices to differentiate \( t\mapsto\int\psi\,d((\exp tv)_\ast\mu_0) \) at \( 0 \). 
    By the definition of the pushforward and \eqref{eq:taylor} at \(x\) with \(\xi=tv(x)\), we obtain
    \[
    \begin{split}
        \int\psi\,d\bigl((\exp tv)_\ast\mu_0\bigr)-\int\psi\,d\mu_0
        & =\int\left(\psi(\exp_x(tv(x))))-\psi(x)\right)d\mu_0(x) \\
        &=t\!\int\langle\nabla\psi,v\rangle\,d\mu_0+O(t^2)\lVert\Hess\psi\rVert_\infty\lVert v \rVert^2_{L^2(\mu_0)},
    \end{split}
    \]
    the remainder being finite because \( \lVert v \rVert_{L^2(\mu_0)}<\infty \).
    Combining the two estimates gives the claim.
\end{proof}

\subsection{Densities and linear response}
\label{sec:exp.deformations}

Both estimates obtained so far compare measures of the form \( (\exp tu)_\ast\mu \) for one and the same \( \mu \).
Proposition~\ref{prp:pushforward.rate} bounds \( \dw \) from above by exhibiting a coupling, and Lemma~\ref{lem:duality.lower} bounds it from below by exhibiting a Kantorovich potential.
In both cases the coupling used is the identity, or a small perturbation of it, and no regularity of \( \mu \) is required.
Neither argument applies to the measure we actually want to describe, which is \( \phi_\ast(\exp tv)_\ast\mu_0 \).

That measure is \( \mu_0 \) pushed forward by \( \phi\circ\exp(tv) \), and \( \phi \) is at a definite distance from the identity.
The plan induced by this map, against a competitor \( (\exp tw)_\ast\mu_0 \), has cost
\[
    \int_Md\left(\phi(\exp_x tv(x)),\ \exp_x(tw(x))\right)^2\,d\mu_0(x),
\]
which converges to \( \int_Md(\phi(x),x)^2\,d\mu_0(x) \) as \( t\searrow0 \), a positive constant unless \( \phi=\Id \).
Nothing is gained by composing \( \phi \) with some other exponential map, since any such plan still displaces the mass of \( \mu_0 \) by \( d(\phi(x),x)+O(t) \).
A plan of cost \( o(t^2) \) must move \( \mu_0 \)-almost every point by \( o(t) \), and must therefore take account of the fact that the mass sitting at \( y \) after the pushforward has been collected from the whole fibre \( \phi^{-1}(y) \).
We know of no formula for such a plan, and we shall not need one, because the two measures can be compared through their densities instead.

This subsection assembles the tools that comparison requires:
Lemma~\ref{lem:density.comparison} bounds \( \dw \) by the \(L^2\) of the difference of the densities, and Proposition~\ref{prp:linear.response} expands the density of \( (\exp tu)_\ast\mu \) to first order, uniformly on \( M \).
Both ask the base measure to be absolutely continuous with a positive \( C^1 \) density, a hypothesis which none of the previous arguments needed.

Throughout, we denote by \( \exp \) the Riemannian exponential of \( M \).
For a field \( u\in L^2(\mu;TM) \), we write \( (\exp tu)_\ast\mu \) for the pushforward of \( \mu \) under \( x\mapsto\exp_x(tu(x)) \). 
This is defined for every \( t \), \( M \) being closed and hence geodesically complete, and no injectivity is required.
Since \( M \) is closed its injectivity radius \( r_0 \) is positive; we fix once and for all \( r_1=r_1(M)\in(0,r_0/4) \) such that any two points of a ball \( B(x,r_1) \) are joined by a unique minimising geodesic, contained in \( B(x,4r_1) \).
The following Lemma gives the estimates on the exponential map that we shall need.
The proof comprises standard computations in Riemannian Geometry and Comparison Theory, and uses no Optimal Transport theory at all.
The proof is in Appendix~\ref{sec:riemannian.estimates}.

\begin{lem}
\label{lem:exp.estimates}
    There are constants \( \Lambda\ge1 \) and \( C_\ast\ge0 \), both depending only on \( M \), such that for every \( x\in M \), the exponential \(\exp_x: B(0, r_1) \to M\) is \(\Lambda\)-Lipschitz.
    Moreover, 
    \[
        \big\lvert\,d(\exp_x\xi,\exp_x\eta)-\lvert\xi-\eta\rvert\,\big\rvert \le C_\ast\max(\lvert\xi\rvert,\lvert\eta\rvert)^2\lvert\xi-\eta\rvert 
    \] 
    for \(\lvert \xi \rvert, \lvert \eta\rvert \le r_1\). 
\end{lem}

We begin with the coupling bound, which is the only result of this subsection requiring no regularity at all.
The pointwise inequality one would like to use, namely \( d(\exp_x\xi,\exp_x\eta)\le\lvert\xi-\eta\rvert \), is false on a closed manifold.
By the Rauch comparison theorem \cite[Theorems~1.28 and 1.29]{cheeger_comparison_1975} it holds when the sectional curvature is non-negative, and it is reversed when the curvature is non-positive.
The lemma below says that, regardless, the curvature loss is absorbed in the limit \( t\searrow0 \), providing an \(L^2\)-bound on the \(2\)-Wasserstein distance.
\begin{lem}
\label{lem:exp.upper}
    Let \( \mu\in\Prob(M) \) and \( u,u'\in L^2(\mu;TM) \).  
    Then
    \[
        \lim_{t\searrow0}\frac1t\left(\int_Md\left(\exp_x(tu(x)),\exp_x(tu'(x))\right)^2\,d\mu(x)\right)^{1/2}
        =\lVert u-u'\rVert_{L^2(\mu)},
    \]
    and consequently
    \[
        \limsup_{t\searrow0}\frac{1}{t}\dw\left((\exp tu)_\ast\mu,(\exp tu')_\ast\mu\right) \ \le\ \lVert u-u'\rVert_{L^2(\mu)} .
    \]
    Moreover, let \( \mathcal U\subset L^2(\mu;TM) \) be a dominated family: there exists \( G\in L^2(\mu) \) with \( \lvert u\rvert\le G \) \( \mu \)-almost everywhere for every \( u\in\mathcal U \).
    Then
    \begin{equation}\label{eq:exp.upper.unif}
        \limsup_{t\searrow0}\ \sup_{u,u'\in\mathcal U}\left(\frac1t\,\dw\left((\exp tu)_\ast\mu,(\exp tu')_\ast\mu\right)-\lVert u-u'\rVert_{L^2(\mu)}\right)\ \le\ 0 .
    \end{equation}
\end{lem}

\begin{proof}
    Put \( h_t(x)\coloneqq t^{-1}d\left(\exp_x(tu(x)),\exp_x(tu'(x))\right) \).
    We can apply the triangle inequality at basepoint \(x\) (cf. Estimate \eqref{eq:triangle.inequality}), thus obtaining \( h_t\le\lvert u\rvert+\lvert u'\rvert\in L^2(\mu) \) for every \( t>0 \).
    For \( \mu \)-a.e.\ \( x \) both \( u(x) \) and \( u'(x) \) are finite, and then Lemma~\ref{lem:exp.estimates} applies as soon as \(t\) is sufficiently small so that 
    \[
        t\max(\lvert u(x)\rvert,\lvert u'(x)\rvert)\le r_1,
    \]
    which yields \( h_t(x)\to\lvert u(x)-u'(x)\rvert \).  
    Using dominated convergence to integrate \(h_t\) gives the first assertion.
    The second claim then follows because \( (\exp tu,\exp tu')_\ast\mu \) is a transport plan of cost \( t^2\lVert h_t\rVert_{L^2(\mu)}^2 \).
    It remains to prove \eqref{eq:exp.upper.unif}.
    For \( R>0 \) put \( A_R\coloneqq\{G>R\} \) and \( \delta_R\coloneqq\left(\int_{A_R}G^2\,d\mu\right)^{1/2} \), so that \( \delta_R\to0 \) as \( R\to\infty \) by dominated convergence.
    Let \( u,u'\in\mathcal U \) and let \( 0<t\le r_1/R \).
    On \( A_R \) the triangle inequality at basepoint \( x \) gives \( d(\exp_xtu,\exp_xtu')\le t(\lvert u\rvert+\lvert u'\rvert)\le2tG \).
    On \( M\setminus A_R \) we have \( \lvert tu\rvert,\lvert tu'\rvert\le tR\le r_1 \), so Lemma~\ref{lem:exp.estimates} gives \( d(\exp_xtu,\exp_xtu')\le t\lvert u-u'\rvert\left(1+C_\ast t^2R^2\right) \).
    Splitting the cost of the trivial coupling along \( A_R \) and its complement, and using \( \sqrt{a+b}\le\sqrt a+\sqrt b \),
    \begin{equation}\label{eq:exp.upper.split}
        \frac1t\,\dw\left((\exp tu)_\ast\mu,(\exp tu')_\ast\mu\right)\ \le\ 2\delta_R+\left(1+C_\ast t^2R^2\right)\lVert u-u'\rVert_{L^2(\mu)} .
    \end{equation}
    Every \( u,u'\in\mathcal U \) satisfies \( \lVert u-u'\rVert_{L^2(\mu)}\le2\lVert G\rVert_{L^2(\mu)} \).
    Given \( \eta>0 \), choose \( R \) with \( 2\delta_R\le\eta/2 \), and then \( t_0\le r_1/R \) with \( 2C_\ast t_0^2R^2\lVert G\rVert_{L^2(\mu)}\le\eta/2 \).
    Then, for \( t<t_0 \), the right-hand side of \eqref{eq:exp.upper.split} is at most \( \lVert u-u'\rVert_{L^2(\mu)}+\eta \), uniformly in \( u,u'\in\mathcal U \), as we wanted.

\end{proof}

We turn to the tool which replaces the coupling.
The distance \( \dw \) is not comparable to any \( L^p \) distance between densities: two densities may be uniformly close and still describe mass which has to travel a long way.
It is, however, comparable to a negative Sobolev norm whenever the reference density is bounded away from zero, and on a closed manifold that norm is in turn dominated by the \( L^2 \) norm through the spectral gap of the Laplacian.

\begin{lem}\label{lem:density.comparison}
    Let \( h_0,h_1\in L^2(\vol) \) be probability densities with respect to \( \vol \) and suppose \( h_0\ge c>0 \) \( \vol \)-a.e.  
    Then
    \[
        \dw\left(h_0\vol,\ h_1\vol\right)\ \le\ \frac{2}{\sqrt{c\,\lambda_1}}\,\lVert h_0-h_1\rVert_{L^2(\vol)} ,
    \]
    where \( \lambda_1>0 \) is the first non-zero eigenvalue of the Laplace--Beltrami operator \( \Delta\coloneqq -\Div\nabla \).
\end{lem}

\begin{proof}
    First we apply the Sobolev estimate from \cite[Corollary 3]{peyre_comparison_2018} to \( \mu=h_0\vol\ge c\vol \) and \( \nu=h_1\vol \), yielding
    \begin{equation}\label{eq:peyre.estimate}
        \dw\left(h_0\vol,\ h_1\vol\right)\ \le\ \frac{2}{\sqrt c}\,\lVert h_0-h_1\rVert_{\Hm}.
    \end{equation}
    \( \lVert\cdot\rVert_{\Hm} \) is a negative Sobolev norm defined by duality,
    \[
        \lVert\nu\rVert_{\Hm}
        =\sup\left\{\lvert\textstyle\int_Mg\,d\nu\rvert\ :\
            \textstyle\int_M\lvert\nabla g\rvert^2\,d\vol\le1\right\} ;
    \]
    
    Now, the Laplace--Beltrami operator has an inverse \(\Delta^{-1}\) on the space \(\{f: \int_M f d\vol = 0\}\) of mean-zero functions. 
    For such a mean-zero \( f\in L^2(\vol) \) we have \( \lVert f\rVert_{\Hm}=\lVert\nabla\Delta^{-1}f\rVert_{L^2(\vol)} \), since 
    \[
        \int_Mfg\,d\vol=\int_M\langle\nabla\Delta^{-1}f,\nabla g\rangle\,d\vol
    \] 
    and Cauchy--Schwarz is an equality at \( g=\Delta^{-1}f \).

    Finally, we consider \( f\coloneqq h_0-h_1 \), noticing that it has zero-mean by hypothesis.
    We then expand \( f = \sum_k\hat f_k\varphi_k \) in an orthonormal basis of eigenfunctions \(\varphi_1, \varphi_2, \cdots,\) of \( \Delta \), associated to eigenvalues \( 0=\lambda_0<\lambda_1\le\cdots \).
    Since \( f \) has zero mean, we have \( \lVert\nabla\Delta^{-1}f\rVert^2_{L^2(\vol)}=\sum_{k\ge1}\lambda_k^{-1}\lvert\hat f_k\rvert^2 \le\lambda_1^{-1}\lVert f\rVert^2_{L^2(\vol)} \).
    Together with the estimate \eqref{eq:peyre.estimate}, this yields the desired result.
\end{proof}

Now, in order to use Lemma~\ref{lem:density.comparison} effectively, we need a good grasp of the densities being compared. 
Our next result computes the first-order effect of an exponential deformation on the density: \( (\exp tu)_\ast\mu=\mu-t\,\nabla\!\cdot(u\mu)+o(t) \), with the remainder uniform on \( M \) rather than merely distributional.
The strength of the statement is entirely in that uniformity, since it is what Lemma~\ref{lem:density.comparison} uses.
Its limitation is the hypothesis \( \rho\in C^1 \), which is where the regularity announced above is needed.

\begin{prp}\label{prp:linear.response}
    Let \( \mu=\rho\,\vol \) with \( \rho\in C^1(M) \) and \( \rho>0 \), and let \( u\in C^1(M;TM) \).  
    Then there is \( t_0>0 \) such that, for \( \lvert t\rvert<t_0 \), the map \( T^u_t(x)\coloneqq \exp_x(tu(x)) \) is a \( C^1 \) diffeomorphism of \( M \), the measure \( (T^u_t)_\ast\mu \) is absolutely continuous with a continuous density \( \rho_t[u] \), and
    \begin{equation}\label{eq:linear.response}
        \big\lVert\,\rho_t[u]-\rho+t\,\Div(\rho u)\,\big\rVert_{C^0(M)}=o(t),
        \qquad t\to0 .
    \end{equation}
    The threshold \( t_0 \) may be chosen depending on \( u \) only through \( \lVert u\rVert_{C^1(M)} \), and is therefore uniform on bounded subsets of \( C^1(M;TM) \).
    The remainder in \eqref{eq:linear.response}, on the other hand, is uniform for \( u \) ranging in a compact subset of \( C^1(M;TM) \), and in particular for \( u \) in a bounded subset of a finite-dimensional subspace.
\end{prp}

On the circle this is \cite[Proposition 3.2]{kloeckner_optimal_2013}, where it is also observed that the remainder depends on the moduli of continuity of \( \nabla u \) and \( \nabla\rho \) and is not \( O(t^2) \) unless both are of class \( C^2 \).
This is the obstruction discussed in Remark~\ref{rmk:frechet.derivative}.

\begin{proof}
    Consider the exponential evaluation \(\mathcal E(x,\xi)=\exp_x\xi\). 
    Since \(u\in C^1(M;TM)\), the map \(x\mapsto (x,tu(x))\) is \(C^1\). 
    By the chain rule,
    \[
        D_xT_t^u = D_1\mathcal E_{(x,tu(x))}+ t\,D_2\mathcal E_{(x,tu(x))}\circ \nabla u(x).
    \]
    At \(t=0\), we have
    \(D_1\mathcal E_{(x,0)}=\Id_{T_xM}\) and \(D_2\mathcal E_{(x,0)}=\Id_{T_xM}\).
    Hence, by compactness of \(M\),
    \[
        D_xT_t^u=\Id+t\,\nabla u(x)+o(t)
    \]
    uniformly in \(x\).
    In particular, \(D_xT_t^u\) is invertible for every \(x\) when \(\lvert t\rvert\) is sufficiently small, with the required bound depending only on \(\lVert u\rVert_{C^1}\).
    Thus \(T_t^u\) is a local diffeomorphism.
    As \(M\) is closed, \(T_t^u\) is proper, hence a covering map, and the number of its sheets equals the absolute value of its degree, \( M \) being connected and oriented.
    The homotopy \( s\mapsto T^u_{st} \) joins \( T^u_t \) to the identity, so that degree is one and \( T^u_t \) is a diffeomorphism.

    Consider the Jacobian \(J_t(x)\coloneqq \lvert\det D_xT_t^u\rvert\), where the determinant is computed with respect to \(\vol\). 
    The map \((t,x)\mapsto\partial_tJ_t(x)\) is continuous, and \(\partial_tJ_t|_{t=0}=\tr(\nabla u)=\Div u\), where \(\nabla u\) is understood as the \((1,1)\)-tensor \(w \mapsto \nabla_w u\). 
    Therefore
    \[
        J_t=1+t\,\Div u+o(t), \qquad J_t^{-1}=1-t\,\Div u+o(t),
    \]
    uniformly on \(M\). 
    By the change of variables formula, \(\rho_t[u]=\bigl(\rho/J_t\bigr)\circ(T_t^u)^{-1}\), which is continuous.
    In other words,
    \[
        \rho_t[u]
        =\bigl(\rho J_t^{-1}\bigr)\circ(T_t^u)^{-1} =\rho\circ(T_t^u)^{-1}-t\,(\rho\,\Div u)\circ(T_t^u)^{-1}+o(t).
    \]        
    Inverting \( T^u_t \) with the uniform bound on \( (D_xT^u_t)^{-1} \) obtained above gives \( d\big((T^u_t)^{-1}(y),\exp_y(-tu(y))\big)=o(t) \) uniformly in \( y \).
    Since \( \rho\in C^1(M) \) it is Lipschitz, and \( \nabla\rho \) is uniformly continuous because \( M \) is compact, so first-order Taylor expansion along the geodesics \( s\mapsto\exp_y(-stu(y)) \) gives \( \rho\circ(T^u_t)^{-1}=\rho-t\langle\nabla\rho,u\rangle+o(t) \) uniformly on \( M \).
    On the other hand, \(T_t^u\to\Id\) uniformly, so, by uniform continuity of \(\rho\,\Div u\), we have \((\rho\,\Div u)\circ(T_t^u)^{-1}=\rho\,\Div u+o(1)\), uniformly as well.
    Therefore
    \[
        \rho_t[u]=\rho-t\langle\nabla\rho,u\rangle-t\,\rho\,\Div u+o(t)=\rho-t\,\Div(\rho u)+o(t),
    \]
    uniformly on \(M\), which is \eqref{eq:linear.response}.

    Finally, let \( K\subset C^1(M;TM) \) be compact.
    Then \( \{u:u\in K\} \) and \( \{\nabla u:u\in K\} \) are uniformly bounded and uniformly equicontinuous families, by Arzelà--Ascoli.
    Every expansion above was obtained from the smooth map \( \exp \) and from \( u \) and \( \nabla u \) by composition and differentiation, and depends on \( u \) only through those two bounds and those two moduli of continuity.
    Hence \( t_0 \) and all the \( o(t) \) terms may be chosen uniformly for \( u\in K \).
    A bounded subset of a finite-dimensional subspace of \( C^1(M;TM) \) is compact, all norms being equivalent there.
\end{proof}
    
The rest of the section consists of applications of the same argument: one expands the densities of the two measures to be compared by Proposition~\ref{prp:linear.response}, observes that their first-order terms coincide, and concludes by Lemma~\ref{lem:density.comparison} that the \( \dw \)-distance between them is \( o(t) \).

\subsection{Metric invisibility of solenoidal directions}
\label{sec:solenoidal}
We show in this subsection that the solenoidal directions are irrelevant up to first-order: they \emph{stir} the mass of \( \mu \) rather than moving it.

The first lemma is the approximation statement that every argument below needs, and it is where the absolute continuity of \( \mu \) is used for the first time.

\begin{lem}
\label{lem:regular.solenoidal}
    Let \( \mu=\rho\,\vol \) with \( \rho\in C^1(M) \) and \( \rho>0 \).
    Then the fields of class \( C^1 \) belonging to \( L^2_0(\mu)^\perp \) are dense in \( L^2_0(\mu)^\perp \) for the norm of \( L^2(\mu;TM) \).
\end{lem}

\begin{proof}
    Let \( z\in L^2_0(\mu)^\perp \) and set \( w\coloneqq\rho z \).
    Since \( \rho \) is continuous and positive on the compact manifold \( M \), we have \( 0<\min\rho\le\rho\le\max\rho<\infty \), and hence
    \[
        \min\rho\cdot\lVert z\rVert^2_{L^2(\mu)} \le \lVert w\rVert^2_{L^2(\vol)} =\int_M\lvert z\rvert^2\rho^2\,d\vol \le \max\rho\cdot\lVert z\rVert^2_{L^2(\mu)}.
    \]
    Thus \( z\mapsto\rho z \) is an isomorphism from \( L^2(\mu;TM) \) onto \( L^2(\vol;TM) \).
    This map, and the equivalence \( \nabla\!\cdot(z\mu)=0\iff\Div(\rho z)=0 \) that we use next, appear in \cite[\S3]{gomes_differential_2024}, where they aid in comparing tangent spaces at different base points.
    Here they are used to transport the problem to \( \vol \), where the Hodge decomposition is available.
    Since \( z \) satisfies \( \nabla\!\cdot(z\mu)=0 \), we have
    \[
        0
        =\int_M\langle\nabla\psi,z\rangle\,d\mu =\int_M\langle\nabla\psi,w\rangle\,d\vol \qquad\forall\psi\in C^\infty(M),
    \]
    that is, \( w \) is divergence-free.
    It suffices therefore to approximate \( w \) in \( L^2(\vol;TM) \) by smooth divergence-free fields and to divide the approximants by \( \rho \) at the end.

    We make use of Hodge theory to finish the argument.
    The musical isomorphisms induced by the metric identify vector fields with \( 1 \)-forms, so that \( \Div w=-\delta w^\flat \).
    Let \( \Delta_H=d\delta+\delta d \) be the Hodge--de Rham Laplacian.
    On a closed manifold, this operator is elliptic, self-adjoint, and non-negative, and its spectrum is a discrete set of non-negative eigenvalues whose eigenspaces are finite dimensional and consist of smooth forms.
    This holds on functions and on \( 1 \)-forms alike, and gives orthogonal decompositions
    \[
        L^2(\vol)=\bigoplus_\lambda F_\lambda \qquad\text{and}\qquad L^2\Omega^1(M)=\bigoplus_\lambda E_\lambda
    \]
    into eigenspaces of \( \Delta_H \) (cf. \cite[Chapter 6]{warner_foundations_1983}).
    From \( d^2=0 \), we have \( d\Delta_H=d\delta d=\Delta_Hd \), so \( d \) maps \( F_\lambda \) into \( E_\lambda \).
    Similarly, \( \delta^2=0 \) yields \( \delta\Delta_H=\delta d\delta=\Delta_H\delta \), so \( \delta \) maps \( E_\lambda \) into \( F_\lambda \).

    Enumerate the distinct eigenvalues as \( 0\le\lambda_0<\lambda_1<\lambda_2<\cdots \), and write \( w^\flat=\sum_{j=0}^\infty\omega_j \), where \( \omega_j\in E_{\lambda_j} \).
    We claim that \( \delta\omega_j=0 \) for every \( j \).
    Fix \( j \) and let \( f\in F_{\lambda_j} \).
    Then \( df\in E_{\lambda_j} \), and since \(\delta w^\flat = -\Div w = 0\), we have
    \[
        0=\langle\delta w^\flat,f\rangle_{L^2}=\langle w^\flat,df\rangle_{L^2}=\langle\omega_j,df\rangle_{L^2}=\langle\delta\omega_j,f\rangle_{L^2}.
    \]
    Since \( \delta\omega_j\in F_{\lambda_j} \) and \( f\in F_{\lambda_j} \) was arbitrary, we conclude that \( \delta\omega_j=0 \).

    The partial sums \( w_n\coloneqq\left(\sum_{j\leq n}\omega_j\right)^\sharp\) are therefore smooth and divergence-free, and they converge to \( w \) in \( L^2(\vol;TM) \).
    Setting \( z_n\coloneqq w_n/\rho \), we obtain fields of class \( C^1 \) satisfying \( \nabla\!\cdot(z_n\mu)=0 \) and converging to \( z \) in \( L^2(\mu;TM) \).
\end{proof}

The following is a known fact which we register here for convenience. 
\begin{lem}
\label{lem:smooth.dense}
    For \( \mu\in\Prob(M) \), the set \( C^\infty(M;TM) \) is dense in \( L^2(\mu;TM) \).
\end{lem}

\begin{proof}
    Embed \( M \) isometrically in some \( \R^N \) and regard a field \( v\in L^2(\mu;TM) \) as an element of \( L^2(\mu;\R^N) \) taking values in \( T_xM \) at \( \mu \)-almost every \( x \).
    Continuous \( \R^N \)-valued maps are dense in \( L^2(\mu;\R^N) \), since \( \mu \) is a finite Borel measure on a compact metric space, and smooth maps are uniformly dense in continuous ones.
    Given \( V_n\in C^\infty(M;\R^N) \) with \( V_n\to v \) in \( L^2(\mu;\R^N) \), let \( \Pi_x:\R^N\to T_xM \) be the fibrewise orthogonal projection, which depends smoothly on \( x \), and put \( v_n(x)\coloneqq\Pi_xV_n(x) \).
    Then \( v_n \) is a smooth section and \( \lvert v_n(x)-v(x)\rvert=\lvert\Pi_x\left(V_n(x)-v(x)\right)\rvert\le\lvert V_n(x)-v(x)\rvert \), so that \( v_n\to v \) in \( L^2(\mu;TM) \).
\end{proof}

\begin{lem}[Uniform smoothing on a finite-dimensional subspace]
\label{lem:uniform.smoothing}
    Let \( \mu=\rho\,\vol \) with \( \rho\in C^1(M) \) and \( \rho>0 \), and let \( W\subset L^2(\mu;TM) \) be a subspace of finite dimension \( p \).
    There is a constant \( c_W\ge1 \), depending only on \( W \) and on a choice of basis of it, with the following property.
    For every \( \epsilon>0 \) there are a subspace \( W_\epsilon\subset C^1(M;TM) \) of dimension at most \( p \) and a linear map \( A_\epsilon:W\to W_\epsilon \) such that
    \begin{enumerate}
        \item\label{item:us.approx}
            \( \lVert w-A_\epsilon w\rVert_{L^2(\mu)}\le c_W\,\epsilon\,\lVert w\rVert_{L^2(\mu)} \) for every \( w\in W \);
        \item\label{item:us.compact}
            for every \( R>0 \) the set \( \big\{A_\epsilon w:\ \lVert w\rVert_{L^2(\mu)}\le R\big\} \) is a compact subset of \( C^1(M;TM) \);
        \item\label{item:us.dominated}
            for every \( R>0 \) there is \( G\in L^2(\mu) \) such that \( \lvert w\rvert\le G \) and \( \lvert A_\epsilon w\rvert\le G \) \( \mu \)-almost everywhere, for every \( w\in W \) with \( \lVert w\rVert_{L^2(\mu)}\le R \).
    \end{enumerate}
    Moreover \( A_\epsilon \) may be chosen with values in \( C^\infty(M;TM) \); with values in \( C^\infty(M;TM)\cap L^2_0(\mu) \) if \( W\subset L^2_0(\mu) \); and with values in \( C^1(M;TM)\cap L^2_0(\mu)^\perp \) if \( W\subset L^2_0(\mu)^\perp \).
\end{lem}

\begin{proof}
    Fix a basis \( e_1,\dots,e_p \) of \( W \).
    The coordinate functionals of a finite-dimensional normed space are bounded, so there is \( c\ge1 \) with \( \max_i\lvert a_i\rvert\le c\,\lVert w\rVert_{L^2(\mu)} \) whenever \( w=\sum_ia_ie_i \); put \( c_W\coloneqq pc \).

    Given \( \epsilon>0 \), choose for each \( i \) a field \( e_i^\epsilon \) with \( \lVert e_i-e_i^\epsilon\rVert_{L^2(\mu)}\le\epsilon \), belonging to the class required by the statement.
    In the unrestricted case this is Lemma~\ref{lem:smooth.dense}; in the case \( W\subset L^2_0(\mu) \) it is \eqref{eq:L20}; in the case \( W\subset L^2_0(\mu)^\perp \) it is Lemma~\ref{lem:regular.solenoidal}.
    Let \( W_\epsilon\coloneqq\operatorname{span}\{e_1^\epsilon,\dots,e_p^\epsilon\} \) and let \( A_\epsilon:W\to W_\epsilon \) be the linear map sending \( e_i \) to \( e_i^\epsilon \).
    
    \medskip
    \noindent Item~\ref{item:us.approx}.\\
    This follows from \( \lVert w-A_\epsilon w\rVert_{L^2(\mu)}\le\sum_i\lvert a_i\rvert\,\lVert e_i-e_i^\epsilon\rVert_{L^2(\mu)}\le pc\,\epsilon\,\lVert w\rVert_{L^2(\mu)} \).

    \medskip
    \noindent Item~\ref{item:us.compact}.\\ 
    The bound \( \lVert w\rVert_{L^2(\mu)}\le R \) implies \( \lVert A_\epsilon w\rVert_{C^1(M)}\le cR\sum_i\lVert e_i^\epsilon\rVert_{C^1(M)} \). 
    Now, a bounded subset of the finite-dimensional space \( W_\epsilon \) is compact in the space \( C^1(M;TM) \), since all norms on \( W_\epsilon \) are equivalent, which concludes.

    \medskip
    \noindent Item~\ref{item:us.dominated}.\\
    It suffices to take \( G\coloneqq cR\sum_i\left(\lvert e_i\rvert+\lvert e_i^\epsilon\rvert\right) \), which lies in \( L^2(\mu) \) and dominates both \( \lvert w\rvert \) and \( \lvert A_\epsilon w\rvert \) pointwise.
\end{proof}

Take \( \mu\in\Prob(M) \) and \( v,w\in L^2(\mu;TM) \), and compare the curves \( t\mapsto(\exp tv)_\ast\mu \) and \( t\mapsto(\exp tw)_\ast\mu \).
Lemma~\ref{lem:exp.upper} bounds their separation from above by \( t\lVert v-w\rVert_{L^2(\mu)}+o(t) \), and Lemma~\ref{lem:duality.lower} bounds it from below by \( t\lVert P_\mu(v-w)\rVert_{L^2(\mu)}+o(t) \).
The two bounds disagree whenever \( v-w \) has a non-zero component in \( L^2_0(\mu)^\perp \), that is, whenever the two fields differ by a solenoidal one.
The theorem below shows that the lower bound is the correct one: adding a solenoidal field to \( v \) changes the curve \( t\mapsto(\exp tv)_\ast\mu \) by \( o(t) \), and therefore does not change it at first order.

\begin{thm}[Solenoidal directions are metrically invisible]
\label{thm:solenoidal}
    Let \( \mu=\rho\,\vol \) with \( \rho\in C^1(M) \) and \( \rho>0 \), let \( v \in L^2(\mu;TM) \), and let \( z\in L^2(\mu;TM) \) satisfy \( \nabla\!\cdot(z\mu)=0 \).
    Then
    \[
        \dw\left((\exp t(v+z))_\ast\mu,\ (\exp tv)_\ast\mu\right)=o(t),\qquad t\searrow0 .
    \]
    The remainder is uniform for \( (v,z) \) in bounded subsets of \( V\times(Z\cap L^2_0(\mu)^\perp) \), for any finite-dimensional subspaces \( V,Z\subset L^2(\mu;TM) \).
\end{thm}

\begin{proof}
    Write \( \Delta(t;v,z)\coloneqq\dw\left((\exp t(v+z))_\ast\mu,(\exp tv)_\ast\mu\right) \) for the quantity of interest.

    First we show that it suffices to prove the assertion for \( v \) and \( z \) of class \( C^1 \).
    By Lemma~\ref{lem:smooth.dense} choose \( v_n\in C^\infty(M;TM) \) with \( v_n\to v \) in \( L^2(\mu;TM) \), and by Lemma~\ref{lem:regular.solenoidal} choose \( z_n\in C^1(M;TM) \) with \( \nabla\!\cdot(z_n\mu)=0 \) and \( z_n\to z \) in \( L^2(\mu;TM) \).

    Inserting \( (\exp t(v_n+z_n))_\ast\mu \) and \( (\exp tv_n)_\ast\mu \) between the two measures and applying Lemma~\ref{lem:exp.upper} twice, once with \( (u,u')=(v+z,\,v_n+z_n) \) and once with \( (u,u')=(v_n,v) \), we obtain
    \[
    \begin{split}
        \limsup_{t\searrow0}\frac{\Delta(t;v,z)}{t}\ &\le\ 2\lVert v-v_n\rVert_{L^2(\mu)}+\lVert z-z_n\rVert_{L^2(\mu)} \\
                                                     &+\limsup_{t\searrow0}\frac{\Delta(t;v_n,z_n)}{t}.
    \end{split}
    \]
    It suffices therefore to prove the assertion for \( v \) and \( z \) of class \( C^1 \).

    Now, in class \(C^1\) Proposition~\ref{prp:linear.response} applies to both fields and gives, uniformly on \( M \), and Equation~\eqref{eq:divergence.equivariance} gives
    \[
    \begin{split}
        \rho_t[v+z]-\rho_t[v]&=-t\,\Div\left(\rho(v+z)\right)+t\,\Div(\rho v)+o(t) \\ 
                             &=-t\,\Div(\rho z)+o(t)=o(t).
    \end{split}
    \]
    Both densities are continuous and bounded below by \( \tfrac12\min\rho \) for sufficiently small \( t \), so Lemma~\ref{lem:density.comparison} applies, yielding
    \begin{equation}\label{eq:estimate.remainder}
        \Delta(t;v,z)\ \le\ 2\left(\tfrac12\min\rho\cdot\lambda_1\right)^{-1/2}\big\lVert\rho_t[v+z]-\rho_t[v]\big\rVert_{L^2(\vol)}\ =\ o(t) .
    \end{equation}

    For the uniformity statement, let \( V,Z\subset L^2(\mu;TM) \) be finite dimensional, put \( Z_0\coloneqq Z\cap L^2_0(\mu)^\perp \), and fix \( R>0 \) and \( \eta>0 \).
    Apply Lemma~\ref{lem:uniform.smoothing} to \( V \) and to \( Z_0 \), obtaining constants \( c_V,c_{Z_0} \) and, for each \( \epsilon>0 \), linear maps \( A_\epsilon:V\to C^\infty(M;TM) \) and \( B_\epsilon:Z_0\to C^1(M;TM)\cap L^2_0(\mu)^\perp \).
    Write \( v^\epsilon\coloneqq A_\epsilon v \) and \( z^\epsilon\coloneqq B_\epsilon z \), and let \( \mathcal K_\epsilon \) be the compact subset of \( C^1(M;TM) \) produced by Lemma~\ref{lem:uniform.smoothing}\ref{item:us.compact} for the subspace \( A_\epsilon V+B_\epsilon Z_0 \) and the radius \( 2R \), so that \( v^\epsilon \), \( z^\epsilon \) and \( v^\epsilon+z^\epsilon \) all lie in \( \mathcal K_\epsilon \) whenever \( \lVert v\rVert_{L^2(\mu)},\lVert z\rVert_{L^2(\mu)}\le R \).

    Inserting \( (\exp t(v^\epsilon+z^\epsilon))_\ast\mu \) and \( (\exp tv^\epsilon)_\ast\mu \) between the two measures, the triangle inequality gives
    \[
    \begin{split}
        \frac{\Delta(t;v,z)}{t}\ &\le\ \frac{\dw\left((\exp t(v+z))_\ast\mu,(\exp t(v^\epsilon+z^\epsilon))_\ast\mu\right)}{t} \\
                                 &+\frac{\Delta(t;v^\epsilon,z^\epsilon)}{t}+\frac{\dw\left((\exp tv^\epsilon)_\ast\mu,(\exp tv)_\ast\mu\right)}{t} .
    \end{split}
    \]
    The four fields occurring in the first and third terms are dominated by a single \( G\in L^2(\mu) \), by Lemma~\ref{lem:uniform.smoothing}\ref{item:us.dominated}, so \eqref{eq:exp.upper.unif} applies to that family.
    Together with Lemma~\ref{lem:uniform.smoothing}\ref{item:us.approx} it produces \( t_2>0 \), depending on \( \epsilon \) but not on \( (v,z) \), such that those two terms are bounded by \( (2c_V+c_{Z_0})R\epsilon+\eta/3 \) for \( t<t_2 \).

    The middle term is the one estimated in \eqref{eq:estimate.remainder}, and what has to be made uniform there is the remainder of Proposition~\ref{prp:linear.response} over \( \mathcal K_\epsilon \), namely
    \[
        \omega_\epsilon(t)\coloneqq\sup_{u\in\mathcal K_\epsilon}\big\lVert\,\rho_t[u]-\rho+t\,\Div(\rho u)\,\big\rVert_{C^0(M)} ,
    \]
    which is \( o(t) \) because \( \mathcal K_\epsilon \) is a compact subset of \( C^1(M;TM) \).
    Subtracting the expansions of \( \rho_t[v^\epsilon+z^\epsilon] \) and of \( \rho_t[v^\epsilon] \), and using \( \Div(\rho z^\epsilon)=0 \), we get \( \lVert\rho_t[v^\epsilon+z^\epsilon]-\rho_t[v^\epsilon]\rVert_{C^0(M)}\le2\omega_\epsilon(t) \).
    The same expansion, together with the bound on \( \lVert\Div(\rho u)\rVert_{C^0(M)} \) for \( u\in\mathcal K_\epsilon \), shows that both densities exceed \( \tfrac12\min\rho \) once \( t<t_1 \), with \( t_1 \) depending only on \( \rho \) and on \( \mathcal K_\epsilon \).
    Lemma~\ref{lem:density.comparison} then bounds the middle term by a fixed multiple of \( \omega_\epsilon(t)/t \), independently of \( (v,z) \).

    It remains to choose the parameters.
    Fix \( \epsilon \) so small that \( (2c_V+c_{Z_0})R\epsilon\le\eta/3 \), and then \( t_0\le\min(t_1,t_2) \) so small that the bound on the middle term is at most \( \eta/3 \) for \( t<t_0 \).
    For such \( t \) we get \( \Delta(t;v,z)\le\eta\,t \) for every admissible pair, as we wanted.
\end{proof}

The cancellation behind the equality \( \rho_t[v+z]-\rho_t[v]=o(t) \) has geometric meaning worth recording.
Let \( z \) be of class \( C^1 \) with \( \nabla\!\cdot(z\mu)=0 \), and let \(\Psi^t\) be its flow.
Then \( \Psi^t_\ast\mu=\mu \) for every \( t \), so that deforming \( \mu \) along \( z \) redistributes its mass inside itself rather than displacing it.
One is stirring \( \mu \), not moving it.

For the invariance, fix \( T \) and \( \psi\in C^\infty(M) \) and put \( u(t,x)\coloneqq\psi( \Psi^{T-t}(x)) \).
The flow property \(\Psi^{s+t}=\Psi^s\circ\Psi^t\) makes it so that \(u(t,\Psi^s(x))=u(t-s,x)\).
Differentiation at \(s=0\) gives \(\langle\nabla u(t,\cdot),z\rangle=-\partial_tu(t,\cdot) \), so that \( G(t)\coloneqq\int u(t,\cdot)\,d\mu \) satisfies \( G'(t)=\int\partial_tu\,d\mu=-\int\langle\nabla u(t,\cdot),z\rangle\,d\mu=0 \).
Therefore, \( \int\psi\circ\Psi^T\,d\mu=G(0)=G(T)=\int\psi\,d\mu \).
The function \(u(t,\cdot)\) is only of class \(C^1\), whereas \(\nabla\!\cdot(z\mu)=0\) is stated against \(C^\infty\) test functions; Remark~\ref{rmk:C1.test} covers the difference.

\begin{prp}
\label{prp:solenoidal.dual}
    Let \( \mu=\rho\,\vol \) with \( \rho\in C^1(M) \) and \( \rho>0 \), and let \( v,w\in L^2(\mu;TM) \) be arbitrary.
    Then
    \[
        \lim_{t\searrow0}\frac{\dw\left((\exp tv)_\ast\mu,(\exp tw)_\ast\mu\right)}{t}=\big\lVert P_\mu(v-w)\big\rVert_{L^2(\mu)} .
    \]
    In particular \( \dw\left(\mu,(\exp tv)_\ast\mu\right)=t\lVert P_\mu v\rVert_{L^2(\mu)}+o(t) \) for all \( v\in L^2(\mu;TM) \), and the two deformations agree to first order exactly when \( v-w\in L^2_0(\mu)^\perp \).
\end{prp}

\begin{proof}
    Put \( z\coloneqq(\Id-P_\mu)(v-w) \) and \( g\coloneqq w+P_\mu(v-w) \), so that \( v=g+z \) and \( g-w=P_\mu(v-w) \).
    Theorem~\ref{thm:solenoidal} gives \( \dw\left((\exp tv)_\ast\mu,(\exp tg)_\ast\mu\right)=o(t) \), and Lemma~\ref{lem:exp.upper} gives
    \[
        \limsup_{t\searrow0}\frac{\dw\left((\exp tg)_\ast\mu,(\exp tw)_\ast\mu\right)}{t}\ \le\ \lVert g-w\rVert_{L^2(\mu)}=\big\lVert P_\mu(v-w)\big\rVert_{L^2(\mu)} ,
    \]
    which is the upper bound.

    For the lower bound apply Lemma~\ref{lem:duality.lower} with \( \alpha_t\coloneqq(\exp tw)_\ast\mu \), \( \beta_t\coloneqq(\exp tv)_\ast\mu \) and \( u\coloneqq v-w \).
    The weak\( ^\ast \) convergence \( \alpha_t\rightharpoonup\mu \) holds because \( \dw(\mu,\alpha_t)\le t\lVert w\rVert_{L^2(\mu)} \), by the coupling \( (\Id,\exp tw)_\ast\mu \).
    For the remaining hypothesis fix \( \psi\in C^\infty(M) \) and apply \eqref{eq:taylor} at each point, with \( \xi=tv(x) \) and with \( \xi=tw(x) \); integrating against \( \mu \) and subtracting gives
    \[
        \int_M\psi\,d\beta_t-\int_M\psi\,d\alpha_t=t\,\big\langle\nabla\psi,\,v-w\big\rangle_{L^2(\mu)}+O(t^2),
    \]
    the remainder being bounded by \( \tfrac{H}{2}t^2\left(\lVert v\rVert^2_{L^2(\mu)}+\lVert w\rVert^2_{L^2(\mu)}\right)<\infty \).
    Applying Lemma~\ref{lem:duality.lower} then gives \( \liminf_tt^{-1}\dw(\beta_t,\alpha_t)\ge\lVert P_\mu(v-w)\rVert_{L^2(\mu)} \).
\end{proof}

Proposition~\ref{prp:solenoidal.dual} extends Corollary~\ref{cor:metric.derivative} from \( v\in L^2_0(\mu) \) to arbitrary \( v\in L^2(\mu;TM) \), at the cost of a regularity hypothesis on \( \mu \) which Corollary~\ref{cor:metric.derivative} does not need.
The two agree on \( L^2_0(\mu) \), where \( P_\mu v=v \).
For \( \dim M=1 \) the solenoidal fields are exactly the multiples of \( \rho^{-1} \), and \( P_\mu \) is the centering operator of \cite[\S5.1]{kloeckner_optimal_2013}.
The statement then degenerates into \cite[Proposition 5.2]{kloeckner_optimal_2013}, which says that a \( C^1 \) field and its centred version deform \( \rho\Leb \) in the same way to first order.
The extreme case is the constant field on \( S^1 \), for which \( (\exp t\xi)_\ast\Leb=\Leb \) for every \( t \): one rotates all the mass, whereas leaving it in place would be more efficient \cite[\S3.2]{kloeckner_optimal_2013}.

For \( \dim M\ge2 \) the space of solenoidal fields is infinite dimensional, which is what makes \( P_\mu \) a genuine projection rather than the subtraction of a constant.
Its explicit Fourier form on \( \T^d \) is the Helmholtz--Hodge projection of Proposition~\ref{prp:fourier.model}.

\subsection{The derivative at an invariant measure}
\label{sec:derivative.acim}

Proposition~\ref{prp:functoriality.pushforward} and Proposition~\ref{prp:image.velocity} already describe the image of a curve under \( \phi_\ast \): if \( (\mu_t) \) is absolutely continuous with velocity field \( (v_t) \), then \( \nu_t=\phi_\ast\mu_t \) is absolutely continuous with velocity field \( P_{\nu_t}\Transf_{\mu_t}v_t \) for almost every \( t \).
That statement is about velocity fields, and it holds only for almost every \( t \).
What we have on the metric side is likewise incomplete: Proposition~\ref{prp:pushforward.rate} places \( \nu_t \) within distance \( O(t) \) of \( \mu_t \), and Proposition~\ref{prp:solenoidal.dual} computes \( \lim_{t\searrow0}t^{-1}\dw \) between two exponential deformations of one and the same measure, but neither identifies the curve \( (\nu_t) \) itself.
The theorem below does.
It fixes an invariant \( \mu_0 \), works at the single time \( t=0 \), and identifies \( \phi_\ast\left((\exp tv)_\ast\mu_0\right) \), up to \( o(t) \) in \( \dw \), with an exponential deformation of \( \mu_0 \) along an explicit field.

Our strategy has four steps.
First, one deforms \( \mu_0=\rho\vol \) along \( v \) and expands the density of the deformed measure by Proposition~\ref{prp:linear.response}.
One then applies \( \phi_\ast \), which acts on densities as a transfer operator bounded on \( C^0(M) \), and identifies the resulting first-order term by Lemma~\ref{lem:transfer.intertwines}, which says that the pushforward of \( \nabla\!\cdot(v\mu_0) \) is \( \nabla\!\cdot(\Transf_0v\,\mu_0) \).
The two densities then differ by \( o(t) \) uniformly on \( M \), and Lemma~\ref{lem:density.comparison} converts that into an \( o(t) \) bound on \( \dw \).
Thirdly, the field so obtained is \( \Transf_0v \), which need not belong to \( L^2_0(\mu_0) \), and Theorem~\ref{thm:solenoidal} removes its solenoidal part at no metric cost, leaving \( \D_\phi v=P_{\mu_0}\Transf_0v \).
Lastly, the first three steps require \( v \) to be of class \( C^1 \), and Lemma~\ref{lem:exp.upper} extends the conclusion to every \( v\in L^2_0(\mu_0) \).

The projection is therefore not imposed by hand.
It is what remains once Theorem~\ref{thm:solenoidal} has been applied, it is the same projection that appears in Proposition~\ref{prp:image.velocity}, and it is the same projection that appeared in the lower bound of Lemma~\ref{lem:duality.lower}.

\begin{thm}[G\^ateaux derivative at an invariant measure]
\label{thm:derivative.acim}
    Let \( \phi:M\to M \) be a \( C^2 \) covering map and let \( \mu_0=\rho\,\vol \) be \( \phi \)-invariant with \( \rho\in C^1(M) \) and \( \rho>0 \).
    Let \( \Transf_0\coloneqq\Transf_{\mu_0} \) and \( \Koop_0\coloneqq\Koop_{\mu_0} \) be the operators of Lemma~\ref{lem:transfer.operator}, which act on \( L^2(\mu_0;TM) \) because \( \phi_\ast\mu_0=\mu_0 \), and set
    \[
        \D_\phi\coloneqq P_{\mu_0}\Transf_0:L^2_0(\mu_0)\lto L^2_0(\mu_0) .
    \]
    Then, for every \( v\in L^2_0(\mu_0) \),
    \[
        \dw\left(\phi_\ast\left((\exp tv)_\ast\mu_0\right),\ \left(\exp t\,\D_\phi v\right)_\ast\mu_0\right)=o(t),\qquad t\searrow0 ,
    \]
    the remainder being uniform on bounded subsets of finite-dimensional subspaces of \( L^2_0(\mu_0) \).
\end{thm}

\begin{proof}
    Write \( \PF_\phi  \) for the Perron-Frobenius operator of \( \phi \) acting on densities with respect to \( \vol \), so that \( \phi_\ast(f\vol)=(\PF_\phi f)\vol \).
    Explicitly
    \[
        \PF_\phi f(y)=\sum_{x\in\phi^{-1}(y)}\frac{f(x)}{\lvert\det D_x\phi\rvert} .
    \]
    Since \( \phi \) is a \( C^2 \) covering map of a closed manifold it has finitely many inverse branches, each of class \( C^2 \), so \( \PF_\phi  \) maps \( C^0(M) \) into itself.
    Since it is positive, it satisfies \( \lvert\PF_\phi f\rvert\le\PF_\phi \lvert f\rvert\le\lVert f\rVert_\infty\PF_\phi \mathbf 1 \), so that its norm is attained at the constant function \( \mathbf 1 \) (compare \cite[Chapter~2]{baladi_positive_2000}):
    \[
        \big\lVert\PF_\phi \big\rVert_{C^0\to C^0}=\big\lVert\PF_\phi \mathbf 1\big\rVert_\infty\eqqcolon\kappa<\infty .
    \]
    The invariance \( \phi_\ast\mu_0=\mu_0 \) reads \( \PF_\phi \rho=\rho \).

	Lemma~\ref{lem:transfer.intertwines}, applied at \( \mu=\mu_0 \) with \( \nu=\phi_\ast\mu_0=\mu_0 \), gives \eqref{eq:transfer.commutes} in the form \( \nabla\!\cdot(\Transf_0v\ \mu_0)=\phi_\ast(\nabla\!\cdot(v\,\mu_0)) \).
    This is the precise sense in which \( \Transf_0 \) is the derivative of \( \phi_\ast \).

    We start in class \(C^1\).
    Let \( v\in C^1(M;TM) \) and \( \tau_1,\dots,\tau_q \) be the inverse branches of \( \phi \), each of class \( C^2 \).
    Disintegrating \( \mu_0 \) over the fibres of \( \phi \) and using \( \phi_\ast\mu_0=\mu_0 \), the operator \eqref{eq:transfer.disintegration} is given explicitly by
    \begin{equation}\label{eq:transfer.branches}
        \Transf_0v(y)=\frac{1}{\rho(y)}\sum_{i=1}^q\rho\left(\tau_i(y)\right)\,\big\lvert\det D_y\tau_i\big\rvert\ D_{\tau_i(y)}\phi\ v\left(\tau_i(y)\right) .
    \end{equation}
    Every factor on the right is of class \( C^1 \), because \( \tau_i \) is \( C^2 \), \( D\phi \) is \( C^1 \), \( \rho \) is \( C^1 \) and \( v \) is \( C^1 \); and \( \rho \) does not vanish, so the quotient is again of class \( C^1 \).
    Hence \( \Transf_0v\in C^1(M;TM) \), and Proposition~\ref{prp:linear.response} applies to \( v \) and to \( \Transf_0v \) alike.

    When applied to \(v\), it gives \((\exp(tv))_\ast\mu_0 = \left(\rho-t\,\Div(\rho v) +r_t\right)\vol\), where \(\lVert r_t\rVert_{C^0}=o(t)\).
    Applying \(\phi_\ast\), equivalently \(\PF_\phi \), we obtain \(\PF_\phi \rho=\rho\) and \(\lVert\PF_\phi r_t\rVert_{C^0} \le\kappa\lVert r_t\rVert_{C^0}=o(t)\).
    Therefore,
    \[
        \phi_\ast\left((\exp(tv))_\ast\mu_0\right)=\left(\rho-t\,\PF_\phi \left(\Div(\rho v)\right)\right)\vol+o(t)
    \]
    uniformly on \(M\).
   
    By applying Equation~\eqref{eq:divergence.equivariance}, we may now read \eqref{eq:transfer.commutes}  as an identity of continuous functions.
    Indeed, both \( \rho v \) and \( \rho\,\Transf_0v \) are of class \( C^1 \) so the left-hand side of \eqref{eq:transfer.commutes} identifies with the measure \( \Div(\rho\,\Transf_0v)\,\vol \).
    On the other hand, on the right-hand side the argument of the pushforward identifies with \( \Div(\rho v)\,\vol \).
	Since \(\Div(\rho v)\in C^0(M)\), the pushforward of \(\Div(\rho v)\,\vol\) is \(\PF_\phi(\Div(\rho v))\,\vol\) by the defining property of \(\PF_\phi\).
    Finally, two continuous densities defining the same measure agree, so
    \begin{equation}\label{eq:transfer.densities}
        \PF_\phi \left(\Div(\rho v)\right)=\Div\left(\rho\,\Transf_0v\right) .
    \end{equation}
    Proposition~\ref{prp:linear.response} applied to \( \Transf_0v \) gives
    \[
        (\exp t\Transf_0v)_\ast\mu_0=\left(\rho-t\,\Div(\rho\,\Transf_0v)\right)\vol+o(t)
    \]
    uniformly, so by \eqref{eq:transfer.densities} the two measures we are comparing have continuous densities differing by \( o(t) \) in \( C^0(M) \), both bounded below by \( \tfrac12\min\rho \) for \( t \) small.
    We have \( \lVert\cdot\rVert_{L^2(\vol)}\le\sqrt{\vol(M)}\lVert\cdot\rVert_{C^0(M)} = \lVert\cdot\rVert_{C^0(M)} \), as the volume is normalised.
    This, together with Lemma~\ref{lem:density.comparison}, yields
    \begin{equation}\label{eq:derivative.step2}
        \dw\left(\phi_\ast\left((\exp tv)_\ast\mu_0\right),\ (\exp t\,\Transf_0v)_\ast\mu_0\right)=o(t).
    \end{equation}

    The third step is the projection onto \( L^2_0(\mu_0) \).
    Write \( \Transf_0v=\D_\phi v+z \) with \( z\coloneqq(\Id-P_{\mu_0})\Transf_0v\in L^2_0(\mu_0)^\perp \).
    The measure \( \mu_0 \) satisfies the hypotheses of Theorem~\ref{thm:solenoidal}, which applies with direction \( \D_\phi v \) and gives
    \[
        \dw\left((\exp t\Transf_0v)_\ast\mu_0,\ (\exp t\D_\phi v)_\ast\mu_0\right)=o(t) .
    \]
    Together with \eqref{eq:derivative.step2} this proves the theorem for \( v\in C^1 \).

    Finally, we move to general \( v\in L^2_0(\mu_0) \).
    By \eqref{eq:L20} choose \( \psi_n\in C^\infty(M) \) such that \( v_n\coloneqq\nabla\psi_n \) converges to \( v \) in \( L^2(\mu_0;TM) \). 
    Each \( v_n \) is smooth and lies in \( L^2_0(\mu_0) \), so the first three steps apply to it.
    The map \( \phi_\ast \) is \( \Lip(\phi) \)-Lipschitz on \( (\Prob(M),\dw) \), as shown in the proof of Proposition~\ref{prp:pushforward.rate}, so Lemma~\ref{lem:exp.upper} gives
    \[
        \limsup_{t\searrow0}\frac1t\,\dw\left(\phi_\ast\left((\exp tv)_\ast\mu_0\right),\ \phi_\ast\left((\exp tv_n)_\ast\mu_0\right)\right)\ \le\ \Lip(\phi)\,\lVert v-v_n\rVert_{L^2(\mu_0)} .
    \]
    The same lemma gives 
    \[ 
        \limsup_{t\searrow0}\frac1t\,\dw\left((\exp t\D_\phi v)_\ast\mu_0,(\exp t\D_\phi v_n)_\ast\mu_0\right)\le\lVert\D_\phi\rVert_{\mathrm{op}}\,\lVert v-v_n\rVert,
    \] 
    and \( \lVert\D_\phi\rVert_{\mathrm{op}}\le\lVert\Transf_0\rVert_{\mathrm{op}}\le\Lip(\phi) \) by Proposition~\ref{prp:functoriality.pushforward}\ref{item:w.bound}, since the projection \( P_{\mu_0} \) has norm at most one.
    Letting \( n\to\infty \) concludes.

    As for the uniformity, let \( W\subset L^2_0(\mu_0) \) be finite dimensional, let \( R>0 \), and let \( A_\epsilon:W\to C^\infty(M;TM)\cap L^2_0(\mu_0) \) be the map given by Lemma~\ref{lem:uniform.smoothing}.
    The operators \( \Transf_0 \), \( \D_\phi \) and \( \Id-P_{\mu_0} \) are linear and bounded, so as \( v \) runs over the ball of radius \( R \) in \( W \) the fields \( A_\epsilon v \), \( \Transf_0A_\epsilon v \), \( \D_\phi A_\epsilon v \) and \( (\Id-P_{\mu_0})\Transf_0A_\epsilon v \) range in bounded subsets of the finite-dimensional spaces \( A_\epsilon W \), \( \Transf_0A_\epsilon W \), \( \D_\phi A_\epsilon W \) and \( (\Id-P_{\mu_0})\Transf_0A_\epsilon W \) respectively, all of which consist of fields of class \( C^1 \) by \eqref{eq:transfer.branches}.
    The uniformity clauses of Proposition~\ref{prp:linear.response} and of Theorem~\ref{thm:solenoidal} therefore cover the first three steps, the latter applied with \( V\coloneqq\D_\phi A_\epsilon W \) and \( Z\coloneqq(\Id-P_{\mu_0})\Transf_0A_\epsilon W\subset L^2_0(\mu_0)^\perp \).
    In the last step Lemma~\ref{lem:exp.upper} is replaced by \eqref{eq:exp.upper.unif}, applied to the four fields \( v \), \( A_\epsilon v \), \( \D_\phi v \) and \( \D_\phi A_\epsilon v \), which are dominated by a single element of \( L^2(\mu_0) \) by Lemma~\ref{lem:uniform.smoothing}\ref{item:us.dominated} and the boundedness of \( \D_\phi \).
    Choosing first \( \epsilon \), by Lemma~\ref{lem:uniform.smoothing}\ref{item:us.approx}, and then \( t \), exactly as in the proof of Theorem~\ref{thm:solenoidal}, yields the conclusion.
\end{proof}

\begin{rmk}[Fr\'echet derivatives]
    \label{rmk:frechet.derivative}
    The remainder in Theorem~\ref{thm:derivative.acim} is uniform on bounded subsets of finite-dimensional subspaces of \( L^2_0(\mu_0) \) and not on the unit ball, so the statement is indeed a G\^ateaux derivative, not a Fr\'echet one, and this cannot be improved.

    Indeed, already for the circle map \( \phi: x \mapsto 2x \), \cite[Proposition 4.3]{kloeckner_optimal_2013} produces a constant \( c \) such that, for every \( \epsilon>0 \), there is a non-zero \( v\in L^2_0(\vol) \) satisfying \( \lVert v\rVert_{L^2}\le\epsilon \),  \( \D_\phi v=0 \) and \( \dw\left((\phi_\ast(\exp tv)_\ast\vol),\vol\right)\ge c\,\epsilon \).
    The remainder of the derivative at such a \( v \) is therefore at least \( c\lVert v\rVert_{L^2} \), and not \( o(\lVert v\rVert_{L^2}) \).

    This example is easily transported to higher dimensions (cf. Section~\ref{sec:conclusion}).
\end{rmk}

\begin{lem}
\label{lem:koopman.invariance}
    Let \( \phi\in C^1(M,M) \) and let \( \mu\in\Prob(M) \) satisfy \( \phi_\ast\mu=\mu \).
    Then \( \Koop_\mu\left(L^2_0(\mu)\right)\subseteq L^2_0(\mu) \).
\end{lem}

\begin{proof}
    Let \( \psi\in C^\infty(M) \).
    Then \( \Koop_\mu\nabla\psi=\nabla(\psi\circ\phi) \), and \( \psi\circ\phi\in C^1(M) \) because \( \phi \) is of class \( C^1 \).
    Since \( M \) is compact, \( C^\infty(M) \) is dense in \( C^1(M) \), so there are \( \chi_n\in C^\infty(M) \) with \( \nabla\chi_n\to\nabla(\psi\circ\phi) \) uniformly on \( M \), hence in \( L^2(\mu;TM) \).
    Therefore \( \Koop_\mu\nabla\psi\in L^2_0(\mu) \).
    The operator \( \Koop_\mu \) is bounded by Lemma~\ref{lem:transfer.operator} and the gradients of smooth functions are dense in \( L^2_0(\mu) \) by \eqref{eq:L20}, so the inclusion holds on all of \( L^2_0(\mu) \).
\end{proof}

The operator \( \D_\phi \) was defined by projecting the first-order transport operator onto the tangent space \( L^2_0(\mu_0) \).
By Lemma~\ref{lem:koopman.invariance} the Koopman operator preserves that space, and the projection is not merely a device for recovering a tangent vector: the restriction \( \Koop_0|_{L^2_0(\mu_0)} \) is adjoint to \( \D_\phi \).

\begin{prp}
\label{prp:derivative.functoriality}
    Let \( \phi:M\to M \) be a \( C^2 \) covering map and let \( \mu_0=\rho\vol \) be \( \phi \)-invariant, with \( \rho\in C^1(M) \) and \( \rho>0 \).
    Then \(\D_\phi=\left(\Koop_0|_{L^2_0(\mu_0)}\right)^\ast \).
    Consequently, if \( \psi:M\to M \) is another \( C^2 \) covering map preserving \( \mu_0 \), then
    \[
        \D_{\phi\circ\psi}=\D_\phi\D_\psi .
    \]
    Moreover,
    \[
        E_\phi \coloneqq \ker(\D_\phi-\Id)=\big\{v\in L^2_0(\mu_0):\phi_\ast\iota_{\mu_0}(v)=\iota_{\mu_0}(v)\big\}.
    \]
\end{prp}

The identity \( \D_{\phi\circ\psi}=\D_\phi\D_\psi \) is the chain rule of \cite[Corollary 3.13]{lessel_differentiable_2020}, read in the present setting. 
The adjoint identity and the description of \( E_\phi \) do not appear there.

\begin{proof}
    Let \( v,w\in L^2_0(\mu_0) \).
    Since \( P_{\mu_0} \) is the orthogonal projection onto \( L^2_0(\mu_0) \), Lemma~\ref{lem:transfer.operator} gives
    \[
        \big\langle\D_\phi v,w\big\rangle_{L^2(\mu_0)}
        =\big\langle P_{\mu_0}\Transf_0v,w\big\rangle_{L^2(\mu_0)}
        =\big\langle\Transf_0v,w\big\rangle_{L^2(\mu_0)}
        =\big\langle v,\Koop_0w\big\rangle_{L^2(\mu_0)} .
    \]
    Thus \( \D_\phi=\left(\Koop_0|_{L^2_0(\mu_0)}\right)^\ast \).

    Suppose \( \phi \) and \( \psi \) both preserve \( \mu_0 \), and write \( \Koop_0^\phi \) for the Koopman operator of \( \phi \) at \( \mu_0 \), displaying the map in the superscript and keeping the measure in the subscript.
    The chain rule \( D_x(\phi\circ\psi)=D_{\psi(x)}\phi\circ D_x\psi \) gives \( \Koop_0^{\phi\circ\psi}=\Koop_0^{\psi}\Koop_0^{\phi} \) on \( L^2_0(\mu_0) \).
    Taking adjoints yields \( \D_{\phi\circ\psi}=\D_\phi\D_\psi \).

    Finally, Lemma~\ref{lem:transfer.intertwines} gives
    \[
        \iota_{\mu_0}(\Transf_0v)=\phi_\ast\iota_{\mu_0}(v).
    \]
    Since \( \Transf_0v-\D_\phi v\in L^2_0(\mu_0)^\perp \), its weighted divergence vanishes, so \( \iota_{\mu_0}(\Transf_0v)=\iota_{\mu_0}(\D_\phi v) \).
    Hence
    \[
        \iota_{\mu_0}(\D_\phi v)=\phi_\ast\iota_{\mu_0}(v).
    \]
    The injectivity of \( \iota_{\mu_0} \) on \( L^2_0(\mu_0) \) now shows that \( \D_\phi v=v \) if and only if \( \phi_\ast\iota_{\mu_0}(v)=\iota_{\mu_0}(v) \).
\end{proof}

The last identity identifies the fixed space of \( \D_\phi \) with the tangent directions at \( \mu_0 \) whose associated first-order variations are invariant under \( \phi_\ast \).
Thus, for two maps \( \phi \) and \( \psi \) preserving \( \mu_0 \), the intersection \( E_\phi\cap E_\psi \) consists precisely of the tangent directions simultaneously invariant under both induced pushforwards, and by Proposition~\ref{prp:acim.derivative}\ref{item:acim.sharp} these are exactly the directions along which \( \mu_0 \) can be deformed without losing invariance at first order.

In this sense, the question of whether \( E_\phi\cap E_\psi \) is trivial is a first-order version of the rigidity phenomenon discussed in Section~\ref{sec:introduction}, and it is the form in which Furstenberg's conjecture \cite{furstenberg_disjointness_1967} and its higher-dimensional analogues can be tested on the tangent space.
On the circle this is the point of view of \cite[\S1.4]{kloeckner_optimal_2013_preprint}.

In the proposition below, Part~\ref{item:acim.lower} holds for any \( C^1 \) map and any invariant measure, and does not depend on the present section at all. Part~\ref{item:acim.sharp} is where Theorem~\ref{thm:derivative.acim} enters, and it turns the inequality into an asymptotic equality.

\begin{prp}[Derivatives at invariant measures]
\label{prp:acim.derivative}
    Let \( \phi\in C^1(M,M) \), let \( \mu_0\in\Prob(M) \) be \( \phi \)-invariant and let \( v_0\in L^2_0(\mu_0) \).
    Set \( \mu_t\coloneqq(\exp tv_0)_\ast\mu_0 \), \( \nu_t\coloneqq\phi_\ast\mu_t \) and \( v_0^\nu\coloneqq P_{\mu_0}\Transf_0v_0 \).
    \begin{enumerate}
        \item\label{item:acim.lower}
            One always has
            \( \displaystyle\liminf_{t\searrow0}\frac{\dw(\mu_t,\nu_t)}{t}\ge\lVert v_0-v_0^\nu\rVert_{L^2(\mu_0)} \).
        \item\label{item:acim.sharp}
            If moreover \( \phi \) is a \( C^2 \) covering map and \( \mu_0=\rho\vol \) with \( \rho\in C^1(M) \), \( \rho>0 \), then
            \[
                \dw(\mu_t,\nu_t)=t\,\lVert v_0-v_0^\nu\rVert_{L^2(\mu_0)}+o(t) ,
            \]
            and in particular \( \dw(\mu_t,\nu_t)=o(t) \) if and only if \( v_0 \) is a fixed vector of \( \D_\phi=P_{\mu_0}\Transf_0 \), that is, an eigenvector with eigenvalue \( 1 \).
    \end{enumerate}
\end{prp}

\begin{proof}
    \ref{item:acim.lower}
    We apply Lemma~\ref{lem:duality.lower} with \( \mu\coloneqq\mu_0 \), \( \alpha_t\coloneqq\mu_t \), \( \beta_t\coloneqq\nu_t \) and \( u\coloneqq v_0^\nu-v_0 \).
    The weak\( ^\ast \) convergence holds because \( \dw(\mu_0,\mu_t)\le t\lVert v_0\rVert_{L^2(\mu_0)} \).

    In order to verify the other hypothesis fix \( \psi\in C^\infty(M) \).
    On the one hand, by \eqref{eq:taylor},
    \begin{equation}\label{eq:alpha.exp}
        \int\psi\,d\mu_t=\int\psi\,d\mu_0+t\,\langle\nabla\psi,v_0\rangle_{L^2(\mu_0)}+O(t^2),
    \end{equation}
    the remainder being bounded by \( \tfrac{H}{2}t^2\lVert v_0\rVert^2_{L^2(\mu_0)}<\infty \).
    On the other hand \( \psi\circ\phi\in C^1(M) \), with
    \[
        \nabla(\psi\circ\phi)(x)=(D_x\phi)^{\!*}\nabla\psi\left(\phi(x)\right)=\left(\Koop_0\nabla\psi\right)(x),
    \]
    and \( \int\psi\,d\nu_t=\int\psi\circ\phi\,d\mu_t=\int_M(\psi\circ\phi)\left(\exp_x(tv_0(x))\right)\,d\mu_0(x) \).
    The difference quotients \( t^{-1}\big[(\psi\circ\phi)(\exp_x(tv_0(x)))-(\psi\circ\phi)(x)\big] \) converge pointwise to \( \langle\nabla(\psi\circ\phi)(x),v_0(x)\rangle \), and they are dominated by \( \Lip(\psi\circ\phi)\lvert v_0(x)\rvert\le \Lip(\psi)\Lip(\phi)\,\lvert v_0(x)\rvert\in L^1(\mu_0) \).
    Dominated convergence and \( \phi_\ast\mu_0=\mu_0 \) therefore give
    \begin{equation}\label{eq:beta.exp}
        \int\psi\,d\nu_t=\int\psi\,d\mu_0+t\,\big\langle\Koop_0\nabla\psi,\ v_0\big\rangle_{L^2(\mu_0)}+o(t).
    \end{equation}
    Subtract \eqref{eq:alpha.exp} from \eqref{eq:beta.exp} and use the adjunction \( \langle\Koop_0\nabla\psi,v_0\rangle=\langle\nabla\psi,\Transf_0v_0\rangle \).
    Since \( \nabla\psi\in L^2_0(\mu_0) \), we conclude that \( \langle\nabla\psi,\Transf_0v_0\rangle=\langle\nabla\psi,P_{\mu_0}\Transf_0v_0\rangle=\langle\nabla\psi,v_0^\nu\rangle \).
    This gives
    \[
        \int\psi\,d\nu_t-\int\psi\,d\mu_t=t\,\big\langle\nabla\psi,\ v_0^\nu-v_0\big\rangle_{L^2(\mu_0)}+o(t),
    \]
    which is the hypothesis of Lemma~\ref{lem:duality.lower}.
    Applying it yields the estimate \( \liminf_tt^{-1}\dw(\mu_t,\nu_t)\ge\lVert P_{\mu_0}u\rVert_{L^2(\mu_0)} \), where \( \lVert P_{\mu_0}u\rVert=\lVert u\rVert \) because \( u=v_0^\nu-v_0 \) lies in the closed subspace \( L^2_0(\mu_0) \), as do both \( v_0 \) and \( v_0^\nu \).

    \ref{item:acim.sharp}
    By Theorem~\ref{thm:derivative.acim}, \( \dw\left(\nu_t,(\exp tv_0^\nu)_\ast\mu_0\right)=o(t) \).
    By Proposition~\ref{prp:solenoidal.dual} applied to \( v=v_0 \) and \( w=v_0^\nu \),
    \[
        \dw\left(\mu_t,(\exp tv_0^\nu)_\ast\mu_0\right)=t\big\lVert P_{\mu_0}(v_0-v_0^\nu)\big\rVert_{L^2(\mu_0)}+o(t)=t\big\lVert v_0-v_0^\nu\big\rVert_{L^2(\mu_0)}+o(t),
    \]
    the last equality because \( v_0 \) and \( v_0^\nu \) both lie in \( L^2_0(\mu_0) \).
    The triangle inequality in both directions gives the assertion.
\end{proof}

For \( M=S^1 \) and \( \phi \) an expanding map of the circle this is \cite[Theorem 5.1 and \S5.1]{kloeckner_optimal_2013_preprint}.
That proof proceeds through the mixture lemma at a Hölder density and the centering operator, and the monotone rearrangement used in the former confines it to dimension one.
For \( M=\T^d \) and \( \phi=\varphi_A \) linear it is Proposition~\ref{prp:derivative.torus} below, where \( \D_\phi \) is computed explicitly as a composition operator on Fourier coefficients.
The present statement covers every \( C^2 \) covering map of every closed manifold, and in particular every \( C^2 \) expanding map, expanding maps of closed manifolds being covering maps \cite{shub_endomorphisms_1969}.

\subsection{Mixtures and nearly-invariant measures}
We end this section by generalizing two other results from \cite{kloeckner_optimal_2013_preprint}.

The first is the mixture lemma, which says that the exponential map commutes, to first order, with the formation of mixtures, i.e., convex combinations, of probabilities in \(\Prob(M)\). 
On the circle this is \cite[Lemma 4.2]{kloeckner_optimal_2013_preprint}, proved there at a measure \( \rho\Leb \) with H\"older positive density by a geometric argument involving a discretisation of the density at a scale depending on \( t \), and a monotone rearrangement. 
The statement below removes the restriction to dimension one altogether, at the cost of upgrading H\"older to \( C^1 \).

The second result is that Euclidean balls of arbitrary dimension can be topologically embedded in \( \Prob(M) \) around an arbitrary measure. 
This is somewhat weaker than its one-dimensional counterpart \cite[Proposition 3.1]{kloeckner_optimal_2013},where a bi-Lipschitz embedding in \(\Prob(S^1)\) is obtained.

\begin{prp}[Mixtures and exponential deformations]
\label{prp:mixture}
    Let \( \mu=\rho\,\vol \) with \( \rho\in C^1(M) \) and \( \rho>0 \), let \( v_1,\dots,v_N\in L^2(\mu;TM) \) and let \( \alpha_1,\dots,\alpha_N>0 \) with
    \( \sum_i\alpha_i=1 \).  
    Write \( \bar v\coloneqq \sum_i\alpha_iv_i \).
    Then
    \[
        \dw\Big((\exp t\bar v)_\ast\mu,\ \sum_{i=1}^N\alpha_i(\exp tv_i)_\ast\mu\Big)=o(t),\qquad t\searrow0 .
    \]
    The remainder is uniform for \( (v_1,\dots,v_N) \) in a bounded subset of \( W^N \), for any finite-dimensional subspace \( W\subset L^2(\mu;TM) \), the
    number \( N \) and the weights \( \alpha_i \) being fixed.
\end{prp}

\begin{proof}
    We consider first fields of class \( C^1 \).  
    By Proposition~\ref{prp:linear.response}, uniformly on \( M \),
    \[
        \rho_t[\bar v]-\sum_i\alpha_i\rho_t[v_i]=-t\,\Div\Big(\rho\bar v-\rho\sum_i\alpha_iv_i\Big)+o(t)=o(t),
    \]
    the middle expression vanishing identically because \( \Div(\rho\,\cdot) \) is linear and \( \bar v=\sum_i\alpha_iv_i \).
    Both densities are continuous and at least \( \tfrac12\min\rho \) for \( t \) small, so Lemma~\ref{lem:density.comparison} gives \( \dw=o(t) \).

    Now, let us look at fields in \( L^2 \).
    Fix \( \epsilon>0 \) and choose \( w_i\in C^\infty(M;TM) \) with \( \lVert v_i-w_i\rVert_{L^2(\mu)}\le\epsilon \).
    By Lemma~\ref{lem:exp.upper},
    \[
        \limsup_{t\searrow0}\frac{\dw\left((\exp t\bar v)_\ast\mu,(\exp t\bar w)_\ast\mu\right)}{t}\le\lVert\bar v-\bar w\rVert_{L^2(\mu)}\le\epsilon ,
    \]
    while, \( \dw^2 \) being jointly convex with respect to mixtures (a convex combination of transport plans is a plan between the corresponding mixtures),
    \begin{align*}
        \dw\Bigg(\sum_i\alpha_i(\exp tv_i)_\ast\mu,\ & \sum_i\alpha_i(\exp tw_i)_\ast\mu\Bigg)^2\\
        &\le \sum_i\alpha_i\,\dw\left((\exp tv_i)_\ast\mu,(\exp tw_i)_\ast\mu\right)^2,
\end{align*}
    whose \( \limsup_{t}t^{-2} \) is at most \( \sum_i\alpha_i\lVert v_i-w_i\rVert^2\le\epsilon^2 \), again by Lemma~\ref{lem:exp.upper}. 
    Combining with Step 1 and the triangle inequality, \( \limsup_tt^{-1}\dw\le2\epsilon \); let \( \epsilon\searrow0 \).

    Finally, for the uniformity, let \( W\subset L^2(\mu;TM) \) be finite dimensional, let \( R>0 \), and let \( A_\epsilon:W\to C^\infty(M;TM) \) be the map given by Lemma~\ref{lem:uniform.smoothing}.
    Taking \( w_i\coloneqq A_\epsilon v_i \), the two estimates above hold with \( \epsilon \) replaced by \( c_WR\epsilon \), uniformly over \( v_1,\dots,v_N\in W \) with \( \lVert v_i\rVert_{L^2(\mu)}\le R \), because \( A_\epsilon \) is linear and therefore \( \bar w=A_\epsilon\bar v \).
    The fields \( w_i \) and \( \bar w \) range in the compact set of Lemma~\ref{lem:uniform.smoothing}\ref{item:us.compact}, so the remainder of the first step is uniform by Proposition~\ref{prp:linear.response}.
\end{proof}

We now turn to the embedding of Euclidean balls in \( \Prob(M) \).
Fix \( \mu\in\Prob(M) \) and let \( v_1,\dots,v_n\in L^2_0(\mu) \) be linearly independent in \( L^2(\mu;TM) \).
Given \( a=(a_1,\dots,a_n)\in\R^n \), write
\[
    v^a\coloneqq\sum_{i=1}^na_iv_i\ \in\ L^2_0(\mu),
    \qquad T^a(x)\coloneqq  T^{v^a}_1(x) = \exp_x\!\left(v^a(x)\right),
\]
so that \( T^a \) displaces each \( x \) along the geodesic issuing from it with initial velocity \( v^a(x) \), and set
\begin{equation}\label{eq:definition.embedding}
\begin{split}
    F:\R^n&\lto\Prob(M) \\
        a &\mapsto (T^a)_\ast\mu.
\end{split}
\end{equation}
We write \( \mathcal B^n_\eta\coloneqq\{a\in\R^n:\lvert a\rvert<\eta\} \), with the convention \( \mathcal B^n_{+\infty}=\R^n \), and we abbreviate
\[
    R_2\coloneqq\Big(\sum_{i=1}^n\lVert v_i\rVert^2_{L^2(\mu)}\Big)^{1/2}, \qquad
    R_\infty\coloneqq\Big(\sum_{i=1}^n\lVert v_i\rVert^2_{L^\infty(\mu)}\Big)^{1/2}\in(0,+\infty] ,
\]
so that \( \lVert v^a-v^b\rVert_{L^2(\mu)}\le R_2\lvert a-b\rvert \) and \( \lvert v^a(x)\rvert\le R_\infty\lvert a\rvert \) for \( \mu \)-almost every \( x \), both by Cauchy--Schwarz.
Regularity of \( F \) is understood in the distributional model of Remark~\ref{rmk:two.models}: we say that \( F \) is of class \( C^k \) on an open set \( U\subseteq\R^n \) if \( a\mapsto F(a)(\psi) \) is \( C^k \) on \( U \) for every \( \psi\in C^\infty(M) \).
Observe that \eqref{eq:definition.embedding} makes sense for every \( a\in\R^n \), no diffeomorphism property of \( T^a \) being required.

The first two assertions below hold for arbitrary square integrable fields.
The remaining three require control of the derivatives of \( \exp_x \) at the points \( v^a(x) \), and there are two distinct circumstances under which such control is available: either the fields are bounded, so that \( v^a(x) \) is confined to a compact subset of \( TM \), or \( M \) is flat, in which case no confinement is needed because \( \exp_x \) is then the composition of a linear isometry with a Riemannian covering.
We record both, the second being the one relevant to \( \T^d \) and not being implied by the first.

\begin{prp}[Embedding of Euclidean balls in \( \Prob(M) \)]
\label{prp:ball.embedding}
    Let \( \mu\in\Prob(M) \), let \( v_1,\dots,v_n\in L^2_0(\mu) \) be linearly independent in \( L^2(\mu;TM) \), and let \( F \) be given by \eqref{eq:definition.embedding}.
    Then
    \begin{enumerate}
        \item\label{item:emb.diff}
            \( F \) is differentiable at the origin, with
            \[
                D_0F(w)=\iota_\mu\left(v^w\right),\qquad w\in\R^n,
            \]
            and \( D_0F \) is injective.
        \item\label{item:emb.rays}
            \( F \) is metrically differentiable along rays through the origin: for every \( w\in\R^n \),
            \[
                \dw\left(\mu,F(tw)\right)=t\lVert v^w\rVert_{L^2(\mu)}+o(t),\qquad t\searrow0 .
            \]
    \end{enumerate}
        If, in addition, either \( v_1,\dots,v_n\in L^\infty(\mu;TM) \) or \( M \) is flat, then the following hold as well, with \( \eta_0\coloneqq r_1/R_\infty \) in the first case and \( \eta_0\coloneqq+\infty \) in the second.
    \begin{enumerate}[start=3]
        \item\label{item:emb.lip}
            \( F \) is Lipschitz on \( \mathcal B^n_{\eta_0} \):
            \[
                \dw\left(F(a),F(b)\right)\le\Lambda R_2\lvert a-b\rvert,
                \qquad a,b\in\mathcal B^n_{\eta_0},
            \]
            where \( \Lambda \) is the constant of Lemma~\ref{lem:exp.estimates}, and \( \Lambda=1 \) when \( M \) is flat.
        \item\label{item:emb.reg}
            \( F \) is of class \( C^k \) on \( \mathcal B^n_{\eta_0} \) for every \( k\in\N \) such that \( v_1,\dots,v_n\in L^k(\mu;TM) \).
            In particular \( F \) is of class \( C^2 \), and of class \( C^\infty \) when the \( v_i \) are bounded.
        \item\label{item:emb.embedding}
            There is \( \eta\in(0,\eta_0] \) such that \( F|_{\mathcal B^n_\eta} \) is a homeomorphism onto its image, the target carrying the weak\( ^\ast \) topology.
    \end{enumerate}
    Consequently, if \( L^2_0(\mu) \) has dimension at least \( n \), then \( \Prob(M) \) contains a topologically embedded \( n \)-ball through \( \mu \); this holds for every \( n\in\N \) as soon as \( \operatorname{supp}\mu \) is infinite.
\end{prp}

\begin{proof}
    \noindent Item \ref{item:emb.diff}. \\
    Fix \( \psi\in C^\infty(M) \) and \( a\in\R^n \).
    For \( \mu \)-almost every \( x \) the curve \( \gamma(s)\coloneqq\exp_x(sv^a(x)) \), \( s\in[0,1] \), is the geodesic issuing from \( x \) with initial velocity \( v^a(x) \), so that \( \tfrac{d^2}{ds^2}\psi(\gamma(s))=\Hess\psi(\gamma'(s),\gamma'(s)) \) with \( \lvert\gamma'\rvert\equiv\lvert v^a(x)\rvert \), and Taylor's theorem with Lagrange remainder gives
    \[
        \Big\lvert\psi\left(T^a(x)\right)-\psi(x)-\big\langle\nabla\psi(x),v^a(x)\big\rangle\Big\rvert
        \le\tfrac12\lVert\Hess\psi\rVert_\infty\lvert v^a(x)\rvert^2 .
    \]
    Integrating against \( \mu \) and using \( \lVert v^a\rVert_{L^2(\mu)}\le R_2\lvert a\rvert \),
    \[
        \Big\lvert F(a)(\psi)-F(0)(\psi)-\int_M\big\langle\nabla\psi,v^a\big\rangle\,d\mu\Big\rvert
        \le\tfrac12\lVert\Hess\psi\rVert_\infty R_2^2\lvert a\rvert^2 ,
    \]
    so that \( a\mapsto F(a)(\psi) \) is differentiable at the origin with
    \begin{equation}\label{eq:embedding.differential}
        D_0F(w)(\psi)=\int_M\big\langle\nabla\psi,v^w\big\rangle\,d\mu ,
    \end{equation}
    that is, \( D_0F(w)=-\nabla\!\cdot(v^w\mu)=\iota_\mu(v^w) \).
    Only \( v_i\in L^2(\mu;TM) \) has been used here.
    If \( D_0F(w)=0 \), then \( \langle\nabla\psi,v^w\rangle_{L^2(\mu)}=0 \) for every \( \psi\in C^\infty(M) \), so \( v^w\perp L^2_0(\mu) \) by \eqref{eq:L20}; as \( v^w\in L^2_0(\mu) \), this forces \( v^w=0 \), and \( w=0 \) by linear independence of the \( v_i \).
    Thus \( D_0F \) is injective.
    
    \medskip
    \noindent Item \ref{item:emb.rays}. \\
    Since \( v^{tw}=tv^w \) we have \( F(tw)=(\exp tv^w)_\ast\mu \), and \( v^w\in L^2_0(\mu) \), so this is Corollary~\ref{cor:metric.derivative}.

    In order to prove the remaining items, we need finer control on the exponential map.
    For \( \psi\in C^\infty(M) \), \( m\ge1 \) and \( \varrho\in(0,+\infty] \) put
    \[
        N_m(\psi,\varrho)\coloneqq\sup\Big\{\big\lVert D^m(\psi\circ\exp_x)_\xi\big\rVert:\
        x\in M,\ \xi\in T_xM,\ \lvert\xi\rvert\le\varrho\Big\} ,
    \]
    the derivative being that of the smooth real function \( \psi\circ\exp_x \) on the Euclidean space \( T_xM \).
    Write \( \varrho_0\coloneqq r_1 \) in the bounded case and \( \varrho_0\coloneqq+\infty \) in the flat case, so that \( \lvert v^a(x)\rvert\le R_\infty\lvert a\rvert\le\varrho_0 \) for \( \mu \)-almost every \( x \) and every \( a\in\mathcal B^n_{\eta_0} \).
    We claim that in both cases \( N_m(\psi,\varrho_0)<\infty \) for every \( m\ge1 \), and that
    \begin{equation}\label{eq:emb.lipschitz.exp}
        d\left(\exp_x\xi,\exp_x\zeta\right)\le\Lambda\lvert\xi-\zeta\rvert
        \qquad\text{for }\lvert\xi\rvert,\lvert\zeta\rvert\le\varrho_0 ,
    \end{equation}
    with \( \Lambda=1 \) in the flat case.
    In the bounded case the set \( \{(x,\xi)\in TM:\lvert\xi\rvert\le r_1\} \) is compact and \( (x,\xi)\mapsto\psi(\exp_x\xi) \) is smooth on \( TM \), which gives the first claim, and \eqref{eq:emb.lipschitz.exp} is the Lipschitz clause of Lemma~\ref{lem:exp.estimates}.
    In the flat case write \( M=\R^d/\Gamma \) with Riemannian covering \( \pi:\R^d\to M \), fix \( \tilde x\in\pi^{-1}(x) \), and let \( \iota_x\coloneqq(d\pi_{\tilde x})^{-1}:T_xM\to\R^d \), which is a linear isometry.
    Then \( \exp_x\xi=\pi(\tilde x+\iota_x\xi) \) for every \( \xi\in T_xM \), and therefore \( \psi\circ\exp_x \) is the composition of an affine isometry of Euclidean spaces with the smooth \( \Gamma \)-periodic function \( \psi\circ\pi \), so that \( N_m(\psi,+\infty)\le\lVert D^m(\psi\circ\pi)\rVert_{L^\infty(\R^d)}<\infty \).
    Since \( \pi \) is \( 1 \)-Lipschitz, we get \eqref{eq:emb.lipschitz.exp} with \( \Lambda=1 \) and no restriction on \( \xi,\zeta \).

    \medskip
    \noindent Item \ref{item:emb.lip}. \\
    For \( a,b\in\mathcal B^n_{\eta_0} \) the measure \( (T^a,T^b)_\ast\mu \) is a transport plan between \( F(a) \) and \( F(b) \), so \eqref{eq:emb.lipschitz.exp} gives
    \[
    \begin{split}
        \dw\left(F(a),F(b)\right)^2 &\le \int_Md\left(T^a(x),T^b(x)\right)^2\,d\mu(x) \\
                                 &\le\Lambda^2\lVert v^a-v^b\rVert^2_{L^2(\mu)}\le\Lambda^2R_2^2\lvert a-b\rvert^2 .
    \end{split}
    \]

    \medskip
    \noindent Item \ref{item:emb.reg}.\\
    Fix \( \psi\in C^\infty(M) \) and \( 1\le m\le k \).
    The map \( a\mapsto v^a(x) \) being linear, for \( a\in\mathcal B^n_{\eta_0} \) and \( 1\le j_1,\dots,j_m\le n \) we have
    \[
        \partial_{a_{j_1}}\!\cdots\,\partial_{a_{j_m}}\psi\left(T^a(x)\right)
        =D^m(\psi\circ\exp_x)_{v^a(x)}\big[v_{j_1}(x),\dots,v_{j_m}(x)\big] ,
    \]
    which is bounded in absolute value by \( N_m(\psi,\varrho_0)\prod_{l=1}^m\lvert v_{j_l}(x)\rvert \), a function independent of \( a \) and lying in \( L^1(\mu) \) by the generalised Hölder inequality with the \( m \) exponents equal to \( m \).
    For \( \mu \)-almost every \( x \) it is also continuous in \( a \).
    Repeated differentiation under the integral sign is therefore legitimate up to order \( k \), and the resulting derivatives are continuous in \( a \) by dominated convergence, so \( a\mapsto F(a)(\psi) \) is of class \( C^k \).
    Since \( v_i\in L^2(\mu;TM) \) by hypothesis, \( F \) is of class \( C^2 \); and if the \( v_i \) are bounded they lie in \( L^k(\mu;TM) \) for every \( k \), so \( F \) is of class \( C^\infty \).

    \medskip
    \noindent Item \ref{item:emb.embedding}.\\
    We claim first that there are \( \psi_1,\dots,\psi_n\in C^\infty(M) \) for which the \( n\times n \) matrix
    \[
        \Theta_{ij}\coloneqq D_0F(e_j)(\psi_i)\overset{\eqref{eq:embedding.differential}}{=}\int_M\langle\nabla\psi_i,v_j\rangle\,d\mu
    \]
    is invertible.
    For \( \psi\in C^\infty(M) \) let \( \ell_\psi\in(\R^n)^\ast \) be given by \( \ell_\psi(w)\coloneqq D_0F(w)(\psi) \).
    The assignment \( \psi\mapsto\ell_\psi \) is linear, so \( S\coloneqq\{\ell_\psi:\psi\in C^\infty(M)\} \) is a linear subspace of \( (\R^n)^\ast \).
    If \( w\in\R^n \) annihilates \( S \), then \( D_0F(w)(\psi)=0 \) for every \( \psi \), that is \( D_0F(w)=0 \), so \( w=0 \) by \ref{item:emb.diff}.
    The annihilator of \( S \) being trivial, \( S=(\R^n)^\ast \), and it suffices to take \( \psi_1,\dots,\psi_n \) with \( \ell_{\psi_1},\dots,\ell_{\psi_n} \) linearly independent.

    Let
    \[
        G:\Prob(M)\lto\R^n,\qquad G(\nu)\coloneqq\left(\nu(\psi_1),\dots,\nu(\psi_n)\right),
    \]
    and set \( H\coloneqq G\circ F \), so that \( H(a)_i=\int_M\psi_i\left(T^a(x)\right)\,d\mu(x) \).
    By \ref{item:emb.reg} the map \( H \) is of class \( C^1 \) on \( \mathcal B^n_{\eta_0} \), and since \( G \) is the restriction to \( \Prob(M) \) of a continuous linear map \( \Dist(M)\to\R^n \), the chain rule gives
    \[
        (D_0H)_{ij}=D_0F(e_j)(\psi_i)=\Theta_{ij} ,
    \]
    which is invertible.
    By the inverse function theorem there is \( \eta\in(0,\eta_0] \) such that \( H|_{\mathcal B^n_\eta} \) is a diffeomorphism onto its image.
    In particular \( H \) is injective on \( \mathcal B^n_\eta \), hence so is \( F \), and
    \[
        F^{-1}=H^{-1}\circ G|_{F(\mathcal B^n_\eta)}
    \]
    is continuous, because \( G \) is weak\( ^\ast \) continuous, the \( \psi_i \) being continuous, and \( H^{-1} \) is continuous.
    Finally \( F \) is itself weak\( ^\ast \) continuous by \ref{item:emb.lip}, since \( \dw \) metrises the weak\( ^\ast \) topology on \( \Prob(M) \).
    Therefore \( F|_{\mathcal B^n_\eta} \) is a homeomorphism onto its image.

    For the last assertion, suppose \( \dim L^2_0(\mu)\ge n \) and choose \( \psi_1,\dots,\psi_n\in C^\infty(M) \) with \( \nabla\psi_1,\dots,\nabla\psi_n \) linearly independent in \( L^2(\mu;TM) \); these fields lie in \( L^2_0(\mu)\cap L^\infty(\mu;TM) \), so \ref{item:emb.embedding} applies giving the embedded \( n \)-ball containing \( \mu=F(0) \).
    If \( \operatorname{supp}\mu \) is infinite, such fields exist for every \( n \): choose distinct \( p_1,\dots,p_n\in\operatorname{supp}\mu \) and pairwise disjoint balls \( B_i\ni p_i \), and let \( \psi_i\in C^\infty(M) \) be supported in \( B_i \) with \( \nabla\psi_i(p_i)\neq0 \).
    Then \( \nabla\psi_i\neq0 \) on a neighbourhood of \( p_i \), which has positive \( \mu \)-measure because \( p_i\in\operatorname{supp}\mu \), so each \( \nabla\psi_i \) is nonzero in \( L^2(\mu;TM) \); and the \( \nabla\psi_i \) are pairwise orthogonal there, having disjoint supports.
\end{proof}

Neither of the two hypotheses of Proposition~\ref{prp:ball.embedding} implies the other, and the second is not a technical convenience.
What has to be dominated in the proof of \ref{item:emb.reg} is \( \lVert D(\exp_x)_{v^a(x)}\rVert\,\lvert v_j(x)\rvert \), and if the sectional curvature of \( M \) is bounded below by \( -\kappa^2 \), then the Rauch comparison theorem \cite{cheeger_comparison_1975} gives only \( \lVert D(\exp_x)_\xi\rVert\le\sinh(\kappa\lvert\xi\rvert)/(\kappa\lvert\xi\rvert) \), which grows exponentially in \( \lvert\xi\rvert \) and is not integrable against an arbitrary square integrable field.
When \( M \) is flat this factor is identically \( 1 \) and the domination costs nothing, which is why square integrable fields suffice on \( \T^d \).
This is the phenomenon already met before Lemma~\ref{lem:exp.upper}.

\section{Linear endomorphisms of \texorpdfstring{\(\T^d\)}{Td}}
    \label{sec:torus}
    In this section we specialise the theory of Section~\ref{sec:derivatives} to
\( M=\T^d \).
The abstract operator \( \D_\phi \) of Theorem~\ref{thm:derivative.acim} is computed explicit: in suitable coordinates it is a composition operator on Fourier data (Proposition~\ref{prp:transfer.torus}).
Moreover, three of the general results, on solenoidal directions, on mixtures and on the derivative, admit here alternative proofs which are geometric rather than analytic, and we provide them for completeness.
Lastly, and this is the point of the section, the flat geometry makes it possible to \emph{compute} the fixed space of \( \D_\phi \), which in the general setting is Question~\ref{qst:general.eigen}.

Let \( \T^d=\R^d/\Z^d \) be the flat torus with its Haar measure \( \Leb \).
Every non-singular matrix \( A\in \mathrm{M}_d(\Z) \) induces a surjective endomorphism
\[
    \varphi_A:\T^d\lto\T^d,\qquad \varphi_Ax=Ax\ \mathrm{mod}\ \Z^d,
\]
which is \(\lvert \det A \rvert\)-to-one and preserves \( \Leb \).
If all eigenvalues of \( A \) have modulus \( >1 \), then \( \varphi_A \) is expanding, and by \cite{krzyzewski_invariant_1969} it has a unique absolutely continuous invariant measure, which by homogeneity is \( \Leb \).
For simplicity, we write \( \D_A = \D_{\varphi_A} \) from now on.
We prove the following result, generalising \cite[Theorem 1.7]{kloeckner_optimal_2013_preprint}.

\begin{thm}\label{thm:main.application}
    Consider a pair of commuting, expanding matrices \( A,B\in \mathrm{M}_d(\Z) \)  whose determinants are coprime.
    Then the space \( E_A\cap E_B \) of tangent vectors at \( \Leb \) that are simultaneously fixed by \( \D_A\) and \( \D_B\) contains an infinite family of linearly independent continuous vector fields.
    Moreover, for every \( n \) there is a topological embedding \( F:\mathcal B^n_1\to\Prob(\T^d) \) of the open unit ball of \( \R^n \), with \( F(0)=\Leb \) and \( F(a) \) atomless for a.e.\ \( a \), such that
    \[
        \dw\big(( \varphi_A )_\ast F(a),F(a)\big)=o(\lvert a \rvert) = \dw\big(( \varphi_B )_\ast F(a),F(a)\big).
    \]
    If, in addition, the characteristic polynomial of \(B^{r}\) is irreducible over \(\Q\) for every integer \(r>0\), and \(a\neq0\) is such that \(F(a)\) is invariant under both \(\varphi_A\) and \(\varphi_B\), then the set of \(A\)-ergodic components of \(F(a)\) with zero entropy has positive measure.
\end{thm}

The last statement is not vacuous, see Lemma~\ref{lem:existence.pairs}.

\subsection{The derivative at Haar measure}
\label{sec:derivative.haar}

We begin by using Fourier analysis on the torus to give a more explicit description of the tangent spaces of \( \Prob(\T^d) \) at the Haar measure.
To that end, it is convenient to use the Fourier transform
\[
    \mathcal{F}f(k) = \hat f(k)=\int_{\T^d}f(x)e^{-2\pi i\langle k,x\rangle}\,d\Leb(x), \quad k\in\Z^d,
\]
as a normalisation for scalar functions \(f\), and  component-wise for vector-valued \( f \).

\begin{prp}\label{prp:fourier.model}
    The tangent space \(T_\Leb \Prob(\T^d)\) is the space of vector fields
    \[
        L^2_0(\Leb)=\Big\{v\in L^2(\Leb;\R^d):\ \hat v(0)=0\ \text{and}\
        \hat v(k)\in\C k\ \text{for all}\ k\neq0\Big\},
    \]
    and the orthogonal projection \( P_\Leb:L^2(\Leb;\R^d)\to L^2_0(\Leb) \) (the \emph{Helmholtz--Hodge projection}) is given in Fourier form by
    \[
        \widehat{P_\Leb g}(0)=0,\qquad
        \widehat{P_\Leb g}(k)=\frac{\langle \hat g(k),k\rangle}{\lvert k \rvert^2}\,k\quad(k\neq0),
    \]
    where \( \langle\cdot,\cdot\rangle \) denotes the complex bilinear pairing on \( \C^d \), so that this is the orthogonal projection with respect to the underlying real Hilbert structure.
    The orthogonal complement of \( L^2_0(\Leb) \) in \( L^2(\Leb;\R^d) \) is the space of divergence-free (or \emph{solenoidal}) fields with vanishing mean, \( \{z\in L^2:\nabla\cdot z=0,\ \hat z(0)=0\} \), together with the constants.
    Writing, for \( v\in L^2_0(\Leb) \),
    \begin{equation}\label{eq:beta.coordinates}
        \hat v(k)=\frac{\beta_k}{2\pi i\lvert k \rvert^2}\,k\quad(k\neq0),\qquad \beta_{-k}=\overline{\beta_k},
    \end{equation}
    one has \( v=\nabla\psi \) with \( \hat\psi(k)=\beta_k/(2\pi i\,\cdot\,2\pi i\lvert k \rvert^2) \), and
    \begin{equation}\label{eq:beta.norm}
        \lVert v \rVert_{L^2(\Leb)}^2=\frac{1}{4\pi^2}\sum_{k\neq0}\frac{\lvert \beta_k \rvert^2}{\lvert k \rvert^2}.
    \end{equation}
\end{prp}

\begin{proof}
    The set \[
        \mathcal{T} \coloneqq  \Big\{v\in L^2(\Leb;\R^d):\ \hat v(0)=0\ \text{and}\
        \hat v(k)\in\C k\ \text{for all}\ k\neq0\Big\}.
    \]
    is the intersection of the kernels of the continuous
    functionals \( g\mapsto\langle\hat g(k),w\rangle \), \( w\perp k \), and \( g\mapsto\hat g(0) \), and therefore it is a closed subspace of \(L^2(\Leb;\R^d)\).
    Moreover, \(\mathcal{T}\) contains all the gradients of smooth functions, for if \( \psi\in C^\infty(\T^d) \), then one has \( \widehat{\nabla\psi}(k)=2\pi i\,\hat\psi(k)\,k \).
    Hence \( L^2_0(\Leb) \subset \mathcal{T} \).
    Conversely, if \( v \in \mathcal{T} \), write \( \hat v(k)=c_kk \) and set the truncation \( \hat\psi_n(k)\coloneqq c_k/(2\pi i) \) for \( 0<\lvert k \rvert\le n \), \( \hat\psi_n(k)\coloneqq 0 \) otherwise; then \( \psi_n \) is a trigonometric polynomial satisfying
    \[
        \mathcal{F}{(\nabla\psi_n-v)}(k) = \begin{cases}
            0,     & \quad \lvert k \rvert \leq n, \\
            -c_kk, & \quad \lvert k \rvert > n.
        \end{cases}
    \]
    Now, \(\mathcal{F}\) is an isometry on \(L^2\) by Plancherel’s Theorem, hence \(\sum_k\lvert c_k \rvert^2\lvert k \rvert^2=\lVert v \rVert^2_{L^2}<\infty\) and consequently
    \[
        \lVert \nabla\psi_n-v \rVert^2_{L^2}=\sum_{\lvert k \rvert>n}\lvert c_k \rvert^2\lvert k \rvert^2\to0.
    \]
    So \( v\in L^2_0(\Leb) \), and \(L^2_0(\Leb) = \mathcal{T}\) as we wished.

    The formula for \( P_\Leb \) is the fibrewise orthogonal projection of \( \hat g(k) \) onto \( \C k \), and Plancherel gives that this is the orthogonal projection in \( L^2 \).
    A field \( z \) has \( \nabla\cdot z=0 \) and \( \hat z(0)=0 \) if and only if \( \langle\hat z(k),k\rangle=0 \) for all \( k \) and \( \hat z(0)=0 \), which is exactly the fibrewise orthogonal complement.
    Writing the Fourier transform of \(v\) as in \eqref{eq:beta.coordinates} and applying Plancherel's theorem,
    \[
        \lVert v \rVert_{L^2(\Leb)}^2
        =
        \sum_{k\neq0}\lvert \hat v(k) \rvert^2.
    \]
    Since
    \[
        \lvert \hat v(k) \rvert^2
        =
        \left|
        \frac{\beta_k}{2\pi i\lvert k \rvert^2}k
        \right|^2
        =
        \frac{\lvert \beta_k \rvert^2}{4\pi^2\lvert k \rvert^4}\lvert k \rvert^2
        =
        \frac{\lvert \beta_k \rvert^2}{4\pi^2\lvert k \rvert^2},
    \]
    it follows that
    \[
        \lVert v \rVert_{L^2(\Leb)}^2
        =
        \frac{1}{4\pi^2}
        \sum_{k\neq0}
        \frac{\lvert \beta_k \rvert^2}{\lvert k \rvert^2},
    \]
    which is precisely \eqref{eq:beta.norm}.
    Finally, the identity
    \( \beta_{-k}=\overline{\beta_k} \) is exactly the reality condition for the Fourier
    coefficients of the real-valued vector field \( v \).
\end{proof}

As in Remark~\ref{rmk:two.models}, we identify \( T_\Leb\Prob(\T^d) \) with \( L^2_0(\Leb) \) and speak of tangent vectors as vector fields on \( \T^d \).
In particular, we shall abuse notation and write
\[
    \Leb+tv\coloneqq (\Id+tv)_\ast\Leb
\]
for \(v\in L^2(\Leb;\R^d)\).
We also use repeatedly the following elementary estimate: for any \( \mu\in \Prob(\T^d) \) and any \( u,u'\in L^2(\mu;\R^d) \), the map \( (\Id+u,\Id+u')_\ast \mu \) is a transport plan between \( (\Id+u)_\ast \mu \) and \( (\Id+u')_\ast \mu \), whence
\begin{equation}\label{eq:trivial.coupling}
    \dw(\mu + u, \mu + u') = \dw\big((\Id+u)_\ast \mu,(\Id+u')_\ast \mu\big)\le\lVert u-u' \rVert_{L^2(\mu; TM)} .
\end{equation}
Note that this is a degenerate version of Lemma~\ref{lem:exp.upper}, consistent with the fact that on the flat torus \(C_\ast = 0\) and \(\Lambda=1\).

The coordinates \(\beta\) introduced in \eqref{eq:beta.coordinates} may be regarded as a linear coordinate map
\begin{align*}
    \beta:T_\Leb \Prob(\T^d) & \lto  \ell^2\!\left(\Z^d,\frac{1}{\lvert k \rvert^2}\right), \\
    v                        & \longmapsto \beta(v) \coloneqq  \{\beta_k(v)\}_{k},
\end{align*}
where \(\widehat v(k)=(2\pi i\lvert k \rvert^2)^{-1}\beta_k(v)k, \, k\neq0.\)
In these coordinates, Proposition~\ref{prp:derivative.torus} below states simply that
\begin{equation}\label{eq:dynamical.equivariance}
    \beta(\D_Av)=\beta(v)\circ A^{\top},
\end{equation}
Thus, in the \(\beta\)-coordinates, the derivative becomes simply the composition operator induced by \(A^{\top}\), which is why we do not simply use the usual Fourier coefficients.
The price of this normalisation is the weighted norm \eqref{eq:beta.norm}, whose decay will govern the convergence arguments later on.

Next, we explicit the transfer operator of \(  \varphi_A  \) at \( \Leb \), which we will denote by \(\Transf_A\) from now on.
Let \(q=\lvert \det A \rvert\), so that the finite Abelian group \(\Z^d/A\Z^d\) has cardinality \(q\).
Choose representatives \(\varepsilon_1,\dots,\varepsilon_q\in\Z^d\) of its cosets, so that every \(n\in\Z^d\) can be written uniquely as
\[
    n=A\ell+\varepsilon_j.
\]
Using this terminology, we have the following result.

\begin{prp}\label{prp:transfer.torus}
    Let \( A\in \mathrm{M}_d(\Z) \) be non-singular, and let \( \varepsilon_1,\dots,\varepsilon_{q}\in\Z^d \) be representatives of \( \Z^d/A\Z^d \).
    Then, for \( \mu=\Leb \) and \( \phi= \varphi_A  \), the transfer operator \( \Transf_A \) is given by \begin{equation}\label{eq:transfer.explicit}
        \Transf_A v(y)=\frac{1}{q}\sum_{j=1}^{q}A\,v\left(A^{-1}(y+\varepsilon_j)\right),
        \qquad v\in L^2(\Leb;\R^d),
    \end{equation}
    and in Fourier form as    
    \begin{equation}\label{eq:transfer.fourier}
        \widehat{\Transf_A v}(m)=A\,\hat v(A^{\top}m),\qquad m\in\Z^d.
    \end{equation}
    Thus, the operator \( \D_A\coloneqq P_\Leb\Transf_A :L^2_0(\Leb)\to L^2_0(\Leb) \) from Theorem~\ref{thm:derivative.acim} acts in the coordinates \eqref{eq:beta.coordinates} by the composition rule \begin{equation}\label{eq:derivative.beta}
        (\D_A\beta)_m=\beta_{A^{\top}m},\qquad m\neq0,
    \end{equation}
    and \( \lVert \D_A \rVert\le\lVert A \rVert_{\mathrm{op}} \).
\end{prp}

\begin{proof}
    Write \(q = \lvert \det A \rvert\).
    Then \(  \varphi_A  \) is \( q\)-to-one and \( ( \varphi_A )_\ast \Leb=\Leb \).
    The fibre over \( y \) is \( \varphi_A ^{-1}(y)=\{A^{-1}(y+\varepsilon_j)\ \mathrm{mod}\ \Z^d: 1\le j\le q\}\), and by the invariance and homogeneity of \( \Leb \) the disintegration of \( \Leb \) over \( ( \varphi_A )_\ast \Leb=\Leb \) is the uniform measure on the fibre, i.e. \(\Leb^y=q^{-1}\sum_{j}\delta_{A^{-1}(y+\varepsilon_j)}.\)
    Indeed, for \( f\in C(\T^d) \),
    \[
        \int_{\T^d}\left(q^{-1}\sum_jf(A^{-1}(y+\varepsilon_j))\right)d\Leb(y) =\int_{\T^d}f\,d\Leb
    \]
    because the \( q \) branches \( y\mapsto A^{-1}(y+\varepsilon_j) \) are measure preserving up to the factor \( q^{-1} \), their images partition \( \T^d \) up to a null set, and the family is clearly concentrated on the fibres.
    By uniqueness of the Rokhlin disintegration, this is therefore the disintegration of \( \Leb \) over \( ( \varphi_A )_\ast\Leb \).
    Since \( D( \varphi_A )_x=A \) for every \( x \), formula \eqref{eq:transfer.explicit} is exactly the conditional expectation \( w(y)=\int_{ \varphi_A ^{-1}(y)}D( \varphi_A )_xv(x)\,d\Leb^y(x) \) appearing in Proposition~\ref{prp:functoriality.pushforward}, which is \eqref{eq:transfer.disintegration}.

    Now, for \( m\in\Z^d \), we compute the Fourier forms
    \begin{align*}
        \widehat{\Transf_A v}(m)
         & =\frac1q\sum_{j=1}^q
        \int_{\T^d}
        A\,v\!\left(A^{-1}(y+\varepsilon_j)\right)
        e^{-2\pi i\langle m,y\rangle}\,d\Leb(y) \\
         & =\frac1q\sum_{j=1}^q
        q\int_{D_j}
        A\,v(x)e^{-2\pi i\langle A^{\top}m,x\rangle}\,d\Leb(x),
    \end{align*}
    where we substituted \(x=\tau_j(y)=A^{-1}(y+\varepsilon_j)\), which maps \(\T^d\) onto a fundamental domain \(D_j\) with \(d\Leb(y)=q\,d\Leb(x)\), and used \(e^{2\pi i\langle m,\varepsilon_j\rangle}=1\).
    The \(D_j\) partition \(\T^d\) up to a null set, so summing over \(j\) and dividing by \(q\) gives
    \[
        \widehat{\Transf_A v}(m)=\int_{\T^d} A\,v(x)e^{-2\pi i\langle A^{\top}m,x\rangle}\,d\Leb(x) = A\,\hat v(A^{\top}m).
    \]
    which is \eqref{eq:transfer.fourier}.

    For the projection, let \( v\in L^2_0(\Leb) \) with coordinates \eqref{eq:beta.coordinates}.
    For \( m\neq0 \), using \eqref{eq:transfer.fourier} and \( \hat v(A^{\top}m)=\frac{\beta_{A^{\top}m}}{2\pi i\lvert A^{\top}m \rvert^2}A^{\top}m \) (note \( A^{\top}m\neq0 \) since \( \det A\neq0 \)),
    \[
        \widehat{\Transf_A v}(m)=\frac{\beta_{A^{\top}m}}{2\pi i\lvert A^{\top}m \rvert^{2}}\,AA^{\top}m .
    \]
    Applying Proposition~\ref{prp:fourier.model},
    \[
        \begin{split}
            \widehat{\D_Av}(m)
             & =\frac{\langle \widehat{\Transf_A v}(m),m\rangle}{\lvert m \rvert^{2}}\,m                    \\
             & =\frac{\beta_{A^{\top}m}}{2\pi i\lvert A^{\top}m \rvert^{2}}\cdot
            \frac{\langle AA^{\top}m,m\rangle}{\lvert m \rvert^{2}}\,m                                        \\
             & =\frac{\beta_{A^{\top}m}}{2\pi i}\cdot\frac{\lvert A^{\top}m \rvert^{2}}{\lvert A^{\top}m \rvert^{2}}\cdot
            \frac{m}{\lvert m \rvert^{2}}                                                                     \\
             & =\frac{\beta_{A^{\top}m}}{2\pi i\lvert m \rvert^{2}}\,m,
        \end{split}
    \]
    where we used \( \langle AA^{\top}m,m\rangle=\langle A^{\top}m,A^{\top}m\rangle=\lvert A^{\top}m \rvert^2 \).
    Comparing with \eqref{eq:beta.coordinates} gives \( \beta_m(\D_Av)=\beta_{A^\top m}(v) \), i.e.\ \eqref{eq:derivative.beta}.
    Finally \( \hat v(0)=0 \) forces \( \widehat{\Transf_A v}(0)=A\hat v(0)=0 \), so \( \D_Av \) is well defined in \( L^2_0(\Leb) \); the bound \( \lVert \D_A \rVert_{\mathrm{op}}\le\lVert \Transf_A \rVert_{\mathrm{op}}\le\lVert A \rVert_{\mathrm{op}} \) follows from \eqref{eq:transfer.explicit} and Jensen's inequality (as in the proof of Proposition~\ref{prp:functoriality.pushforward} ), \( P_\Leb \) being a projection.
\end{proof}
The computation shows where the projection comes from. 
The mode indexed by \(A^\top m\) is sent to the coefficient \(A\hat v(A^\top m)\), whose direction is \(AA^\top m\) and not \(m\). 
These are parallel for every \(m\) only when \(AA^\top\) is a multiple of the identity, so \(\mathcal L_A\) does not preserve gradient fields, and \(P_\Leb\) is what restores them.

Next, we give an alternative proof of Theorem~\ref{thm:solenoidal} in the case \( \mu=\Leb \), specific to \( \T^d \) and independent of Proposition~\ref{prp:linear.response} and of Lemma~\ref{lem:density.comparison}.
It is geometric: it exhibits the transport plan explicitly, as the deformation \( \Id+tv \) composed with the flow of \( z \), and it makes literal the statement that a solenoidal field stirs \( \Leb \) rather than moving it.
Recall that \( \Leb+t(v+z) \) denotes the pushforward measure \( (\Id+t(v+z))_\ast\Leb \), and not a linear combination of measures.

\begin{proof}[Second proof of Theorem~\ref{thm:solenoidal} at \(\mu=\Leb\)]
    Assume first that \( z \) is smooth and solenoidal.
    Let \( (\Psi_t)_{t\in\R} \) be the flow of \( z \), i.e. \( \partial_t\Psi_t(x)=z(\Psi_t(x)) \), \( \Psi_0=\Id \).
    It is a well-defined flow of diffeomorphisms of \( \T^d \) since \( \T^d \) is compact and \( z \) is smooth.
    As \( \nabla\cdot z=0 \), Liouville's theorem gives \( (\Psi_t)_\ast \Leb=\Leb \) for all
    \( t \).
    Since \(z\in C^\infty\), the flow \(\Psi_t\) admits the uniform first-order expansion \(\Psi_t=\Id+tz+O(t^2)\) as \(t\to0\).
    More precisely, by a Taylor expansion in \( t \) we get
    \begin{equation}\label{eq:flow.taylor}
        \left\lVert \Psi_t-(\Id+tz) \right\rVert_\infty\le C_zt^2 ,\qquad \lvert t \rvert\le1,
    \end{equation}
    where \(C_z\) depends only on \(\lVert z \rVert_\infty\) and \(\lVert Dz \rVert_\infty\), both of which are finite.
    Now consider the map \( \Theta_t\coloneqq (\Id+tv)\circ\Psi_t \).
    Since \( (\Psi_t)_\ast \Leb=\Leb \), we have \( (\Theta_t)_\ast \Leb=(\Id+tv)_\ast \Leb=\Leb+tv \).
    Therefore \( \left(\Id+t(v+z),\ \Theta_t\right)_\ast \Leb \) is a transport plan between \( \Leb+t(v+z) \) and \( \Leb+tv \), and consequently
    \[
        \dw\left(\Leb+t(v+z),\Leb+tv\right)^2
        \le\int_{\T^d}\big\lvert x+t v(x)+tz(x)-\Theta_t(x)\big \rvert^2\,d\Leb(x).
    \]
    Write \( \Theta_t(x)=\Psi_t(x)+tv(\Psi_t(x)) \) and use \eqref{eq:flow.taylor}:
    \[
        \begin{split}
            \big\lvert x+tv(x)+tz(x)-\Theta_t(x)\big \rvert
             & \le\big\lvert x+tz(x)-\Psi_t(x)\big \rvert+t\big\lvert v(x)-v(\Psi_t(x))\big \rvert \\
             & \le C_zt^2+t\big\lvert v(x)-v(\Psi_t(x))\big \rvert .
        \end{split}
    \]
    Hence, by the triangle inequality in \( L^2(\Leb) \),
    \begin{equation}\label{eq:smooth.z.bound}
        \dw\left(\Leb+t(v+z),\Leb+tv\right)
        \le C_zt^2+t\,\left\lVert v-v\circ\Psi_t \right\rVert_{L^2(\Leb)} .
    \end{equation}
    Since \( \Psi_t\to\Id \) uniformly as \( t\to0 \) and each \( \Psi_t \) preserves \( \Leb \), composition with measure-preserving maps converging uniformly to the identity is continuous on \(L^2\), so \( \lVert v-v\circ\Psi_t \rVert_{L^2(\Leb)}\to0 \): indeed, for \( v \) continuous this is uniform convergence, and for general \( v\in L^2 \) approximate \( v \) by a continuous \( \tilde v \) with \( \lVert v-\tilde v \rVert_{L^2}\le\epsilon \) and note that \(\lVert v\circ\Psi_t-\tilde v\circ\Psi_t \rVert_{L^2(\Leb)}=\lVert v-\tilde v \rVert_{L^2(\Leb)} \le\epsilon\) because \( (\Psi_t)_\ast \Leb=\Leb \).
    Thus the right-hand side of \eqref{eq:smooth.z.bound} is \( o(t) \).

    For general \(z\) we mollify by convolutions.
    Let \( \chi_n \) be a smooth approximation of the identity on \( \T^d \) and \( z_n\coloneqq z*\chi_n \).
    Convolution on the group commutes with differentiation, so \( \nabla\cdot z_n=(\nabla\cdot z)*\chi_n=0 \).
    Moreover \( z_n\to z \) in \( L^2(\Leb) \) and each \( z_n \) is smooth.
    By \eqref{eq:trivial.coupling},
    \[
        \dw\left(\Leb+t(v+z),\Leb+t(v+z_n)\right)\le t\lVert z-z_n \rVert_{L^2(\Leb)} .
    \]
    Combining with the first case applied to \( z_n \) and the triangle inequality,
    \[
        \begin{split}
            \limsup_{t\to0}\frac{\dw\left(\Leb+t(v+z),\Leb+tv\right)}{t}
             & \le\lVert z-z_n \rVert_{L^2(\Leb)} \\
            +\limsup_{t\to0} \frac{\dw\left(\Leb+t(v+z_n),\Leb+tv\right)}{t}
             & =\lVert z-z_n \rVert_{L^2(\Leb)}.
        \end{split}
    \]
    Letting \( n\to\infty \) gives the claim.

    Finally, fix a basis \( v_1,\dots,v_N \) of a given \( V\subset L^2(\Leb;\R^d) \).
    In \eqref{eq:smooth.z.bound} the constant \( C_z \) does not involve \( v \), and
    \[
        \lVert v-v\circ\Psi_t \rVert_{L^2(\Leb)}\le\sum_{i=1}^N\lvert c_i \rvert\,\lVert v_i-v_i\circ\Psi_t \rVert_{L^2(\Leb)}
        \le C(V,R)\max_i\lVert v_i-v_i\circ\Psi_t \rVert_{L^2(\Leb)}
    \]
    for \( v=\sum c_iv_i \) with \( \lVert v \rVert_{L^2}\le R \) (all norms on \( V \) being equivalent), and the right-hand side tends to \( 0 \) as \( t\to0 \) independently of \( v \).
    The mollification step is likewise uniform, since the estimate
    \[
        \dw(\Leb+t(v+z),\Leb+t(v+z_n))
        \le t\lVert z-z_n \rVert_{L^2(\Leb)}
    \]
    is independent of \(v\).
\end{proof}

Proposition~\ref{prp:mixture} already gives, at \(\Leb\), everything we shall use. 
The second proof below is geometric: it displays the transport plan, it produces a rate at each fixed discretisation scale, and it is the version that makes the comparison with the circle meaningful.

\begin{proof}[Second proof of Proposition~\ref{prp:mixture} at \( \Leb \)]
    Fix \( \epsilon>0 \).
    We begin by approximating the fields \(\{v_i\}\) by locally constant fields, in the following manner.
    Choose \( n\in\N \) and fields \( \bar v_i \) that are constant on each cube of the partition \( \mathcal Q_n \) of \( \T^d \) into \( n^d \) half-open cubes of side \( 1/n \), such that \( \lVert \bar v_i-v_i \rVert_{L^2(\Leb)}\le\epsilon \) for every \( i \).
    Such choice is always possible as step functions on such grids are dense in \( L^2 \).
    Note that \( n \) and the \( \bar v_i \) depend on \( \epsilon \) but \emph{not} on \( t \); in particular \( K\coloneqq \max_{i,j}\lVert \bar v_i-\bar v_j \rVert_\infty<\infty \) is independent of \( t \).
    By \eqref{eq:trivial.coupling},
    \[
        \dw\Big(\Leb+t\sum_i\alpha_iv_i,\ \Leb+t\sum_i\alpha_i\bar v_i\Big)
        \le t\Big\lVert \sum_i\alpha_i(v_i-\bar v_i) \Big\rVert_{L^2(\Leb)}\le\epsilon t,
    \]
    and, since \( \dw \) is jointly convex with respect to mixtures (a convex combination of transport plans is a transport plan between the corresponding mixtures), we have
    \[
    \begin{split}
        \dw\Big(\sum_i\alpha_i(\Leb+t\bar v_i),\ \sum_i\alpha_i(\Leb+tv_i)\Big)^2
        &\le\sum_i\alpha_i\,\dw\big(\Leb+t\bar v_i,\Leb+tv_i\big)^2 \\
        &\le\epsilon^2t^2.
    \end{split}
    \]
    By the triangle inequality it therefore suffices to prove
    \begin{equation}\label{eq:mixture.discrete}
        \dw\Big(\Leb+t\sum_i\alpha_i\bar v_i,\ \sum_i\alpha_i(\Leb+t\bar v_i)\Big)
        \le C\,t^{3/2}n^{1/2}K^{3/2}
    \end{equation}
    for a dimensional constant \( C \) and all \( t \) small enough, since then
    \[
        \limsup_{t\to0}\frac{1}{t}\,
        \dw\Big(\Leb+t\sum_i\alpha_iv_i,\ \sum_i\alpha_i(\Leb+tv_i)\Big)\le2\epsilon,
    \]
    and \( \epsilon>0 \) was arbitrary.

    Now, both measures in \eqref{eq:mixture.discrete} are mixtures of the measures obtained by translating \( \Leb|_Q \), \( Q\in\mathcal Q_n \), by constant vectors.
    Precisely, fix \( Q\in\mathcal Q_n \) and let \( u_i\coloneqq \bar v_i|_Q\in\R^d \) and \( \bar u\coloneqq \sum_i\alpha_iu_i \).
    Restricted to \( Q \), the left measure is \( n^{-d}\,\mathcal U(Q+t\bar u) \) and the right one is \( n^{-d}\sum_i\alpha_i\,\mathcal U(Q+tu_i) \), where \( \mathcal U(E) \) denotes the uniform probability measure on \( E \) and \( Q+w \) is the translate of \( Q \) by
    \( w\in\R^d \) (these translates are cubes in \( \T^d \) for \( t \) small.)
    We claim that
    \begin{equation}\label{eq:one.cube}
        \dw\Big(\mathcal U(Q+t\bar u),\ \sum_i\alpha_i\,\mathcal U(Q+tu_i)\Big)^2
        \le C_d\,t^3K^3n\qquad\text{for } tK\le\tfrac{1}{2n}.
    \end{equation}
    Indeed, all the cubes \( Q+tu_i \) and \( Q+t\bar u \) are translates of the fixed cube \( Q \) of side \( 1/n \) by vectors of length at most \( tK \).
    Write \( \tau_i \) for the translation by \( t(\bar u-u_i) \), so that \( \tau_i(Q+tu_i)=Q+t\bar u \).
    Consider the transport plan that, for each \( i \), matches the mass of \( \alpha_i\mathcal U(Q+tu_i) \) with the mass of \( \alpha_i\mathcal U(Q+t\bar u) \) by \( \tau_i \), that is, the plan \( \Pi\coloneqq \sum_i\alpha_i(\Id,\tau_i)_\ast \mathcal U(Q+tu_i) \) read as a coupling of \( \sum_i\alpha_i\mathcal U(Q+tu_i) \) with \( \mathcal U(Q+t\bar u) \).
    Its second marginal is \( \sum_i\alpha_i\mathcal U(Q+t\bar u)=\mathcal U(Q+t\bar u) \), as required, and its cost is at most \( \max_i\lvert t(\bar u-u_i) \rvert^2\le t^2K^2 \).

    This crude plan is not good enough :it gives \( \dw\le tK \) per cube, hence \( \dw\le tK \) overall, which is \( O(t) \) and not \( o(t) \).
    The point, exactly as in \cite[Figure 2]{kloeckner_optimal_2013_preprint}, is that most of the mass does not have to move at all, so we improve it as follows
    Let \( R\coloneqq (Q+t\bar u)\cap\bigcap_i(Q+tu_i) \) be the common core.
    Since all the cubes are translates of \(Q+t\bar u\) by vectors of length at most \(tK\) and \( Q \) has side \( 1/n \), the set \( R \) contains a cube of side \( 1/n-2tK\ge0 \) and
    \[
        \Leb(Q\setminus R)\le C_d\,\frac{tK}{n^{d-1}}=C_d\,tK\,n\cdot n^{-d}
    \]
    (the symmetric difference of two translates of \( Q \) by a vector of length \( \le\ell \) has measure at most \( 2d\,\ell\,n^{-(d-1)} \)).
    Define the plan \( \Pi' \) that leaves in place all the mass lying in \( R \), where both measures have the same density \( n^{d} \) and we couple them by the identity, and moves the remaining mass arbitrarily.
    That mass has total weight at most \( C_dtKn\cdot n^{-d}\cdot n^{d}=C_dtKn \) relative to \( \mathcal U \), and it is supported in the boundary strips where the translated cubes differ.
    Every cube involved is a translate of \( Q+t\bar u \) by a vector of length at most \( tK \), so no point of those strips has to travel further than \( 2tK \) to reach \( R \).
    Then \( \Pi' \) is a transport plan whose squared cost is bounded by the amount of moved mass times the distance squared, i.e.
    \[
        \dw\Big(\mathcal U(Q+t\bar u),\sum_i\alpha_i\mathcal U(Q+tu_i)\Big)^2
        \le C_dtKn\cdot(2tK)^2
        =C_d'\,t^3K^3n,
    \]
    which is \eqref{eq:one.cube}.

    To finish the argument, we sum the plans above over \( Q\in\mathcal Q_n \) with weights \( n^{-d} \), producing a transport plan between \( \Leb+t\sum_i\alpha_i\bar v_i \) and \( \sum_i\alpha_i(\Leb+t\bar v_i) \): the marginals add up correctly because both measures decompose over \( \mathcal Q_n \), the fields \( \bar v_i \) being constant on each cube.
    Its cost is at most \( \sum_{Q}n^{-d}\cdot C_d't^3K^3n=C_d't^3K^3n \), whence \eqref{eq:mixture.discrete}, as we wished.

    As for the uniformity claim, given a finite-dimensional \( W \) and a bound \( R \) on \( \lVert v_i \rVert_{L^2} \), choose the grid \( n=n(\epsilon,W,R) \) and the discretisations \( \bar v_i \) by discretising a fixed basis of \( W \) and taking the corresponding linear combinations; then \( K\le C(W,R) \) and \( n \) are uniform over the bounded set, and all the estimates above depend on \( (v_i) \) only through \( K \) and \( \epsilon \).
\end{proof}

The two proofs differ in the following way.
The proof of \cite[Lemma 4.2]{kloeckner_optimal_2013_preprint} has three ingredients: the discretisation of the fields, a discretisation of the \emph{density} at a \( t \)-dependent scale, and the overlap construction.
The second is the delicate one, and it is what forces the Hölder hypothesis there; at \( \Leb \) the density is constant, so that step is vacuous and the grid \( n \) may be kept fixed as \( t\searrow0 \).
The analytic proof of Proposition~\ref{prp:mixture} bypasses the difficulty entirely, at non-constant \( C^1 \) density and in any dimension, because it never discretises: the cancellation happens at the level of the linear response \eqref{eq:linear.response}, where it is the linearity of \( u\mapsto\Div(\rho u) \) that does all the work.
The price paid is that the analytic proof produces neither a transport plan nor a rate.

Theorem~\ref{thm:derivative.acim} applies verbatim to \( \phi=\varphi_A \), which is a \( C^\infty \) covering map preserving \( \Leb=\vol \), and identifies the derivative as \( \D_A=P_\Leb\Transf_A \), computed in Proposition~\ref{prp:transfer.torus}.
We restate the conclusion in toral notation and then give a second proof, independent of Section~\ref{sec:derivatives}, which uses neither the linear response formula nor Lemma~\ref{lem:density.comparison}: it decomposes the pushforward along the inverse branches of \( \varphi_A \) and appeals only to Proposition~\ref{prp:mixture} and Theorem~\ref{thm:solenoidal}.
Since it is exactly the route followed by \cite[\S5]{kloeckner_optimal_2013} on the circle, we keep it to make the generalisation transparent.

\begin{prp}[Derivative of \( ( \varphi_A )_\ast  \) at the Haar measure]
    \label{prp:derivative.torus}
    Let \( A\in \mathrm{M}_d(\Z) \) with \( q\coloneqq\lvert \det A \rvert\neq 0 \) and let \( \D_A=P_\Leb\Transf_A  \) be as in Proposition~\ref{prp:transfer.torus}.
    Then \( ( \varphi_A )_\ast  \) has G\^ateaux derivative \( \D_A \) at \( \Leb \): for every \( v\in L^2_0(\Leb)\cong T_\Leb \Prob(\T^d) \),
    \begin{equation}\label{eq:gateaux}
        \dw\left(( \varphi_A )_\ast (\Leb+tv),\ \Leb+t\,\D_Av\right)=o(t)\qquad(t\to0),
    \end{equation}
    and the remainder is uniform on bounded subsets of any finite-dimensional subspace of \( L^2_0(\Leb) \).
    In the coordinates \eqref{eq:beta.coordinates} the derivative is the composition operator \( (\D_A\beta)_m=\beta_{A^{\top}m} \).
\end{prp}

\begin{proof}[Second proof, elementary.]
    Let \( v\in L^2_0(\Leb) \) and let \( \varepsilon_1,\dots,\varepsilon_{q} \) and the branches \( \tau_j(y)\coloneqq A^{-1}(y+\varepsilon_j) \) be as in Proposition~\ref{prp:transfer.torus}.
    We claim that \begin{equation}\label{eq:branch}
        ( \varphi_A )_\ast \left(\Leb+tv\right)
        =\sum_{j=1}^{q}\frac{1}{q}\left(\Leb+t\,u_j\right),
        \qquad u_j(y)\coloneqq A\,v\left(\tau_j(y)\right).
    \end{equation}
    Indeed, \( \Leb=\sum_j q^{-1}(\tau_j)_\ast \Leb \) as the branches parametrise the \( q \) pieces of a partition of \( \T^d \) into fundamental domains for \( A\Z^d \), each carrying mass \( q^{-1} \).
    Hence \( \Leb+tv=(\Id+tv)_\ast  \Leb=\sum_jq^{-1}\left((\Id+tv)\circ\tau_j\right)_\ast \Leb \) and therefore \( ( \varphi_A )_\ast (\Leb+tv)=\sum_jq^{-1}\left(\varphi_A \circ(\Id+tv)\circ\tau_j\right)_\ast \Leb \).
    Now for \( x\in\T^d \) we have, in \( \T^d \), \(  \varphi_A (x+tv(x))=Ax+tAv(x)= \varphi_A (x)+tAv(x) \), since \( A \) is linear and integer.
    Thus \( \varphi_A \circ(\Id+tv)\circ\tau_j(y)= \varphi_A (\tau_j(y))+tAv(\tau_j(y)) =y+tu_j(y)\) , using  \( \varphi_A \circ\tau_j=\Id\).
    This is \eqref{eq:branch}.

    Now, from Equation \eqref{eq:transfer.explicit} we get \( q^{-1}\sum_ju_j=\Transf_A v \).
    As the \( \tau_j \) push \( \Leb \) to a multiple of \( \Leb \), so that \( \lVert u_j \rVert_{L^2(\Leb)}\le q^{1/2}\lVert A \rVert_{\mathrm{op}}\lVert v \rVert_{L^2(\Leb)} \), we have \( u_j\in L^2(\Leb;\R^d) \).
    Then, Proposition~\ref{prp:mixture} applied with \( N=q \), \( \alpha_j=q^{-1} \) and the fields \( u_j \) gives
    \begin{equation}\label{eq:step2}
        \dw\Big(\sum_j q^{-1}\big(\Leb+tu_j\big),\ \Leb+t\,\Transf_A v\Big)=o(t).
    \end{equation}

    Finally, we write \( \Transf_A v=\D_Av+z \) with \( z\coloneqq (\Id-P_\Leb)\Transf_A v \).
    By Proposition~\ref{prp:fourier.model}, \( z \) is divergence-free with \( \hat z(0)=0 \)  (recall \( \widehat{\Transf_A v}(0)=A\hat v(0)=0 \)).
    Hence Theorem~\ref{thm:solenoidal}, applied in the direction \(\D_Av\) with solenoidal part \(z\), yields
    \begin{equation}\label{eq:step3}
        \dw\left(\Leb+t\,\Transf_A v,\ \Leb+t\,\D_Av\right)=o(t).
    \end{equation}
    Combining \eqref{eq:branch}, \eqref{eq:step2}, \eqref{eq:step3} and the triangle inequality proves \eqref{eq:gateaux}.

    As for uniformity, let \( V\subset L^2_0(\Leb) \) be finite dimensional and \( R>0 \).
    The fields \( u_j=Av\circ\tau_j \), for \( v\in V \) with \( \lVert v \rVert_{L^2}\le R \), range in the finite-dimensional space \( W\coloneqq \{Aw\circ\tau_j:w\in V,\ 1\le j\le q\} \) and are bounded there; so the uniform version of Proposition~\ref{prp:mixture} applies.
    Likewise \( \D_Av \) ranges in the finite-dimensional \( \D_A(V) \) and is bounded, so the uniform version of Theorem~\ref{thm:solenoidal} applies, and the field \( z \) occurring there is \( z=(\Id-P_\Leb)\Transf_A v \) with \( \lVert z \rVert_{L^2}\le\lVert A \rVert_{\mathrm{op}}R \).
    Inspecting the second proof of Theorem~\ref{thm:solenoidal}, the bound depends on \( z \) only through \( \lVert z-z_n \rVert_{L^2} \), and on the finite-dimensional space \( (\Id-P_\Leb)\Transf_A (V) \), so it too is uniform.
\end{proof}

\begin{exm}[Sharpness of Proposition~\ref{prp:pushforward.rate}]\label{exm:rate.sharpness}
    Let \( d=1 \), \( A=(2) \), so that \( \varphi_A(x)=2x \bmod 1 \) and \( \Lip(\varphi_A)=2 \), and let
    \[
        v(x)\coloneqq\tfrac1\pi\sin(2\pi x)\in L^2_0(\Leb),\qquad \lVert v\rVert_{L^2(\Leb)}=\tfrac{1}{\pi\sqrt2} .
    \]
    By \eqref{eq:transfer.explicit} with \( q=2 \),
    \[
        \Transf_A v(y)=\tfrac12\sum_{j=0}^{1}2\,v\!\left(\tfrac{y+j}{2}\right)
        =\tfrac1\pi\sin(\pi y)+\tfrac1\pi\sin(\pi y+\pi)=0 ,
    \]
    so \( \D_Av=0\neq v \); equivalently, in the coordinates \eqref{eq:beta.coordinates} the sequence \( \beta(v) \) is supported on \( k=\pm1 \) and \( (\D_A\beta)_m=\beta_{2m} \) annihilates it.

    Fix \( \alpha\in(0,1) \) and set \( \mu_t\coloneqq\Leb+t^{\alpha}v \) for \( t\in[0,t_0] \), \( t_0^\alpha<\tfrac12 \).
    Since \( \left(s\,\tfrac1\pi\sin2\pi x\right)'=2s\cos2\pi x>-1 \) and no mass moves further than \( s/\pi<\tfrac12 \) for \( s<\tfrac12 \), the segment \( s\mapsto\Leb+sv \) is a constant-speed geodesic of speed \( \lVert v\rVert_{L^2(\Leb)} \) \cite[\S3.3]{kloeckner_optimal_2013}.
    Hence \( (\mu_t) \) is absolutely continuous with
    \[
        \big\lVert v_t\big\rVert_{L^2(\mu_t)}=\alpha\,\lVert v\rVert_{L^2(\Leb)}\,t^{\alpha-1}
        \ \in\ L^p(0,t_0)\quad\text{for every }p<\tfrac{1}{1-\alpha},
    \]
    while \( \esssup_{[0,t_0]}\lVert v_t\rVert_{L^2(\mu_t)}=\infty \).
    Writing \( \sigma_t\coloneqq(\varphi_A)_\ast\mu_t \), Proposition~\ref{prp:derivative.torus} and \( \D_Av=0 \) give \( \dw(\sigma_t,\Leb)=o(t^{\alpha}) \), whence
    \[
        \dw(\mu_t,\sigma_t)\ \ge\ \dw(\mu_t,\Leb)-\dw(\Leb,\sigma_t)
        \ =\ t^{\alpha}\lVert v\rVert_{L^2(\Leb)}-o(t^{\alpha}) ,
    \]
    which is not \( O(t) \).
    Letting \( \alpha\nearrow1 \) shows that \( L^p \) integrability of the speed is insufficient for every \( p<\infty \).
    Note that the failure is one of parametrisation: the image of \( (\mu_t) \) is a geodesic segment, and \eqref{eq:pushforward.rate} holds throughout.
\end{exm}

\subsection{Joint fixed directions and nearly invariant families}
We now begin the proof of Theorem~\ref{thm:main.application}. 
First we show that for a commuting pair of expanding matrices with coprime determinants the space of directions in \(L^2_0(\Leb)\) fixed by both derivatives is infinite dimensional (Theorem~\ref{thm:joint.eigen}); then we show that those directions carry families of measures that remain invariant to first order under both pushforwards (Theorem~\ref{thm:families}). 
Together these give Theorem~\ref{thm:main.application}

We begin with two technical lemmas.
The first is simply a statement of uniform expansion.

\begin{lem}[Orbit growth control]
    \label{lem:orbit.growth}
    Let \( A\in \mathrm{M}_d(\Z) \) be expanding and set \( \theta_A\coloneqq \min\lvert \spec(A) \rvert>1 \).
    Then for every \( 1<\theta<\theta_A \) there is \( C=C(A,\theta)\ge1 \) such that
    \[
        \big\lvert (A^{\top})^{a}w\big \rvert\ \ge\ C^{-1}\theta^{a}\,\lvert w \rvert \qquad\text{for all } w\in\R^d,\ a\ge0 .
    \]
\end{lem}

\begin{proof}
    \( A^{\top} \) and \( A \) have the same spectrum, so all eigenvalues of \( A^{\top} \) have modulus \( \ge\theta_A>\theta \).
    Hence the spectral radius of \( (A^{\top})^{-1} \) is \( 1/\theta_A<1/\theta \), and by Gelfand's formula \( \lVert ((A^{\top})^{-1})^{a} \rVert_{\mathrm{op}}^{1/a}\to1/\theta_A \), so there is \( C\ge1 \) with \( \lVert (A^{\top})^{-a} \rVert_{\mathrm{op}}\le C\theta^{-a} \) for all \( a\ge0 \).
    Then for \( w\in\R^d \), one has
    \[
        \lvert w \rvert=\lvert (A^{\top})^{-a}(A^{\top})^{a}w \rvert\le C\theta^{-a}\lvert (A^{\top})^{a}w \rvert.
    \] 
\end{proof}

The second lemma is an analogue, for a commuting pair of integer matrices with coprime determinants, of the elementary arithmetic fact that every positive integer factors uniquely as \( n=2^{a}3^{b}j \) with \( \gcd(j,6)=1 \), which is what makes the computation of \cite[\S7]{kloeckner_optimal_2013_preprint} possible on the circle.

\begin{dft}
    An integer vector \( j\in\Z^d\setminus\{0\} \) is said to be \emph{bi-primitive} with respect to a pair of commuting integer matrices \(A, B\) if \( j\notin A^{\top}\Z^d\cup B^{\top}\Z^d \).
\end{dft}

\begin{lem}
    \label{lem:freeness}
    Let \( A,B\in \mathrm{M}_d(\Z) \) be commuting expanding matrices such that \( \gcd(\det A,\det B)=1 \).
    Then:
    \begin{enumerate}
        \item\label{item:intersection}
            \( A^{\top}\Z^d\cap B^{\top}\Z^d=A^{\top}B^{\top}\Z^d \);
        \item\label{item:finite}
            for every \( k\in\Z^d\setminus\{0\} \) and \(M \in \{A,B\}\), there exists the maximum \(\nu_M(k) \coloneqq \max\{n\ge0:k\in(M^{\top})^{n}\Z^d\} \);
        \item\label{item:unique}
            every \( k\in\Z^d\setminus\{0\} \) can be written in a unique way as
            \[
                k=(A^{\top})^{a}(B^{\top})^{b}\,j,\qquad a,b\ge0,\ j\ \text{bi-primitive}.
            \]
    \end{enumerate}
\end{lem}

\begin{proof}
    ~\ref{item:intersection}  The inclusion \( \supseteq \) is clear.
    Conversely let \( k=A^{\top}u=B^{\top}w \) with \( u,w\in\Z^d \), and put \( x\coloneqq (A^\top B^\top)^{-1}k,\in\Q^d \).
    Using the adjugate identity \( (\det B)I = (\det B^\top)I = \adj(B^{\top})B^{\top}\) together with the relation \( (A^{\top})^{-1}k=u\in\Z^d \), we have
    \[
        \det B\,x=\adj(B^{\top})\,(A^{\top})^{-1}k=\adj(B^{\top})\,u\in\Z^d.
    \]
    Symmetrically, using \( (B^{\top})^{-1}k=w\in\Z^d \),
    \[
        \det A \,x=\adj(A^{\top})\,(B^{\top})^{-1}k=\adj(A^{\top})\,w\in\Z^d .
    \]
    Thus \( x\in(\det A)^{-1}\Z^d\cap(\det B)^{-1}\Z^d \).
    Since \( \gcd(\det A,\det B)=1 \), applying Bezout gives \( \alpha,\beta\in\Z \) with \( \alpha\det A+\beta\det B=1 \), so \( x=\alpha\det A x+\beta\det Bx\in\Z^d \).
    Hence \( k=A^{\top}B^{\top}x\in A^{\top}B^{\top}\Z^d \) as wanted.

    ~\ref{item:finite}  Suppose \( k=(A^{\top})^{a}k_a \) with \( k_a\in\Z^d\setminus\{0\} \) for arbitrarily large \( a \).
    By Lemma~\ref{lem:orbit.growth} (with any \( 1<\theta<\theta_A \)), \( \lvert k \rvert=\lvert (A^{\top})^{a}k_a \rvert\ge C^{-1}\theta^{a}\lvert k_a \rvert\ge  C^{-1}\theta^{a} \), because \( k_a\in\Z^d\setminus\{0\} \) forces \( \lvert k_a \rvert\ge1 \).
    As \( \theta>1 \) this is impossible for \( a \) large.
    Hence \( \nu_A(k)<\infty \); same for \( \nu_B \).

    ~\ref{item:unique}  First we show the existence of such a decomposition with bi-primitive \(j\).
    Let \( k\neq0 \), \( a\coloneqq \nu_A(k) \), \( b\coloneqq \nu_B(k) \), both finite by~\ref{item:finite}.
    We show by induction on \( a+b \) that \( k \) has a decomposition as claimed.
    If \( a=b=0 \) then \( k \) is bi-primitive and we are done.
    If \( a\ge1 \), write \( k=A^{\top}k' \); then \( \nu_A(k')=a-1 \) and \( \nu_B(k')\le\nu_B(k) \) (if \( k'=(B^{\top})^{c}k'' \) then \( k=(B^\top)^cA^\top k'' \), so \( c\le\nu_B(k) \)), and by induction \( k' \) decomposes, hence so does \( k \).
    The case \( b\ge1 \) is symmetric.

    Finally, for the uniqueness, suppose \( (A^{\top})^{a}(B^{\top})^{b}j=(A^{\top})^{a'}(B^{\top})^{b'}j' \) with \( j,j' \) bi-primitive.
    We first claim \( a=\nu_A\left((A^{\top})^{a}(B^{\top})^{b}j\right) \).
    Clearly \( \nu_A\ge a \).
    Suppose \( \nu_A\ge a+1 \), i.e. \( (A^{\top})^{a}(B^{\top})^{b}j=(A^{\top})^{a+1}\ell \) for some \( \ell\in\Z^d \); cancelling the injective \( (A^{\top})^{a} \) we get \( (B^{\top})^{b}j=A^{\top}\ell \).
    If \( b=0 \) this says \( j\in A^{\top}\Z^d \), contradicting bi-primitivity.
    If \( b\ge1 \), then \( (B^{\top})^{b}j\in A^{\top}\Z^d\cap B^{\top}\Z^d=A^{\top}B^{\top}\Z^d \) by~\ref{item:intersection}, so \( (B^{\top})^{b}j=A^{\top}B^{\top}p \) for some \( p\in\Z^d \); cancelling \( B^{\top} \) gives \( (B^{\top})^{b-1}j=A^{\top}p \), and iterating \( b \) times yields \( j\in A^{\top}\Z^d \), again a contradiction.
    Hence \( a=\nu_A \), and symmetrically \( b=\nu_B \); so \( a=a' \), \( b=b' \) and, cancelling the injective term \( (A^{\top})^{a}(B^{\top})^{b} \) gives \( j=j' \).
\end{proof}

We write
\[
    E_A\coloneqq E_{\varphi_A}=\ker\left(\D_A-\Id\right)\subset L^2_0(\Leb)\cong T_\Leb\Prob(\T^d)
\]
for the space of directions in which \(\Leb\) can be deformed while remaining invariant to first order.
In light of Proposition~\ref{prp:acim.derivative}\ref{item:acim.sharp}, this says exactly that \(\dw\big((\varphi_A)_\ast(\Leb+tv),\Leb+tv\big)=o(t)\).

Recall the dynamical coordinates \( \beta(v)=\{\beta_k(v)\}_{k\neq0} \) were chosen so that \eqref{eq:dynamical.equivariance} holds.
It follows from \eqref{eq:derivative.beta} that \( \D_A \) acts on them by \( \beta_m(\D_Av)=\beta_{A^{\top}m}(v) \).
In these coordinates Proposition~\ref{prp:derivative.functoriality} becomes explicit.
If \( A,B\in\mathrm M_d(\Z) \) are non-singular, then \( \varphi_A\circ\varphi_B=\varphi_{AB} \) and, for \( m\neq0 \),
\[
    \beta_m\left(\D_A\D_Bv\right)
    =\beta_{A^{\top}m}\left(\D_Bv\right)
    =\beta_{B^{\top}A^{\top}m}(v)
    =\beta_{(AB)^{\top}m}(v)
    =\beta_m\left(\D_{AB}v\right),
\]
which is the chain rule \( \D_{\varphi_A\circ\varphi_B}=\D_{\varphi_A}\D_{\varphi_B} \); in particular \( \D_A \) and \( \D_B \) commute whenever \( A \) and \( B \) do.
Likewise, the last identity of Proposition~\ref{prp:derivative.functoriality} reads here
\[
    E_A=\big\{v\in L^2_0(\Leb):\ (\varphi_A)_\ast\iota_\Leb(v)=\iota_\Leb(v)\big\},
\]
so that the elements of \( E_A\cap E_B \) are exactly the tangent directions at \( \Leb \) whose associated first-order variations are invariant under both pushforwards.

\begin{thm}[Joint fixed vectors at \( \Leb \)]
    \label{thm:joint.eigen}
    Let \( A,B\in \mathrm{M}_d(\Z) \) be commuting expanding matrices whose determinants are coprime.
    For a bi-primitive \( j\in\Z^d\setminus\{0\} \) let
    \[
        O_j\coloneqq \big\{(A^{\top})^{a}(B^{\top})^{b}j:\ a,b\ge0\big\}
    \]
    be its orbit.
    Then:
    \begin{enumerate}
        \item\label{item:char}
              \( v\in E_A\cap E_B \) if and only if, in the dynamical coordinates \eqref{eq:beta.coordinates}, the sequence \( \{\beta_k(v)\}_{k\neq0} = \beta(v) \) is constant on each orbit \( O_j \) and satisfies the reality condition \( \beta_{-k}(v)=\overline{\beta_k}(v) \);
        \item\label{item:conv}
              for every bi-primitive \( j \) and every choice of a constant value on \( O_j \), the resulting field lies in \( L^2_0(\Leb) \); 
        \item\label{item:infinite}
              \( E_A\cap E_B \) is infinite dimensional.
              Explicitly, for each bi-primitive \( j \) let
              \[
                  v_j^{(1)}(x)\coloneqq \sum_{k\in O_j}\frac{k}{\pi\lvert k \rvert^{2}}\,
                  \sin\!\left(2\pi\langle k,x\rangle\right),\qquad v_j^{(2)}(x)\coloneqq \sum_{k\in O_j}\frac{k}{\pi\lvert k \rvert^{2}}\, \cos\!\left(2\pi\langle k,x\rangle\right)
              \]
              Then, as \( j \) ranges over a set of representatives of the bi-primitive vectors modulo \( j\sim-j \), the fields \( v_j^{(1)},v_j^{(2)} \) form an infinite family of continuous, pairwise orthogonal, nonzero elements of \( E_A\cap E_B \).
    \end{enumerate}
\end{thm}

\begin{proof}
    \noindent Item~\ref{item:char}. \\
    Our choice of dynamical coordinates \(\beta\) was such that Equation \eqref{eq:dynamical.equivariance} holds, so \( \D_Av=v \) reads \( \beta_{A^{\top}m}=\beta_m \) for all \( m\neq0 \).
    Similarly, \( \beta_{B^{\top}m}=\beta_m \) for all \( m\neq0 \), from \( \D_Bv=v \).
    Hence \( v\in E_A\cap E_B \) if and only if its coordinates \( \beta_v \) are invariant under the two maps \( m\mapsto A^{\top}m \) and \( m\mapsto B^{\top}m \), i.e.\ constant on the equivalence classes of the relation they generate on \( \Z^d\setminus\{0\} \).
    By Lemma~\ref{lem:freeness}\,\ref{item:unique}, those classes are exactly the orbits \( O_j \), \( j \) bi-primitive.
    The reality condition is part of the description of real fields in Proposition~\ref{prp:fourier.model}.

    \medskip
    \noindent Item~\ref{item:conv}.\\
    Fix \(1<\kappa_A<\min\lvert\spec(A)\rvert\) and \(1<\kappa_B<\min\lvert\spec(B)\rvert\), and let \(C_0\coloneqq\max\{C(A,\kappa_A),C(B,\kappa_B)\}\ge1\), so that both matrices satisfy the estimate of Lemma~\ref{lem:orbit.growth} with the same constant.
    Since the matrices commute, applying that lemma twice gives, for \(k=(A^\top)^a(B^\top)^bj\)
    \[
        \lvert k \rvert=\big\lvert (A^{\top})^{a}\left((B^{\top})^{b}j\right)\big \rvert  \ \ge\ C_0^{-1}\kappa_A^{a}\,\big\lvert (B^{\top})^{b}j\big \rvert \ \ge\ C_0^{-2}\kappa_A^{a}\kappa_B^{b}\,\lvert j \rvert .
    \]
    As the map \( (a,b)\mapsto(A^{\top})^{a}(B^{\top})^{b}j \) is injective on \( \N^2 \) by Lemma~\ref{lem:freeness}\,\ref{item:unique}, we get
    \[
        \sum_{k\in O_j}\frac{1}{\lvert k \rvert^{2}}
        \le\sum_{a,b\ge0}\frac{C_0^{4}}{\kappa_A^{2a}\kappa_B^{2b}\lvert j \rvert^{2}}
        =\left(\frac{C}{\lvert j \rvert}\right)^2\cdot\frac{1}{1-\kappa_A^{-2}}\cdot
        \frac{1}{1-\kappa_B^{-2}}<\infty
    \]
    where \(C \coloneqq C_0^2\).
    By \eqref{eq:beta.norm}, a sequence \( \beta_k \) that is bounded and supported on \( O_j\cup(-O_j) \) therefore defines an element \(v \in  L^2_0(\Leb) \) by the assignment \(\beta(v) = \{\beta_k\}_k\).

    \medskip
    \noindent Item~\ref{item:infinite}.\\
    First, we note that the fields \( v_j^{(1)},v_j^{(2)} \) are obtained from \eqref{eq:beta.coordinates} by taking, respectively, \( \beta_k=1 \) and \( \beta_k=i \) for \( k\in O_j \), extending by the reality condition to \( -O_j \), and setting \( \beta_k=0 \) elsewhere.
    This is well defined.
    Indeed, \( -j \) is bi-primitive whenever \( j \) is, because \( A^{\top}\Z^d \) and \(B^{\top}\Z^d \) are subgroups, so \( j\notin B^{\top}\Z^d\cup A^{\top}\Z^d \iff -j\notin B^{\top}\Z^d\cup A^{\top}\Z^d \); and \( O_{-j}=-O_j \).
    Moreover, \( O_j\cap(-O_j)=\emptyset \): an element of \( O_j \) has bi-primitive part \( j \) and an element of \( -O_j=O_{-j} \) has bi-primitive part \( -j\ne j \), so they cannot coincide by the uniqueness in Lemma~\ref{lem:freeness}\,\ref{item:unique}.
    Hence prescribing \( \beta \) on \( O_j \) and extending by \( \beta_{-k}=\overline{\beta_k} \) to \( -O_j \) is consistent, and produces a sequence constant on each of the two orbits, as~\ref{item:char} requires.
    Pairing the terms \( k \) and \( -k \) in the Fourier series and using \eqref{eq:beta.coordinates}, we obtain
    \[
        \frac{k\beta_k}{2\pi i\lvert k \rvert^{2}}e^{2\pi i\langle k,x\rangle}
        +\frac{-k\overline{\beta_k}}{2\pi i\lvert k \rvert^{2}}e^{-2\pi i\langle k,x\rangle}
        =\begin{cases}
            \dfrac{k}{\pi\lvert k \rvert^{2}}\sin\left(2\pi\langle k,x\rangle\right), & \beta_k=1, \\[2mm]
            \dfrac{k}{\pi\lvert k \rvert^{2}}\cos\left(2\pi\langle k,x\rangle\right), & \beta_k=i,
        \end{cases}
    \]
    which gives the displayed real expressions for \( v_j^{(1)} \) and \( v_j^{(2)} \).
    The same geometric bound on \(\lvert k \rvert\) gives \(\sum_{k\in O_j}\lvert k \rvert^{-1}<\infty\), so the series converges absolutely and uniformly.

    We claim there are infinitely many bi-primitive vectors \(j\).
    Indeed, if there were only finitely many, then by Lemma~\ref{lem:freeness}\,\ref{item:unique} the set \( \Z^d\setminus\{0\} \) would be
    a finite union of orbits \( O_j \); but each orbit satisfies \( \sum_{k\in O_j}\lvert k \rvert^{-2}<\infty \) by~\ref{item:conv}, whereas \( \sum_{k\neq0}\lvert k \rvert^{-2}=\infty \) for \( d\ge2 \).
    For \( d=1 \) one checks directly that there are infinitely many bi-primitive integers, namely those
    prime to \( \det A\det B \) \cite[\S7]{kloeckner_optimal_2013_preprint}.

    Finally, the fields \( v_j^{(1)},v_j^{(2)} \) associated with distinct bi-primitive \( j \) (modulo \( \pm \)) have Fourier supports contained in the pairwise disjoint sets \( O_j\cup(-O_j) \), by Lemma~\ref{lem:freeness}\,\ref{item:unique}; hence they are pairwise orthogonal in \( L^2(\Leb) \) by Plancherel, and \( v_j^{(1)}\perp v_j^{(2)} \) for the same \( j \) since \( \sin \) and \( \cos \) of the same frequency are orthogonal.
    In particular they are linearly independent.
    They are nonzero, continuous, lie in \( L^2_0(\Leb) \) by~\ref{item:conv}, and lie in \( E_A\cap E_B \) by~\ref{item:char}.
    We conclude \( E_A\cap E_B \) has infinite dimension, as wanted.
\end{proof}

We can now construct the families of \emph{simultaneously} nearly invariant measures.
We begin with the following lemma, an analogue of \cite[Proposition 1.5]{kloeckner_optimal_2013}.

\begin{lem}
    \label{lem:atomless}
    Let \( \mu\in \Prob(\T^d) \) be atomless and \( v\in L^2(\mu;\R^d) \).
    Then, for all but countably many \( t\in\R \), the measure \( \mu+tv=(\Id+tv)_\ast \mu \) has no atom.
\end{lem}

\begin{proof}
    Let \( \Gamma\coloneqq (\Id,v)_\ast \mu \), a probability measure on \( \T^d\times\R^d \); it is atomless, since \( \mu \) is and \( \Gamma \) is the image of \( \mu \) under an injective map (its first marginal is \( \mu \), so a \( \Gamma \)-atom at \( (x_0,y_0) \) would force a \( \mu \)-atom at \( x_0 \)).
    For \( t\in\R \) and \( s\in\T^d \) set
    \[
        L_{t,s}\coloneqq \big\{(x,y)\in\T^d\times\R^d:\ x+ty=s\ \text{in}\ \T^d\big\}.
    \]
    Then \( \mu+tv \) has an atom at \( s \) if and only if \( \Gamma(L_{t,s})>0 \).

    Fix \( t\neq t' \) and \( s,s' \).
    We claim \( \Gamma(L_{t,s}\cap L_{t',s'})=0 \).
    If \( (x,y)\in L_{t,s}\cap L_{t',s'} \) then, subtracting the two defining relations in \( \T^d \), \( (t-t')y\in s-s'+\Z^d \), i.e.\ \( y \) lies in the countable set \( \frac{1}{t-t'}\left(s-s'+\Z^d\right) \) (a translate of a lattice, as \( t\ne t' \)); and for each such \( y \), \( x=s-ty \) is determined in \( \T^d \).
    Hence \( L_{t,s}\cap L_{t',s'} \) is countable, and \( \Gamma \) being atomless it is \( \Gamma \)-null.

    Now suppose, for contradiction, that for uncountably many \( t \) the measure \( \mu+tv \) has an atom.
    Then there are \( n\in\N \) and an uncountable set \( \mathcal T\subset\R \) such that for each \( t\in\mathcal T \) there is \( s(t) \) with \( \Gamma(L_{t,s(t)})\ge1/n \).
    Pick \( n+1 \) distinct elements \( t_1,\dots,t_{n+1}\in\mathcal T \): the sets \( L_{t_i,s(t_i)} \) pairwise intersect in \( \Gamma \)-null sets by the claim, so
    \[
        1\ge\Gamma\left(\bigcup_{i=1}^{n+1}L_{t_i,s(t_i)}\right)=\sum_{i=1}^{n+1}\Gamma\left(L_{t_i,s(t_i)}\right)\ \ge\ \frac{n+1}{n}>1,
    \]
    a contradiction.
    Hence for each \( n \) only finitely many \( t \) admit an atom of mass \( \ge1/n \), and the set of \( t \) for which \( \mu+tv \) has some atom is a countable union of finite sets.
\end{proof}

The second alternative of Proposition~\ref{prp:ball.embedding} applies on \( \T^d \), with \( \Lambda=1 \) and \( \eta_0=+\infty \).
For linearly independent \( v_1,\dots,v_n\in L^2_0(\Leb) \) the map \( F(a)=\Leb+\sum_ia_iv_i \) is therefore \( R_2 \)-Lipschitz on all of \( \R^n \), metrically differentiable along rays, and a topological embedding on some \( \mathcal B^n_\eta \).

\begin{thm}[Nearly invariant measures on \( \T^d \); cf.\ {\cite[Theorem 1.4]{kloeckner_optimal_2013}}]
    \label{thm:families}
    Let \( A,B\in \mathrm{M}_d(\Z) \) be commuting expanding matrices with coprime determinants.
    Then for every \( n\in\N \) there are \( \eta>0 \) and continuous linearly independent fields \(\{v_1, \dots, v_n\} \subset L_0^2(\Leb)\) such that the mapping
    \begin{align*}
        F:\mathcal{B}^n_\eta & \lto  \Prob(\T^d)  \\
        a = (a_1,\dots,a_n)    & \longmapsto \Leb+\sum_ia_iv_i
    \end{align*}
    is Lipschitz and a topological embedding of the open ball \( \mathcal{B}^n_\eta \subset \R^n \), satisfying
    \begin{enumerate}
        \item\label{item:families.atomless} \( F(a) \) is atomless for almost every \( a\in \mathcal{B}_\eta^n \); more precisely, along every ray through \( 0 \) all but countably many \( F(a) \) are atomless;
        \item\label{item:families.invariant} \( F \) is nearly invariant under both endomorphisms:
        \begin{align*}
            \dw\left((\varphi_A)_\ast F(a),F(a)\right) &= o(\lvert a \rvert) \qquad (a\to0), \\
            \dw\left((\varphi_B)_\ast F(a),F(a)\right) &= o(\lvert a \rvert) \qquad (a\to0).
        \end{align*}
        Consequently, for every \( \epsilon>0 \) and every \( K\in\N \) there is \( r>0 \) such that, for all \( k\le K \) and all \( a\in \mathcal{B}_r^n \),
        \begin{align*}
            \dw\left(( \varphi_A )^k_\ast F(a),F(a)\right)&\le\epsilon\lvert a \rvert, \\
            \dw\left(( \varphi_B )^k_\ast F(a),F(a)\right)&\le\epsilon\lvert a \rvert.
        \end{align*}
    \end{enumerate}
\end{thm}

Whether \(F\) is bi-Lipschitz we do not know; see Section~\ref{sec:conclusion}.

\begin{proof}
    Using Theorem~\ref{thm:joint.eigen}\,\ref{item:infinite}, we choose continuous, linearly independent \( v_1,\dots,v_n\in E_A\cap E_B \) (e.g.\ \( n \) of the fields \( v_j^{(1)} \)).
    These fields lie in \(L^2_0(\Leb)\) and \(\T^d\) is flat, so Proposition~\ref{prp:ball.embedding} applies with \(\eta_0=+\infty\) and gives \(\eta>0\) such that \(F|_{\mathcal B^n_\eta}\) is a topological embedding, with \(F(0)=\Leb\).

    \medskip
    \noindent Item~\ref{item:families.atomless}.\\
    Fix a direction \( e\in S^{n-1} \) and apply Lemma~\ref{lem:atomless} with the atomless measure \( \Leb \) and the field \( v=\eta\sum_ie_iv_i \): for all but countably many \( s\in\R \), the measure \( \Leb+sv=F(se) \) is atomless.
    Fubini in polar coordinates then gives that \( F(a) \) is atomless for Lebesgue-a.e.\ \( a\in \mathcal{B}^n_\eta \), as we wanted.

    \medskip
    \noindent Item~\ref{item:families.invariant}. \\
    Let \(V\coloneqq \operatorname{span}\{v_1,\dots,v_n\}\subset L^2_0(\Leb)\) and let \(R\coloneqq \eta\max_i\lVert v_i \rVert_{L^2} \cdot n^{1/2}\), so that \( \eta\sum_ia_iv_i\in V \) has norm \( \le R \) for \( a\in \mathcal{B}^n_\eta \).
    Write \( a=\lvert a \rvert e \) with \( e\in S^{n-1} \) and \( v_e\coloneqq \eta\sum_ie_iv_i\in V \), so \( F(a)=\Leb+\lvert a \rvert v_e \) and \( \lVert v_e \rVert_{L^2}\le R \).
    Proposition~\ref{prp:derivative.torus}, applied with the finite-dimensional subspace \( V \) and the bound \( R \), gives
    \[
        \dw\left(( \varphi_A )_\ast \left(\Leb+\lvert a \rvert v_e\right),\ \Leb+\lvert a \rvert\,\D_A v_e\right)= o(\lvert a \rvert)\qquad\text{uniformly in } e\in S^{n-1} .
    \]
    But \( v_e\in E_A \), i.e.\ \( \D_Av_e=v_e \), because \( E_A \) is a linear subspace containing each \( v_i \).
    Hence \( \dw\left(( \varphi_A )_\ast F(a),\ \Leb+\lvert a \rvert v_e\right)=\dw\left(( \varphi_A )_\ast F(a),F(a)\right) =o(\lvert a \rvert) \), uniformly in the direction, which is exactly the assertion.
    The same argument with \( B \) in place of \( A \) gives the second estimate\footnote{
        It is here, and only here, that the uniformity in Proposition~\ref{prp:derivative.torus} is used: without it, the \( o(\lvert a \rvert) \) would hold along each fixed ray but not necessarily uniformly, and~\ref{item:families.invariant} could fail.
    }.

    Finally, let \( L_A\coloneqq \Lip( \varphi_A )=\lVert A \rVert_{\mathrm{op}} \).
    Then the map \( ( \varphi_A )_\ast  \) is \( L_A \)-Lipschitz on \( (\Prob(\T^d),\dw) \).
    Indeed, if \( \Pi \) is a plan between \( \mu,\nu \) then \( ( \varphi_A \times  \varphi_A )_\ast \Pi \) is a plan between the images, of cost at most \( L_A^2 \) times that of \( \Pi \).
    We may assume \( L_A>1 \).
    Given \( \epsilon>0 \) and \( K \), use~\ref{item:families.invariant} to pick \( r>0 \) with
    \[
        \lvert a \rvert<r\ \lto \
        \dw\left(( \varphi_A )_\ast F(a),F(a)\right)\le\frac{L_A-1}{L_A^{K}-1}\,\epsilon\,\lvert a \rvert .
    \]
    Then, for \( k\le K \) and \( \lvert a \rvert<r \), by the triangle inequality along the orbit and the Lipschitz bound, we get
    \[
        \begin{split}
            \dw\left(( \varphi_A )^k_\ast F(a),F(a)\right) & \le\sum_{\ell=0}^{k-1}\dw\left(( \varphi_A )^{\ell+1}_\ast F(a),( \varphi_A )^{\ell}_\ast F(a)\right) \\
                                                        & \le\sum_{\ell=0}^{k-1}L_A^{\ell}\,\dw\left(( \varphi_A )_\ast F(a),F(a)\right),
        \end{split}
    \]
    and since \(\sum_{\ell=0}^{k-1}L_A^{\ell}=\frac{L_A^{k}-1}{L_A-1} \le\frac{L_A^{K}-1}{L_A-1}\), the right-hand side is \( \le\epsilon\lvert a \rvert \).
    Same arguments hold for \( B \); taking the smaller of the two radii gives simultaneous statement.
\end{proof}

\begin{proof}[Proof of Theorem~\ref{thm:main.application}]
    Theorems~\ref{thm:joint.eigen} and~\ref{thm:families} give all the assertions but the last, after rescaling \(\mathcal B^n_\eta\) to the unit ball by \(a\mapsto\eta a\).
    
    We are left with the last assertion. 
    Were the set of \( \varphi_A \)-ergodic components of \( F(a) \) with zero entropy of measure zero, then Host's generalization of Furstenberg's Conjecture (Theorem~\ref{thm:host.rigidity}), applied with \( B \) in the role of the matrix satisfying \ref{hyp:host.irreducible} and \( A \) in that of the expanding one, would force \( F(a)=\Leb \).
    But \( F \) is injective by Theorem~\ref{thm:families}, so \( a=0 \), a contradiction.
\end{proof}

\begin{cor}[cf.\ {\cite[Corollary 1.8]{kloeckner_optimal_2013_preprint}}]
    \label{cor:weak}
    Let \( A,B\in \mathrm{M}_d(\Z) \) be commuting expanding matrices whose determinants are coprime.
    Then there exists a weakly continuous path \( (\mu_t)_{t\in(-\epsilon,\epsilon)} \) of probability measures on \( \T^d \) with \( \mu_0=\Leb \), with \( \mu_t \) atomless for almost every \( t \), such that
    \[
        \frac{d}{dt}\Big|_{t=0}\int_{\T^d}\psi_0\,d\mu_t\neq0\qquad\text{for some }\psi_0\in C^\infty(\T^d),
    \]
    and such that, for every \( \psi\in C^\infty(\T^d) \),
    \[
        \frac{d}{dt}\Big|_{t=0}\int_{\T^d}\psi\,d\mu_t
        =\frac{d}{dt}\Big|_{t=0}\int_{\T^d}\psi\,d\left(( \varphi_B )_\ast \mu_t\right)
        =\frac{d}{dt}\Big|_{t=0}\int_{\T^d}\psi\,d\left(( \varphi_A )_\ast \mu_t\right).
    \]
    Equivalently
    \( \frac{d}{dt}\big\lvert _{0}\int\psi\,d\mu_t=\frac{d}{dt}\big \rvert_{0}\int\psi\circ  \varphi_B \,d\mu_t =\frac{d}{dt}\big|_{0}\int\psi\circ  \varphi_A \,d\mu_t.\)
\end{cor}

\begin{proof}
    Using Theorem~\ref{thm:joint.eigen} pick \( v\in E_A\cap E_B \), \( v\neq0 \), and set \( \mu_t\coloneqq \Leb+tv \).
    Weak continuity is clear from \eqref{eq:trivial.coupling} and the fact that \( \dw \) metrises the weak\( ^\ast \) topology; atomlessness for a.e.\ \( t \) is Lemma~\ref{lem:atomless}.

    Since \( v\in L^2_0(\Leb) \), there is \( \psi_0\in C^\infty(\T^d) \) with \( \lVert \nabla\psi_0-v \rVert_{L^2(\Leb)}\le\frac12\lVert v \rVert_{L^2(\Leb)} \), whence
    $\int\langle\nabla\psi_0,v\rangle d\Leb\ge\lVert v \rVert^2_{L^2}-\frac12\lVert v \rVert^2_{L^2} =\frac12\lVert v \rVert^2_{L^2}>0$.
    By Lemma~\ref{lem:pointwise.continuity} above, $\frac{d}{dt}\big|_0\int\psi_0\,d\mu_t=\int\langle\nabla\psi_0,v\rangle d\Leb \neq0$.

    It follows from Proposition~\ref{prp:derivative.torus} and the equality \( \D_Av=v \) that the curve \( t\mapsto( \varphi_B )_\ast \mu_t \) is differentiable at \( 0 \) with the \emph{same} tangent vector \( v \) at the same base point \( \Leb \), i.e., it satisfies \( \dw\left(( \varphi_B )_\ast \mu_t,\Leb+tv\right)=o(t) \).
    Lemma~\ref{lem:pointwise.continuity} applied to both curves gives, for every \( \psi\in C^\infty(\T^d) \),
    \[
        \frac{d}{dt}\Big|_{0}\int\psi\,d\left(( \varphi_B )_\ast \mu_t\right)
        =\int\langle\nabla\psi,v\rangle\,d\Leb
        =\frac{d}{dt}\Big|_{0}\int\psi\,d\mu_t ,
    \]
    and the same with \( B \).
    The reformulation follows from \( \int\psi\,d\left(( \varphi_B )_\ast \mu_t\right)=\int\psi\circ  \varphi_B \,d\mu_t \).
\end{proof}

\begin{rmk}\label{rmk:weaker}
    As observed \cite[Remark 1.9(2)]{kloeckner_optimal_2013_preprint}, this corollary is intrinsically weaker than Theorem~\ref{thm:main.application}: \( \dw \)-differentiability implies differentiability of the integrals of test functions, but not conversely, since a curve \( t\mapsto t\mu+(1-t)\nu \) has affine test integrals while being non-rectifiable for \( \dw \).
    Its interest is that its statement involves no optimal transport at all.
\end{rmk}

\section{Conclusion}
    \label{sec:conclusion}
    Sections~\ref{sec:setting} and~\ref{sec:derivatives} identify the derivative of \( \phi_\ast \) at an absolutely continuous invariant measure as \( \D_\phi=P_{\mu_0}\Transf_0 \), the adjoint of the Koopman operator restricted to the tangent space, and by Proposition~\ref{prp:acim.derivative}\ref{item:acim.sharp} the directions in which \( \mu_0 \) may be deformed without losing invariance at first order are exactly the fixed vectors of that operator.
The invariance rigidity of \( \mu_0 \) thus becomes a question about \( E_\phi \), which we could not answer in the general setting.

Section~\ref{sec:torus}, on the other hand, answers it for the linear expanding endomorphisms of \( \T^d \), where the flat geometry turns \( \D_A \) into a composition operator on the frequency lattice and the arithmetic of Lemma~\ref{lem:freeness} takes the place of spectral theory.
The answer is that the space \( E_A\cap E_B \) is infinite dimensional for every admissible pair (cf. Appendix~\ref{sec:pairs}), and such pairs exist in every dimension by Lemma~\ref{lem:existence.pairs}.
Thus, for every admissible pair, first-order rigidity fails, and whatever mechanism makes Theorem~\ref{thm:host.rigidity} true leaves no trace on the tangent space at \(\Leb\).

We conclude with the questions left open by the argument.

\medskip
\paragraph*{\textbf{Fr\'echet derivatives}.}
The example in Remark~\ref{rmk:frechet.derivative} transplants to \( A=2\Id \) on \( \T^d \): reading \( v \) as a field depending on the first coordinate alone yields a field in \( L^2_0(\Leb) \) with the same norm and still satisfying \( \D_Av=0 \).
The deformed measure and its image are then products whose first factors are those of the circle example, and pushing forward by the projection onto the first coordinate, which is \( 1 \)-Lipschitz, does not increase \( \dw \).
No Fr\'echet statement is thus available in any dimension.

The use of finite dimensionality in the proof of Theorem~\ref{thm:derivative.acim} is confined to Lemma~\ref{lem:uniform.smoothing}, since the remainder of Proposition~\ref{prp:linear.response} is uniform on compact subsets of \( C^1(M;TM) \).
Indeed, a \( C^1 \)-bounded family can never be \( \epsilon \)-dense in the unit ball of \( L^2_0(\mu_0) \), since by Arzel\`a--Ascoli such a family is totally bounded in \( L^2(\mu_0;TM) \), whereas that ball is not.

What remains open is the suggestion of \cite[\S4.3]{kloeckner_optimal_2013}, namely that the fields above have large total variation and that a uniform remainder might be recovered on a subspace of \( L^2_0(\mu_0) \) carrying a norm which controls derivatives.

\medskip
\paragraph*{\textbf{Bi-Lipschitz embeddings}.}
Proposition~\ref{prp:ball.embedding} embeds \( \mathcal B^n_\eta \) topologically in \( \Prob(M) \), and Lipschitz-continuously when the fields are bounded or \( M \) is flat, whereas on the circle \cite[Proposition 3.1]{kloeckner_optimal_2013} obtains a bi-Lipschitz embedding.
The lower bound there is obtained by displacing \( F(a) \) along a field \( \tilde w \) built by precomposing gradients with a diffeomorphism, and by evaluating \( \dw\big(\nu,(\exp\tilde w)_\ast\nu\big) \) as \( \lVert\tilde w\rVert_{L^2(\nu)} \), an identity available on the circle because \( \exp\tilde w \) is monotone there.
    
Neither step survives in dimension \( d\ge2 \): optimality would require \( \exp\tilde w \) to be a gradient, which precomposition destroys, and for \( \nu \) with positive \( C^1 \) density Proposition~\ref{prp:solenoidal.dual} shows that the evaluation already fails at first order as soon as \( P_\nu\tilde w\neq\tilde w \).
We do not know whether \( F \) is bi-Lipschitz on some \( \mathcal B^n_\eta \) for \( d\ge2 \).

\medskip
\paragraph*{\textbf{Invariant directions}.}

The fixed space \( E_\phi=\ker(\D_\phi-\Id) \) is an eigenspace, so asking whether \( \mu_\phi \) admits deformations preserving invariance at first order is asking whether \( 1 \) belongs to the point spectrum of \( \D_\phi \).
If \( \phi \) and \( \psi \) commute and preserve the same measure, then \( \D_\phi\D_\psi=\D_\psi\D_\phi \) by Proposition~\ref{prp:derivative.functoriality}, so \( \D_\psi \) leaves \( E_\phi \) invariant and
\[
    E_\phi\cap E_\psi=\ker\big(\D_\psi|_{E_\phi}-\Id\big).
\]

\begin{qst}
\label{qst:general.eigen}
    Let \( \phi \) be a \( C^2 \) expanding map of a closed manifold \( M \), with absolutely continuous invariant measure \( \mu_\phi \) of positive \( C^1 \) density.
    Is \( E_\phi \) infinite dimensional?
    If \( \psi \) is a second such map, commuting with \( \phi \) and preserving \( \mu_\phi \), is \( E_\phi\cap E_\psi\neq\{0\} \)?
\end{qst}

A positive answer to these questions would extend the construction of Section~\ref{sec:torus} from linear toral endomorphisms to the general setting of Section~\ref{sec:derivatives}.

For a single map the answer is affirmative on the circle, where Kloeckner deduces it \cite[Proposition 5.3]{kloeckner_optimal_2013} from the spectral theory of weighted transfer operators \cite[Theorem 2.5]{baladi_positive_2000}, and our results give the same answer on \( \T^d \) for the linear endomorphisms \( \varphi_A \).
In neither case does the argument single out the value \( 1 \).
Indeed, by \eqref{eq:derivative.beta} the operator \( \D_A \) acts on the coordinates \( \beta \) by precomposition with \( A^{\top} \), that is, it moves each coordinate one step backwards along the forward orbits of \( A^{\top} \) on \( \Z^d\setminus\{0\} \), so the eigenvalue equation \( \D_Av=\alpha v \) reads \( \beta_{A^{\top}m}=\alpha\beta_m \) and forces \( \beta_{(A^{\top})^am}=\alpha^a\beta_m \) along each orbit.
Writing \( \theta\coloneqq\min\lvert\spec(A)\rvert \), Lemma~\ref{lem:orbit.growth} makes \( \lvert(A^{\top})^am\rvert \) grow at least like \( \theta^a \), so by \eqref{eq:beta.norm} the resulting field lies in \( L^2_0(\Leb) \) as soon as \( \lvert\alpha\rvert<\theta \).
Every \( \alpha \) with \(\lvert \alpha \rvert < \theta\) is thus an eigenvalue of \( \D_A \) of infinite multiplicity.
Thus the first half of Question~\ref{qst:general.eigen} amounts to determining whether the point spectrum contains a disc of radius larger than \(1\).

For \( d=1 \) and \( A=(d_0) \), \( d_0\ge2 \), one recovers \cite[Proposition 4.4]{kloeckner_optimal_2013}, where the eigenvalues are exactly the open disc of radius \( d_0 \) and the spectrum is the closed disc, the proof conjugating \( \D_A \) to a countable product of shifts.
For general \( A \) the orbit growth rates depend on the direction, so there is no obvious analogue to such conjugation, and the spectrum of \( \D_A \) remains undetermined\footnote{
    That, however, is a finer question than Question~\ref{qst:general.eigen}, which needs only the eigenvalue \( 1 \).
}.

For a nonlinear expanding map of \( \T^d \), or of any closed manifold of dimension at least two, the circle arguments have no immediate analogue.
They rely on the positivity of \( \Transf_0 \) on \( C^0(S^1) \), which holds because \( TS^1 \) is trivial and \( \phi'>0 \), so that \( \Transf_0 \) preserves the cone of non-negative fields.
For \( \dim M\ge2 \) the fibrewise average \eqref{eq:transfer.disintegration} is weighted by \( D_x\phi \), the fibres of \( TM \) carry no canonical cone, and an invariant cone field for \( D\phi \) is an additional hypothesis we do not have.
What seems to be missing is a spectral theory for transfer operators acting on sections of a vector bundle rather than on functions.

The second half depends genuinely on both maps, and not only on their spectra.
Two commuting operators may each have \( 1 \) as an eigenvalue of infinite multiplicity without sharing a single eigenvector, and what \( \D_\psi|_{E_\phi} \) does depends on \( E_\phi \) as a concrete object rather than on its dimension.
On \( \T^d \) the relation between maps is Lemma~\ref{lem:freeness}, which decomposes \( \Z^d\setminus\{0\} \) into the orbits of the semigroup generated by \( A^{\top} \) and \( B^{\top} \) and thereby describes \( E_A \) coordinate by coordinate.
We know of no substitute for that decomposition when there is no frequency lattice.

There is, however, one class of manifolds on which the setting that makes Section~\ref{sec:torus} work is still available.
If \( M=G \) is a compact Lie group, then \( \Prob(G) \) is parallelisable in the sense of \cite[\S4]{gomes_differential_2024}, \cite{gomes_riemannian_2025}.
The role played on \( \T^d \) by the characters is played on \( G \) by the Peter--Weyl decomposition \( L^2(G)=\bigoplus_{\pi\in\hat G}\mathcal H_\pi \), an orthogonal sum closed in \( L^2 \).
Here, \( \hat G \) is the set of irreducible unitary representations of \( G \) up to equivalence and the \emph{isotypic component} \( \mathcal H_\pi \) is the span of the matrix coefficients \( g\mapsto\langle\pi(g)u,v\rangle \), of dimension \( (\dim\pi)^2 \).
For \( G=\T^d \) one has \( \hat G=\Z^d \) and each \( \mathcal H_k \) is the line spanned by \( e^{2\pi i\langle k,x\rangle} \), so \( \mathcal H_\pi \) is what replaces a single Fourier mode when the group is not Abelian.

Write \( \lambda_\pi \) for the eigenvalue of \( -\Delta \) on \( \mathcal H_\pi \), which is a scalar because the metric is bi-invariant, and which is \( 4\pi^2\lvert k\rvert^2 \) on \( \T^d \).
Let now \( \phi \) be an expanding endomorphism of \( G \).
On gradients one has \( \Koop_0\nabla\psi=\nabla(\psi\circ\phi) \), so through \eqref{eq:L20} the operator \( \Koop_0 \) is the Koopman operator of \( \phi \) acting on scalars.
It sends a matrix coefficient of \( \pi \) to the corresponding matrix coefficient of \( \pi\circ\phi \), and \( \pi\circ\phi \) is again irreducible because \( \phi \), being a covering map, is surjective, so that \( \pi \) and \( \pi\circ\phi \) have the same image and hence the same commutant. 
The same surjectivity makes \( \pi\mapsto\pi\circ\phi \) injective on \( \hat G \).
Thus \( \Koop_0 \) carries \( \mathcal H_\pi \) onto \( \mathcal H_{\pi\circ\phi} \). Its adjoint \( \D_\phi \) carries \( \mathcal H_{\pi\circ\phi} \) back to \( \mathcal H_\pi \), multiplied by \( \lambda_{\pi\circ\phi}/\lambda_\pi \), and annihilates every isotypic component outside the range of \( \pi\mapsto\pi\circ\phi \).
An operator of this shape, moving data one step along the orbits of a map of the index set and rescaling as it goes, is a weighted shift, and \eqref{eq:derivative.beta} is the case \( G=\T^d \), where the renormalisation \eqref{eq:beta.coordinates} was chosen precisely so that the weight becomes \( 1 \).

For endomorphisms of compact Lie groups, this would reduce the first half of Question~\ref{qst:general.eigen} to estimating the growth of \( \lambda_{\pi\circ\phi} \) along the orbits of \( \pi\mapsto\pi\circ\phi \), providing an analogue of Lemma~\ref{lem:orbit.growth}.
This approach does not address manifolds that are not groups, or expanding maps of\( G \) that are not endomorphisms.
It also leaves the second half of the question unchanged: for a commuting pair, one would still need to understand the joint orbit structure of the induced maps on \( \hat G \) in analogy with Lemma~\ref{lem:freeness}.
We have not carried any of this out.

%%%%%%%% APPENDICES
\appendix
\section{Estimates on the exponential map}
    \label{sec:riemannian.estimates}
    This section contains the proof of Lemma~\ref{lem:exp.estimates}, with the notational conventions from the beginning of Section~\ref{sec:exp.deformations}.

We start by remarking the following form of the triangle inequality
    \begin{equation}\label{eq:triangle.inequality}
        d\left(\exp_x\xi,\exp_x\eta\right)\le\le\lvert\xi\rvert+\lvert\eta\rvert,\qquad \xi,\eta\in T_xM.
    \end{equation}
Notice that this holds for \emph{all} \(\xi,\eta\in T_xM\), however large.
It is indeed simply the triangle inequality, as the curve \( s\mapsto\exp_x(s\xi) \), \( s\in[0,1] \), joining \( x \) to \( \exp_x\xi \) has length \( \lvert\xi\rvert \), so that \( d(x,\exp_x\xi)\le\lvert\xi\rvert \).

\begin{proof}[Proof of Lemma~\ref{lem:exp.estimates}]
    Fix \(x \in M\).
    Our argument is based on the use of normal charts around \(x\).
    Write \( R\coloneqq \max(\lvert\xi\rvert,\lvert\eta\rvert)\le r_1 \), let \( g^x\coloneqq (\exp_x)^\ast g \) be the pullback of the metric to \( T_xM \), and identify \( T_p(T_xM)\cong T_xM \) canonically, so that the differential \( D(\exp_x)_p \) is a linear map \( T_xM\to T_{\exp_xp}M \).

    Consider the smooth function \(G(x,p,w)\coloneqq \big\lvert D(\exp_x)_p\,w\big\rvert^2\) on \(TM \oplus TM\).
    For fixed \( (x,w) \) the function \( p\mapsto G(x,p,w) \) satisfies
    \[
        G(x,0,w)=\lvert w\rvert^2,\qquad \partial_pG(x,0,w)=0 ,
    \]
    the first because \( D(\exp_x)_0=\Id \), the second because the Christoffel symbols vanish at the centre of a geodesic normal chart (remark both derivatives here are taken in the fibre \( T_xM \), which is a vector space, so they are intrinsic; no choice of frame is involved).  
    Taylor's theorem along the ray \( s\mapsto(x,sp,w) \) therefore gives
    \[
        \big\lvert G(x,p,w)-\lvert w\rvert^2\big\rvert
        \le\tfrac12\left(\sup_{\mathcal K}\big\lVert\partial^2_pG\big\rVert\right)\lvert p\rvert^2 ,
    \]
    where \(\mathcal K\coloneqq \{(x,p,w):\lvert p\rvert\le4r_1,\ \lvert w\rvert\le1\}\) and the supremum is finite because \( \partial^2_pG \) is continuous and \( \mathcal K\subset TM\oplus TM \) is compact.  
    Calling \( c_M \) that supremum and using homogeneity in \( w \), we get \( \big\lvert G(x,p,w)-\lvert w\rvert^2\big\rvert\le c_M\lvert p\rvert^2\lvert w\rvert^2 \) for all \( w \).
    Factoring the difference of squares gives us, for \(\lvert p\rvert\le4r_1\)
    \begin{equation}\label{eq:infinitesimal.comparison}
    \begin{split}
        \Big\lvert\;\big\lvert D(\exp_x)_p\,w\big\rvert-\lvert w\rvert\;\Big\rvert &= \frac{\big\lvert G(x,p,w)-\lvert w\rvert^2\big\rvert}{\big\lvert D(\exp_x)_p\,w\big\rvert+\lvert w\rvert}\\
        &\le \frac{\big\lvert G(x,p,w)-\lvert w\rvert^2\big\rvert}{\lvert w\rvert} \\ &\le\ c_M\lvert p\rvert^2\lvert w\rvert. 
    \end{split}
    \end{equation}
    Equivalently: on \( B(0,4r_1)\subset T_xM \) the metric \( g^x \) is within a factor \( 1\pm c_M\lvert p\rvert^2 \) of the Euclidean one. 
    This translates to distances on \(T_xM\) in the following way.
    For \( 0<r\le3r_1 \) let \( d^x_r \) denote the length distance of \( \left(\overline B(0,r),g^x\right) \).
    We claim that, for \( \xi,\eta\in\overline B(0,r) \),
    \begin{equation}\label{eq:ball.comparison}
        \left(1-c_Mr^2\right)\lvert\xi-\eta\rvert
        \ \le\ d^x_r(\xi,\eta)
        \ \le\ \left(1+c_Mr^2\right)\lvert\xi-\eta\rvert .
    \end{equation}
    For the upper bound, the Euclidean segment from \( \xi \) to \( \eta \) stays in \( \overline B(0,r) \) by convexity, and \eqref{eq:infinitesimal.comparison} bounds its \( g^x \)-length by \( (1+c_Mr^2)\lvert\xi-\eta\rvert \). 
    For the lower bound, we may assume \( c_Mr^2<1 \), the claim being vacuous otherwise.
    Then every Lipschitz curve in \( \overline B(0,r) \) joining \( \xi \) to \( \eta \) has, by \eqref{eq:infinitesimal.comparison}, \( g^x \)-length at least \( (1-c_Mr^2) \) times its Euclidean length, hence at least \( (1-c_Mr^2)\lvert\xi-\eta\rvert \).
    The distance \(d^x_r(\xi,\eta)\) is the infimum of all such lengths, and the claim follows.  
    Note that no curve here is required to be geodesic.
 
    Finally, we translate the bounds above to distances in \(M\).
    Since \( 4r_1<r_0 \), the map \( \exp_x \) restricts to an isometry from \( \left(B(0,4r_1),g^x\right) \) onto \( \left(B(x,4r_1),g\right) \), and \( \lvert\exp_x^{-1}y\rvert=d(x,y) \) there by the Gauss lemma.
    Let \( \gamma:[0,1]\to M \) be a minimising geodesic from \( \exp_x\xi \) to \( \exp_x\eta \).    
    From the triangle inequality~\eqref{eq:triangle.inequality}, \( L_g(\gamma)=d(\exp_x\xi,\exp_x\eta)\le\lvert\xi\rvert+\lvert\eta\rvert\le2r_1 \), so for every \( s \)
    \[
        d\left(x,\gamma(s)\right)\le d(x,\exp_x\xi)+L_g(\gamma)\le3r_1<4r_1 ;
    \]
    hence \( \widetilde\gamma\coloneqq \exp_x^{-1}\circ\,\gamma \) is a well-defined curve in \( \overline B(0,3r_1) \) joining \( \xi \) to \( \eta \).  
    The upper bound in \eqref{eq:ball.comparison}, applied with \( r=R \), gives
    \[
        \lvert\widetilde\gamma(s)\rvert=d\left(x,\gamma(s)\right)
        \le\lvert\xi\rvert+\left(1+c_MR^2\right)\lvert\xi-\eta\rvert
        \le\left(3+2c_MR^2\right)R\le C_2R,
    \]
    where \( C_2\coloneqq 3+2c_Mr_1^2 \).
    Put \( r\coloneqq \min\{C_2R,\,3r_1\} \).  
    Then \( L_g(\gamma) = L_{g^x}(\widetilde\gamma) \ge d^x_r(\xi,\eta) \), so \eqref{eq:ball.comparison} and \( r\le C_2R \) give
    \[
        d(\exp_x\xi,\exp_x\eta)\ \ge\ d^x_r(\xi,\eta)
        \ \ge\ \left(1-c_MC_2^2R^2\right)\lvert\xi-\eta\rvert .
    \]
    Putting this together with the upper bound of \eqref{eq:ball.comparison} evaluated at \( r=R \), we obtain
    \[
        \left(1-c_MC_2^2R^2\right)\lvert\xi-\eta\rvert \le d(\exp_x\xi,\exp_x\eta)\le (1+c_MR^2)\lvert\xi-\eta\rvert,
    \]  
    and by setting \( C_\ast\coloneqq c_MC_2^2=c_M(3+2c_Mr_1^2)^2 \) in the above inequalities we get \( \big\lvert\,d(\exp_x\xi,\exp_x\eta)-\lvert\xi-\eta\rvert\,\big\rvert \le C_\ast\max(\lvert\xi\rvert,\lvert\eta\rvert)^2\lvert\xi-\eta\rvert \) for \(\lvert \xi \rvert, \lvert \eta\rvert \le r_1\), as we wanted.
    Finally, take \(\Lambda\coloneqq 1+C_\ast r_1^2\) to obtain \(\exp_x: B(0, r_1) \to M\) \(\Lambda\)-Lipschitz for any \(x\in M\).
\end{proof}

\section{Host's theorem and admissible pairs}
    \label{sec:pairs}
    We recall the hypotheses of \cite{host_uniform_2000} verbatim.
For matrices \( A, B \in\mathrm M_d(\Z) \), we consider the following properties.
\begin{enumerate}[label=\textup{(H\arabic*)}, ref=(H\arabic*), start=1]
    \item \label{hyp:host.irreducible}
        The characteristic polynomial of \(A^r\) is irreducible over \(\Q\) for every integer \(r>0\).
    \item \label{hyp:host.expanding}
        All eigenvalues of \(A\) have modulus \(>1\).
    \item \label{hyp:host.coprime}
        For a pair \(A,B\in\mathrm{M}_d(\Q)\), the integers \(\det A\) and \(\det B\) are relatively prime.
\end{enumerate}
Observe that \ref{hyp:host.irreducible} forces \( \det A\neq0 \).
Host's rigidity result is the following.

\begin{thm}[Host, {\cite[Theorem 3]{host_uniform_2000}}]
\label{thm:host.rigidity}
    Let \( A,B\in\mathrm M_d(\Z) \) be such that \( A \) satisfies \ref{hyp:host.expanding}, the pair \( (A,B) \) satisfies \ref{hyp:host.coprime} and \( B \) satisfies \ref{hyp:host.irreducible}.
    Let \( \mu \) be a probability measure on \( \T^d \), invariant under \( \varphi_A \) and \( \varphi_B \), and such that almost all of its \( \varphi_A \)-ergodic components have positive entropy.
    Then \( \mu=\Leb \).
\end{thm}
This result assumes neither that \( A \) and \( B \) commute nor that \( \mu \) is ergodic for the pair.
Moreover, unlike Furstenberg's conjecture, it carries no atomlessness hypothesis, which is replaced by the entropy assumption.

\begin{dft}
\label{dft:admissible.pair}
    A pair \( (A,B)\in\mathrm M_d(\Z)^2 \) is \emph{admissible} if \( A \) and \( B \) commute, both satisfy \ref{hyp:host.expanding}, the pair satisfies \ref{hyp:host.coprime}, and \( B \) satisfies \ref{hyp:host.irreducible}.
\end{dft}

The first three conditions are the ones our own results use, in Theorems~\ref{thm:joint.eigen} and~\ref{thm:families}; the fourth is needed only to invoke Theorem~\ref{thm:host.rigidity} at the end of the proof of Theorem~\ref{thm:main.application}.
Hypothesis \ref{hyp:host.irreducible} is the first of the three conditions in Berend's characterisation \cite[Theorem 2.1]{berend_multi_1983} of the commutative semigroups of endomorphisms of \( \T^d \) whose only infinite closed invariant subset is \( \T^d \) itself, and Host observes \cite[\S1.2.2]{host_uniform_2000} that some assumption of this kind is needed to exclude obvious counterexamples.
For \( d\ge2 \) a scalar matrix \( \ell\,\Id \) fails \ref{hyp:host.irreducible}, its characteristic polynomial \( (x-\ell)^d \) being reducible, so no pair of scalar matrices is admissible.
Nothing prevents a scalar matrix from occupying the first slot, and the construction below exploits precisely that.
Theorem~\ref{thm:joint.eigen} makes no use of \ref{hyp:host.irreducible}, and therefore applies to a strictly larger class of pairs than the rigidity statement it obstructs.

To verify that \ref{hyp:host.irreducible} is satisfied we employ the following criterion, whose mechanism stems from \cite[\S6]{berend_multi_1983}.

\begin{lem}
\label{lem:irreducibility.criterion}
    Let \( B\in\mathrm M_d(\Z) \) have characteristic polynomial irreducible over \( \Q \), with roots \( \beta_1,\dots,\beta_d \).
    If no quotient \( \beta_i/\beta_j \) with \( i\neq j \) is a root of unity, then \( B \) satisfies \ref{hyp:host.irreducible}.
    In particular this holds whenever the \( \beta_i \) have pairwise distinct moduli.
\end{lem}

\begin{proof}
    Fix an integer \( r>0 \).
    The characteristic polynomial of \( B^r \) is \( \prod_i(x-\beta_i^r) \), and the hypothesis makes the \( \beta_i^r \) pairwise distinct.
    The Galois group of the splitting field of the characteristic polynomial of \( B \) acts transitively on \( \{\beta_1,\dots,\beta_d\} \), hence on \( \{\beta_1^r,\dots,\beta_d^r\} \), so the latter is a single Galois orbit of \( d \) distinct algebraic integers and \( \prod_i(x-\beta_i^r) \) is the minimal polynomial of \( \beta_1^r \).
    Being a minimal polynomial, it is irreducible.
    The last assertion follows because \( \lvert\beta_i/\beta_j\rvert\neq1 \) for \( i\neq j \).
\end{proof}

\begin{lem}
\label{lem:existence.pairs}
    Admissible pairs exist in \( \mathrm M_d(\Z)^2 \) for every \( d\ge1 \).
\end{lem}

\begin{proof}
    By Dirichlet's theorem choose a prime \( p\equiv1 \pmod{2d} \).
    The field \( \Q(\zeta_p+\zeta_p^{-1}) \) is totally real and cyclic over \( \Q \) of degree \( (p-1)/2 \), which is a multiple of \( d \), so it contains a subfield \( K \) that is totally real of degree \( d \) over \( \Q \).
    Let \( \theta \) be an algebraic integer generating \( K \), with conjugates \( \theta_1,\dots,\theta_d \), and put \( \beta\coloneqq\theta+m \) for an integer \( m>1+\max_i\lvert\theta_i\rvert \).
    The conjugates \( \beta_i=\theta_i+m \) are then pairwise distinct real numbers greater than \( 1 \).
    Let \( B \) be the matrix of multiplication by \( \beta \) on \( \Z[\beta]\cong\Z^d \), that is, the companion matrix of the minimal polynomial of \( \beta \).
    Its eigenvalues are the \( \beta_i \), so \( B \) satisfies \ref{hyp:host.expanding}, and it satisfies \ref{hyp:host.irreducible} by Lemma~\ref{lem:irreducibility.criterion}.
    Finally let \( \ell \) be a prime not dividing \( \det B=\prod_i\beta_i \) and put \( A\coloneqq\ell\,\Id \).
    Then \( A \) commutes with every matrix, has all eigenvalues equal to \( \ell \), and \( \det A=\ell^d \) is prime to \( \det B \).
    The pair \( (A,B) \) is therefore admissible.
\end{proof}

Pairing a matrix satisfying \ref{hyp:host.irreducible} with a scalar one is the device of \cite[Example 6.1]{berend_multi_1983}; the only point not covered there is \ref{hyp:host.coprime}, which is the choice of \( \ell \).
Nothing forces the two matrices to be so unbalanced, and in low dimension explicit symmetric pairs are easy to write down.

\begin{exm}
\label{exmp:admissible.pair}
    Identify \( \Z^2 \) with \( \Z[\sqrt2] \) by \( (x,y)\mapsto x+y\sqrt2 \).
    Multiplication by \( 3+\sqrt2 \) and by \( 4+2\sqrt2 \) has matrices
    \[
        A=\begin{pmatrix}3&2\\1&3\end{pmatrix},
        \qquad
        B=\begin{pmatrix}4&4\\2&4\end{pmatrix}
    \]
    in the basis \( \{1,\sqrt2\} \).
    They commute, \( \Z[\sqrt2] \) being commutative, and their determinants are the norms \( \det A=N(3+\sqrt2)=7 \) and \( \det B=N(4+2\sqrt2)=8 \), which are coprime.
    The eigenvalues \( 3\pm\sqrt2 \) and \( 4\pm2\sqrt2 \) are real, greater than \( 1 \), and of distinct moduli within each pair, so both matrices satisfy \ref{hyp:host.expanding} and, by Lemma~\ref{lem:irreducibility.criterion}, also \ref{hyp:host.irreducible}.
    Hence \( (A,B) \) is admissible, and so is \( (B,A) \).
\end{exm}

\section*{Acknowledgements}

C. S. R. has been partially supported by S\~{a}o Paulo Research Foundation (FAPESP): grant \#2018/13481-0, and grant \#2020/04426-6. The opinions, hypotheses and conclusions or recommendations expressed in this work are the responsibility of the authors and do not necessarily reflect the views of FAPESP.

The authors would like to acknowledge support from the Max Planck Society, Germany, through the award of a Max Planck Partner Group for Geometry and Probability in Dynamical Systems.\\

%%%%%%%% REFERENCES
\bibliographystyle{amsplain}
\bibliography{bib/biblioWasserstein}

\end{document}